\documentclass{article}

\usepackage{microtype}
\usepackage{graphicx}
\usepackage{subfigure}
\usepackage{booktabs} 
\usepackage{makecell}
\usepackage{multirow}
\usepackage{multicol}
\usepackage{enumitem}

\usepackage{hyperref}

\usepackage[accepted]{icml2023}

\usepackage{amsmath}
\usepackage{amssymb}
\usepackage{mathtools}
\usepackage{amsthm}

\usepackage[capitalize,noabbrev]{cleveref}

\theoremstyle{plain}
\newtheorem{theorem}{Theorem}[section]
\newtheorem{proposition}[theorem]{Proposition}
\newtheorem{lemma}[theorem]{Lemma}

\theoremstyle{definition}

\newtheorem{assumption}[theorem]{Assumption}
\theoremstyle{remark}

\usepackage[textsize=tiny]{todonotes}

\def \S {\mathbf{S}}

\def \R {\mathbb{R}}

\def \w {\mathbf{w}}
\def \v {\mathbf{v}}
\def \V {\mathcal{V}}

\def \x {\mathbf{x}}

\def \E {\mathbb{E}}

\def \x {\mathbf{x}}
\def \p {\mathbf{p}}

\def \e {\mathbf{e}}

\def \1 {\mathbf{1}} 

\def \z {\mathbf{z}}
\def \s {\mathbf{s}}

\def \y {\mathbf{y}}
\def \u {\mathbf{u}}

\def \F {\mathcal{F}}
\def \I {\mathbb{I}}
\def \P {\mathcal{P}}
\def \Q {\mathcal{Q}}

\def \D {\mathcal{D}}

\def \mI {\mathcal{I}}
\def \B {\mathcal{B}}
\def \cB {\mathcal{B}}

\def \wB {\widetilde{\mathcal{B}}}
\def \btau {\bar{\tau}}
\def \balp {\bar{\alpha}}
\def \bgam {\bar{\gamma}}

\def \O {\mathcal{O}}
\def \h {\mathbf{h}}
\def \mS {\mathcal{S}}

\def \name {\text{BSVRB}}
\def \vone {\text{BSVRB}^{\text{v1}}}
\def \vtwo {\text{BSVRB}^{\text{v2}}}

\DeclareMathOperator*{\argmin}{arg\,min}

\icmltitlerunning{Blockwise Stochastic Variance-Reduced Methods with Parallel Speedup for Multi-Block Bilevel Optimization}

\begin{document}

\twocolumn[
\icmltitle{Blockwise Stochastic Variance-Reduced Methods with Parallel Speedup\\ for Multi-Block Bilevel Optimization}




\begin{icmlauthorlist}
\icmlauthor{Quanqi Hu}{tamu}
\icmlauthor{Zi-Hao Qiu}{nju,vis}
\icmlauthor{Zhishuai Guo}{tamu}
\icmlauthor{Lijun Zhang}{nju}
\icmlauthor{Tianbao Yang}{tamu}
\end{icmlauthorlist}

\icmlaffiliation{tamu}{Department of Computer Science and Engineering, Texas A\&M University, College Station, TX, USA}
\icmlaffiliation{nju}{National Key Laboratory for Novel Software Technology, Nanjing University, Nanjing, China}
\icmlaffiliation{vis}{Most work of Z.H. Qiu was done when visiting the OptMAI lab
at TAMU}

\icmlcorrespondingauthor{Tianbao Yang, Quanqi Hu}{tianbao-yang, quanqi-hu@tamu.edu}

\icmlkeywords{Machine Learning, ICML}

\vskip 0.3in
]



\printAffiliationsAndNotice{}  

\begin{abstract}
In this paper, we consider non-convex multi-block bilevel optimization (MBBO) problems, which involve $m\gg 1$ lower level problems and have important applications in machine learning. 
Designing a stochastic gradient and controlling its variance is more intricate due to the hierarchical sampling of blocks and data and the unique challenge of estimating hyper-gradient. We aim to achieve three nice properties for our algorithm: (a) matching the state-of-the-art complexity of standard BO problems with a single block;  (b) achieving parallel speedup by sampling   $I$ blocks and sampling $B$ samples for each sampled block per-iteration; (c) avoiding the computation of the inverse of a high-dimensional Hessian matrix estimator. However, it is non-trivial to achieve all of these by observing that existing works only achieve one or two of these properties. To address the involved challenges for achieving (a, b, c), we propose two stochastic algorithms by using advanced blockwise variance-reduction techniques for tracking the Hessian matrices (for low-dimensional problems) or the Hessian-vector products (for high-dimensional problems), and prove an iteration complexity of $\O(\frac{m\epsilon^{-3}\mathbb I(I<m)}{I\sqrt{I}} + \frac{m\epsilon^{-3}}{I\sqrt{B}})$ for finding an $\epsilon$-stationary point under appropriate conditions. We also conduct experiments to verify the effectiveness of the proposed algorithms comparing with existing MBBO algorithms. 

\end{abstract}

\setlength{\textfloatsep}{2pt}

\setlength\abovedisplayskip{2pt}
\setlength\belowdisplayskip{2pt}
\section{Introduction}
This paper considers solving the following  generalized  bilevel optimization  problem with  multi-block structure: 
\begin{equation}\label{eqn:msbo}
\begin{aligned}
&\min_{\x\in \R^{d_x}} F(\x) = \frac{1}{m}\sum_{i=1}^m \underbrace{f_i(\x, \y_i(\x))}\limits_{F_i(\x)},\\
&\y_i(\x) = \arg\min_{\y_i\in\R^{d_{y,i}}} g_i(\x, \y_i), i=1, \ldots, m.
\end{aligned}
\end{equation}
where $f_i,g_i$ are continuously differentiable functions in expectation forms and $g_i(\x,\y_i)$ is strongly convex with respect to $\y_i$. To be specific, $f_i,g_i$ are defined as $f_i(\x, \y_i(\x)):=  \E_{\xi\sim \mathcal P_i}[f_i(\x, \y_i(\x); \xi)]$ and $g_i(\x, \y_i)= \E_{\zeta\sim \Q_i}[g_i(\x, \y_i;  \zeta)]$. The number of blocks $m$ is considered to be greatly larger than $1$. We refer to the above problem as \textbf{multi-block bilevel optimization} (MBBO). When $m=1$, the MBBO problem reduces to the standard BO problem. The MBBO problem has found many interesting applications in machine learning and AI, e.g., multi-task compositional AUC maximization~\cite{https://doi.org/10.48550/arxiv.2206.00260}, top-$K$ normalized discounted cumulative gain (NDCG) optimization for learning to rank~\cite{https://doi.org/10.48550/arxiv.2202.12183}, and meta-learning~\cite{DBLP:journals/corr/abs-1909-04630}. Recently, \citet{https://doi.org/10.48550/arxiv.2206.00439} uses MBBO to formulate a family of risk functions for optimizing performance at the top. 

The theoretical study of MBBO was initiated by~\cite{https://doi.org/10.48550/arxiv.2105.02266}. In their paper, the authors proposed a randomized stochastic variance-reduced method (RSVRB) for solving MBBO aiming to achieve a state-of-the-art (SOTA) iteration complexity in the order of $O(1/\epsilon^3)$ for finding an $\epsilon$-stationary solution. However, RSVRB and its analysis suffer from several drawbacks: (i) RSVRB requires computing the inverse of the Hessian matrix estimator, which is prohibited for high-dimensional lower-level problems; (ii) the Jacobian estimators maintained for each block could be memory consuming and slow down the algorithm in practice for problems with high-dimensional $\x$; (iii) RSVRB does not achieve a parallel speed-up when using a mini-batch of samples to estimate the gradients, Jacobians and Hessians. While these issues have been tackled for the standard BO problems, e.g., the Hessian matrix can be estimated by the Neumann series and there are works achieving SOTA complexity without maintainng Jacobian estimator~\cite{NEURIPS2021_71cc107d,https://doi.org/10.48550/arxiv.2102.07367}, they become trickier for MBBO problems due to extra noise caused by sampling blocks. Although some later studies for particular MBBO problems  have achieved parallel speed-up and eschewed computing the inverse of a Hessian estimator~\cite{https://doi.org/10.48550/arxiv.2206.00260}, they do not match the SOTA complexity of $O(1/\epsilon^3)$. 

In this paper, we aim to achieve three nice properties for solving MBBO problems: (a) matching the SOTA $O(1/\epsilon^3)$ complexity of standard BO problems with a single block;  (b) achieving parallel speedup by sampling multiple blocks and multiple samples for each sampled block per-iteration; (c) avoiding the computation of the inverse of a  Hessian matrix estimator for high-dimensional lower level problems. To the best of our knowledge, this is the first work that enjoys all of these three properties for solving MBBO problems. We propose two algorithms named $\vone$ and $\vtwo$ for low-dimensional and high-dimensional lower-level problems, respectively.  For $\vone$, we propose to use an advanced blockwise stochastic variance-reduced estimator namely MSVR~\cite{https://doi.org/10.48550/arxiv.2207.08540}  to track and estimate the Hessian matrices and the partial gradients of the lower level problems. To further achieve (c) in $\vtwo$, we explore the idea of converting the inverse of the Hessian matrix multiplied by a partial gradient for each block into solving  another lower level problem using matrix-vector products. 
To maintain the same iteration complexity of $\vone$, we update the estimators of Hessian-vector products of all blocks without compromising the sample complexity per-iteration. At the end, we manage to prove the same iteration complexity of $\O(\frac{m\epsilon^{-3}\mathbb I(I<m)}{I\sqrt{I}} + \frac{m\epsilon^{-3}}{I\sqrt{B}})$ for both algorithms, which reduces to the SOTA complexity $O(\epsilon^{-3}/\sqrt{B})$ of the standard BO with one block. 

Our contributions are summarized as following: 
\begin{itemize}[noitemsep,topsep=0pt,parsep=0pt,partopsep=0pt]
\item We propose two efficient algorithms by using blockwise stochastic variance reduction for solving MBBO problems with low-dimensional and high-dimensional lower-level problems, respectively.
\item We prove the iteration complexity of the two algorithms, which not only matches the SOTA complexity of existing algorithms for solving the standard BO but also achieves parallel speed-up of using multiple blocks and multiple samples of sampled blocks.
\item We conduct experiments on both algorithms for low-dimensional and high-dimensional lower problems and demonstrate the effectiveness of the proposed algorithms against existing algorithms of MBBO. 
\end{itemize}

\begin{table*}[t] 
	\caption{Comparison of iteration complexity and the three properties of different methods for solving MBBO and FCCO problems. We use MMBO-v2 to refer to the second algorithm proposed in~\cite{https://doi.org/10.48550/arxiv.2206.00260} for solving a MBBO problem with a min-max objective.  The iteration complexity only considers the case $I<m$, where $m$ is the total number of blocks, $I$ is the number of sampled blocks per-iteration and $B$ is the number of sampled data for each sampled block per-iteration. (c) is not applicable to FCCO problems.    
	}\label{tab:1} 
	\centering
	\label{tab:2}
	\scalebox{1}{\begin{tabular}{lcccc}
			\toprule
			Method&Objective& Iteration Complexity&	Satisfying (a), (b), (c)\\
        \midrule
          MSVR-v2~\cite{https://doi.org/10.48550/arxiv.2207.08540}& FCCO&$\O(\frac{m\epsilon^{-3}}{I\sqrt{I}} + \frac{m\epsilon^{-3}}{I\sqrt{B}})$& (a), (b)\\
          \hline
        MMBO-v2~\cite{https://doi.org/10.48550/arxiv.2206.00260}&\makecell{MBBO\\ (min-max)}&$\O(\max\left\{\frac{m}{IB},\frac{1}{\min\{I,B\} } \right\}\epsilon^{-4})$&(b), (c)\\
        K-SONG~\cite{https://doi.org/10.48550/arxiv.2202.12183}&MBBO&$\O(\max\left\{\frac{m}{IB},\frac{1}{\min\{I,B\} } \right\}\epsilon^{-4})$& (b)\\
        RSVRB~\cite{https://doi.org/10.48550/arxiv.2105.02266} & MBBO&$\O(m\epsilon^{-3})$& (a)\\
	  $\vone$ (this work)& MBBO&$\O(\frac{m\epsilon^{-3}}{I\sqrt{I}} + \frac{m\epsilon^{-3}}{I\sqrt{B}})$&(a), (b)\\
        $\vtwo$ (this work)& MBBO&$\O(\frac{m\epsilon^{-3}}{I\sqrt{I}} + \frac{m\epsilon^{-3}}{I\sqrt{B}})$&(a), (b), (c)\\
		\bottomrule
	\end{tabular}}
\end{table*}

\section{Related Work}
\textbf{Stochastic Bilevel Optimization (SBO).} SBO algorithms have garnered increasing attention recently. The first non-asymptotic convergence analysis for non-convex SBO with strongly convex lower level problem was given by \cite{99401}. The authors proposed a double-loop stochastic algorithm, where the inner loop solves the lower level problem and the outer loop solves the upper level, and established a sample complexity of $\O(\epsilon^{-6})$ for finding an $\epsilon$-stationary point of $F(\x)$, i.e., a point $\x$ such that $\|\nabla F(\x)\|\leq \epsilon$ in expectation. With a large mini-batch size, \cite{DBLP:journals/corr/abs-2010-07962} improved the sample complexity to $\O(\epsilon^{-4})$. A single-loop two timescale algorithm (TTSA) based on SGD was proposed in \cite{DBLP:journals/corr/abs-2007-05170}, but suffers from a worse sample complexity of $\O(\epsilon^{-5})$. By utilizing variance-reduction method (STORM) to estimate second-order gradients, i.e., Jacobian $\nabla_{xy}^2 g(\x, \y)$ and Hessian $\nabla_{yy}^2 g(\x, \y)$, \cite{DBLP:journals/corr/abs-2102-04671} proposed a single-loop single timescale algorithm (STABLE) that enjoys a sample complexity of $\O(\epsilon^{-4})$ without large mini-batch. Recently, \cite{https://doi.org/10.48550/arxiv.2102.07367,NEURIPS2021_71cc107d,https://doi.org/10.48550/arxiv.2105.02266} further improved the sample complexity to $\O(\epsilon^{-3})$ by fully utilizing variance-reduced estimator for gradients of both upper and lower level objectives. \cite{https://doi.org/10.48550/arxiv.2107.12301} proposed Bregman distance-based algorithms for solving nonsmooth BO with and without  variance reduction. 

One of the difficulties for solving SBO problems lies at how to efficiently compute the Hessian inverse in the gradient estimation. To avoid such potentially expensive matrix inverse operation, many existing works have employed the Neumann series approximation with independent mini-batches following~\cite{99401}.   Another method is to transfer the product of the Hessian inverse and a vector to the solution to a quadratic problem \cite{https://doi.org/10.48550/arxiv.2112.04660,https://doi.org/10.48550/arxiv.2201.13409,DBLP:journals/corr/abs-1909-04630} and to solve it by using deterministic methods (e.g., conjugate gradient) or stochastic methods that only involve matrix-vector products. However, these methods are tailored to single-block BO problems, and their direct applications to MBBO may suffer from per iteration computation inefficiency. Thus, with the potential efficiency issue in consideration, it is trickier to achieve faster rates for MBBO problems~\cite{https://doi.org/10.48550/arxiv.2206.00260}.

\textbf{MBBO.} Besides~\cite{https://doi.org/10.48550/arxiv.2105.02266}, two recent works have considered MBBO and their applications in ML \cite{https://doi.org/10.48550/arxiv.2202.12183,https://doi.org/10.48550/arxiv.2206.00260}. In particular, \citet{https://doi.org/10.48550/arxiv.2202.12183} formulated top-$K$ NDCG optimization for learning-to-rank as a MBBO problem with a compositional objective function, which can be formulated as our MBBO problem. There are many lower-level problems with each having only an one-dimensional variable for optimization. They proposed a stochastic algorithm (K-SONG) that uses blockwise sampling and moving average estimators for tracking gradients and Hessians, and proved an iteration complexity of $\O(\max\{\frac{m}{IB\epsilon^4},\frac{1}{\min\{I,B\}\epsilon^4}\})$. \citet{https://doi.org/10.48550/arxiv.2206.00260} considered a MBBO problem with a min-max objective which includes our considered MBBO problem as a special case. They proposed two algorithms that use moving average estimators  for tracking gradients and Hessians or  Hessian-vector products for lower-dimensional and high-dimensional lower-level problems, respectively, and established a similar iteration complexity of $\O(\max\{\frac{m}{IB\epsilon^4},\frac{1}{\min\{I,B\}\epsilon^4}\})$. In their second algorithm, they avoided computing the inverse of the Hessian matrix estimator by using SGD to solve a quadratic problem. It is notable that the iteration complexities of these two works do not match the SOTA result for the standard BO. As discussed before and later, achieving (a), (b) and (c) simultaneously is not just applying variance-reduction techniques such as SPIDER/SARAH/STORM, etc.~\cite{DBLP:journals/corr/abs-1807-01695,pmlr-v70-nguyen17b,cutkosky2019momentum, NIPS2013_37f0e884}, as done in \cite{https://doi.org/10.48550/arxiv.2105.02266}.

Finally, we would like to point out a related work~\cite{https://doi.org/10.48550/arxiv.2207.08540} that considered the finite-sum coupled  compositional optimization (FCCO) problem, which is a special case of MBBO with the lower problems being quadratic problems with an identity Hessian matrix. They proposed multi-block-Single-probe Variance Reduced (MSVR) estimator for tracking the inner functional mappings in a blockwise stochastic manner. MSVR helps achieve both the SOTA complexity and the parallel speed-up, which is also leveraged in this work. However, since MBBO is more general than FCCO and involves estimating the hyper-gradient, our algorithmic design and analysis face a new challenge for tracking the Hessian-vector-products, which is not present in their work. We make a comparison between different works for solving MBBO and FCCO problems in Table~\ref{tab:1}. 

\section{Preliminaries}
\paragraph{Notations.} Let $\|\cdot\|$ denote the $\ell_2$ norm of a vector or the spectral norm of a matrix. Let $\Pi_{\Omega}[\cdot]$ denote Euclidean projection onto a convex set $\Omega$ for a vector and $\S_{\lambda}[X]$ denotes a projection onto the set $\{X\in \R^{d\times d}: X\succeq \lambda I\}$. The matrix projection operation $\S_{\lambda}[X]$ can be implemented by using singular value decomposition and thresholding the singular values. 
For multi-block structured vectors, we use vector name with subscript $i$ to denote its $i$-th block. For a twice differentiable function $f:X\times Y\to \R$, let $\nabla_x f(x,y)$ and $\nabla_y f(x,y)$ denote its partial gradients taken with respect to $x$ and $y$ respectively, and let $\nabla^2_{xy} f(x,y)$ and $\nabla^2_{yy} f(x,y)$ denote the Jacobian and  the Hessian matrix w.r.t $y$ respectively. We use $f(\cdot;\B)$ to represent an unbiased stochastic estimator of $f(\cdot)$ depending on a sampled mini-batch $\B$.
An unbiased stochastic estimator using one sample $\xi$ is said to have bounded variance $\sigma^2$ if $\E_\xi[\|f(\cdot;\xi)-f(\cdot)\|^2]\leq \sigma^2$. A mapping $f:X\to \R$ is $C$-Lipschitz continuous if $\|f(x)-f(x')\|\leq C\|x-x'\|$ $\forall x,x'\in X$. Function $f$ is $L$-smooth if its gradient $\nabla f(\cdot)$ is $L$-Lipschitz continuous. A function $g:X\to \R$ is $\lambda$-strongly convex if $\forall x,x'\in X$, $g(x)\geq g(x')+\nabla g(x')^T(x-x')+\frac{\lambda}{2}\|x-x'\|^2$.  A point $\x$ is called $\epsilon$-stationary of $F(\cdot)$ if $\|\nabla F(\x)\|\leq \epsilon$.

In order to understand the proposed algorithms, we first present following proposition about the (hyper-)gradient of $F(\x)$, which follows from the standard result in the literature of bilevel optimization~\citep{99401}.  
\begin{proposition}\label{prop:1}
When $g_i(\x, \y_i)$ is strongly convex w.r.t. $\y_i$, we have
\begin{equation*}
\begin{aligned} 
&\nabla F(\x)  =\frac{1}{m}\sum_{i=1}^m \big\{\nabla_x f_i(\x, \y_i(\x)) \\
&\, - \nabla_{xy}^2 g_i(\x, \y_i(\x))[\nabla_{yy}^2g_i(\x, \y_i(\x))]^{-1}\nabla_y f_i(\x, \y_i(\x))\big\}.
\end{aligned}
\end{equation*}
\end{proposition}
There are three sources of computational costs involved in the above gradient: (i) the sum over all $m$ blocks; (ii) the costs for computing the partial gradients, Jacobians and Hessian matrices of individual blocks, which usually depend on many samples; and (iii) the inverse of Hessian matrices. The last two have been tackled in the existing literature of BO.  The first cost can be alleviated by sampling a mini-batch of blocks. However, due to the compositional structure of the hyper-gradient, designing a variance-reduced stochastic gradient estimator is complicated due to the existence of multiple blocks~\cite{https://doi.org/10.48550/arxiv.2207.08540}. In particular, we need to track multiple Hessian matrices  $\nabla_{yy}^2g_i(\x, \y_i(\x))$ or Hessian-vector products $[\nabla_{yy}^2g_i(\x, \y_i(\x))]^{-1}\nabla_y f_i(\x, \y_i(\x))$. To this end, we will leverage the MSVR estimator~\cite{https://doi.org/10.48550/arxiv.2207.08540}, which is described below. 

\textbf{MSVR estimator.} Consider multiple functional  mappings $(h_1(\e), \ldots, h_m(\e))$, at the $t$-th iteration we need to estimate their values by an estimator $\h_t=(\h_{1,t}, \ldots, \h_{m,t})$. Given the constraint that  only a few blocks of mappings $h_i(\e)$ are sampled for assessing their stochastic values, the MSVR update is given by ~\cite{https://doi.org/10.48550/arxiv.2207.08540}:
\begin{equation*}
\h_{i,t+1}=\begin{cases} \bigg[(1-\alpha)\h_{i,t}+\alpha h_i(\e_{t};\B_i^t)\\
\quad+\underbrace{\gamma(h_i(\e_{t};\B_i^t)- h_i(\e_{t-1};\B_i^t))}\limits_{\text{error correction}}\bigg], \,\,  i\in \mI_t\\
\h_{i,t},\,\, \text{o.w.} \end{cases}
\end{equation*}
The update for the sampled $I=|\mI_t|$ blocks have a customized error correction term, which is inspired by previous variance reduced estimator STORM~\cite{cutkosky2019momentum} but has a subtle difference in setting the value of $\gamma$.  Different from the setting of STORM, i.e., $\gamma=1-\alpha$, MSVR sets $\gamma = \frac{m-I}{I(1-\alpha)}+(1-\alpha)$ to account for the randomness and noise induced from block sampling. Due to the need of tracking individual $\y_i$ for each block and the boundedness in our analysis, we extend the above MSVR estimator with two changes: (i) adding a projection onto a convex domain $\Omega$ for the update $\h_{i,t+1}$ of sampled blocks whenever boundedness is required, (ii) the input argument $\e_t$ is changed to individual input $\e_{i,t}$. 



\section{Algorithms}

Due to the compositional structure in terms of $\y_i(\x)$ and $\nabla_{yy}^2g_i(\x_t, \y_i(\x))$ in the hyper-gradient as shown in Proposition~\ref{prop:1},  we need to maintain and estimate variance-reduced estimators for these variables. Below, we present two algorithms for low-dimensional and high-dimensional lower-level problems, respectively. For low-dimensional lower-level problems, we directly estimate the Hessian matrices and compute their inverse if needed. For high-dimensional lower-level problems, we propose to estimate the Hessian-vector products $[\nabla_{yy}^2g_i(\x, \y_{i}(\x)]^{-1}\nabla_y f_i(\x, \y_{i}(\x))$.  

\subsection{For low-dimensional lower-level problems}

\begin{algorithm}[t]
\caption {Blockwise Stochastic Variance-Reduced Bilevel Method (version 1): $\vone$}\label{alg:4}
\begin{algorithmic}[1]
\STATE{Initialization: $\x_0=\x_1, \y_0=\y_1,\s_1,H_1, \z_1$}
\FOR{$t=1,2, ..., T$}
\STATE Sample a subset of lower problems $\mI_t$ 
\STATE Sample two batches $\B_i^t\sim \P_i,\widetilde{\B}_i^t\sim \Q_i$  for $i\in \mI_t$.
\STATE Update  $\s_{i,t+1}$ and $H_{i,t+1}$ according to (\ref{sH_compute}) for $i\in\mI_t$.

\STATE Compute $G_t,\widetilde{G}_t$ according to (\ref{G_v1}).
\STATE Update $\z_{t+1}=(1-\beta_t)(\z_t-\widetilde{G}_t) +G_t$
\STATE Update $\y_{t+1} = \y_{t} - \tau \tau_t \s_{t} $
\STATE Update $\x_{t+1} = \x_{t} - \eta_t \z_{t+1}$
\ENDFOR 
\STATE{\textbf{return} $(\x_{\Tilde{t}}, \y_{\Tilde{t}},\s_{\Tilde{t}},H_{\Tilde{t}},\z_{\Tilde{t}})$ for a randomly selected $\Tilde{t}$}
\end{algorithmic}
\end{algorithm}
We first discuss updates for estimators of the (partial) gradients and the Hessian matrices as they are the major costs per-iteration. Then we discuss the updates of $\x$ and $\y=(\y_1, \ldots, \y_m)$, and finally compare with RSVRB. 

{\bf Updates for Gradient/Hessian Estimators.} We need to estimate $\nabla_y g_i(\x_t, \y_{i,t})$ for updating $\y_{i,t}$, to estimate $\nabla_{yy}^2g_i(\x_t, \y_{i,t})$ for updating $\x_t$.  To this end,  at each iteration $t$, we randomly sample a subset of blocks $\mI_t\subset[m]$. For each sampled block $i\in \mI_t$, we sample a mini-batch $\wB_i^t\sim \Q_i$  for the lower-level problem, and a mini-batch $\B_i^t\sim \P_i$ for the upper-level problem.    
We update the following MSVR estimators of  $\nabla_y g_i(\x_t, \y_{i,t})$ and $\nabla_{yy}^2g_i(\x_t, \y_{i,t})$ for $i\in\mI_t$ and keep their other coordinates unchanged:   
\begin{align}\label{sH_compute}
&\s_{i,t+1}=(1-\alpha_t)\s_{i,t}+\alpha_t \nabla_y g_i(\x_t, \y_{i,t}; \wB_i^t)\notag\\
&\quad +\gamma_t(\nabla_y g_i(\x_t, \y_{i,t}; \wB_i^t)-\nabla_y g_i(\x_{t-1}, \y_{i,t-1}; \wB_i^t))\notag\\
&H_{i,t+1} = \S_{\lambda}\bigg[(1-\balp_t)H_{i,t}+\balp_t\nabla_{yy}^2 g_i(\x_t,\y_{i,t};\wB_i^t) \\
&\quad+\bgam_t\left(\nabla_{yy}^2 g_i(\x_t,\y_{i,t};\wB_i^t) -\nabla_{yy}^2 g_i(\x_{t-1},\y_{i,t-1};\wB_i^t) \right)\bigg],\notag
\end{align}
where $\gamma_t = \frac{m-I}{I(1-\alpha_t)}+(1-\alpha_t)$ and $\bgam_t = \frac{m-I}{I(1-\balp_t)}+(1-\balp_t)$,  $\lambda$ is the lower bound of the Hessian matrix (cf. Assumption~\ref{ass:1}) and $\S_{\lambda}$ is a projection operator to ensure the eigen-value of $H_{i,t+1}$  is lower bounded so that its inverse can be appropriately bounded.

To compute the variance-reduced estimator of $\nabla F(\x_t)$ ,  we compute the stochastic gradient estimations at two iterations: 
\begin{small}
\begin{align}\label{G_v1}\small
&G_t=\frac{1}{ I }\sum_{i\in \mI_t}\big[\nabla_x f_i(\x_t,\y_{i,t};\B_i^t)\notag\\
&\quad\quad\quad  -\nabla_{xy}^2 g_i(\x_t,\y_{i,t};\wB_i^t)[H_{i,t}]^{-1}\nabla_y f_i(\x_t,\y_{i,t};\B_i^t)\big],\notag\\
&\widetilde{G}_t = \frac{1}{ I }\sum_{i\in \mI_t}\big[\nabla_x f_i(\x_{t-1},\y_{i,t-1};\B_i^t)\\
&\!-\!\nabla_{xy}^2 g_i(\x_{t-1},\y_{i,t-1};\wB_i^t)[H_{i,t-1}]^{-1}\nabla_y f_i(\x_{t-1},\y_{i,t-1};\B_i^t)\big]. \notag
\end{align}
\end{small}

Then the STORM gradient estimator $\z_{t+1}$ of $\nabla F(\x_t)$ is updated by
$\z_{t+1}=(1-\beta_t)(\z_t-\widetilde{G}_t) +G_t$.
Note that in the above updates, only stochastic partial gradients, Jacobians, and Hessians based on {\it two mini-batches of data} $\B_i^t$ and $\wB_i^t$ for the sampled blocks $i\in\mI_t$ are computed. This is in sharp contrast with the previous SOTA variance-reduced methods~\cite{https://doi.org/10.48550/arxiv.2102.07367,NEURIPS2021_71cc107d} that require {\it three or four independent} mini-batches due to the use of the Neumann series for estimating the Hessian inverse.  It is also notable that we use the Hessian estimator $H_{i,t}$ from the previous iteration in computing $G_t$ to  decouple its dependence from $\nabla_{xy}^2 g_i(\x_t,\y_{i,t};\wB_i^t)$ due to using the same mini-batch of data $\wB_i^t$; otherwise we need two independent mini-batches~\cite{wang2022finite,https://doi.org/10.48550/arxiv.2206.00260}.


{\bf Updates for $\x_{t+1}$ and $\y_{t+1}$.} While the update for $\x_{t+1}=\x_t - \eta_t \z_{t+1}$ is simple, the update of $\y_{t+1}$ is trickier as there are multiple blocks $\y_{i,t+1},i\in[m]$.  A simple approach is to only update $\y_{i,t+1}$ for $i\in\mI_t$ as only their gradient estimators $\s_{i,t+1}$ are updated. This is adopted by~\cite{https://doi.org/10.48550/arxiv.2206.00260}.  However, since we use MSVR estimators $s_{i,t+1}$ for deriving a fast rate, additional error terms of MSVR estimators will emerge and cause a blow-up on the dependence of $m/I$. In particular, if we only update $\y_{i,t+1}$ for $i\in\mI_t$ and keep other blocks unchanged, we will have an iteration complexity of $T =  \O\left(\max\{\frac{m\mathbb I(I<m)}{I\sqrt{I}},\frac{m^{1.5}}{I^{1.5}\sqrt{B}}\}\epsilon^{-3}\right)$, which has an additional scaling  $\sqrt{m/I}$ compared that in~\cite{  https://doi.org/10.48550/arxiv.2206.00260} albeit with an improved order on $\epsilon^{-1}$. To avoid this unnecessary blow up, a simple remedy will work by updating all blocks of $\y_{i,t+1}$, i.e., $\y_{t+1} = \y_t - \tau\tau_t\s_t$, where $\s_t=(\s_{1,t}, \ldots, \s_{m, t})$, $\tau$ is a parameter and $\tau_t$ is scaled stepsize.  
In fact, such all-block updates can be avoided by using a lazy update strategy. Since the unsampled blocks $\y_{i,t}, \y_{i,t-1}$ are not used in computing gradient/Hessian estimators until $i$ is sampled again, one may accumulate the $\y_{i,t}$ updates and leave it to the future. In particular, at iteration $t$, we replace the full-block updates of $\y_{i,t+1}$ with the following updates for sampled blocks $i\in \mI_t$ at the beginning of the iteration:
\begin{equation}\label{delay_update}
\begin{aligned}
    &\y_{i,t-1} = \y_{i,t-1}-(K_{i,t}-1)\tau \tau_t \s_{i,t},\\
    &\y_{i,t} = \y_{i,t-1}-\tau \tau_t \s_{i,t},
\end{aligned}
\end{equation}
where $K_{i,t}$ denotes the number of iterations passed since the last time $i$ was sampled. For unsampled blocks $i\not\in \mI_t$, no update is needed.
Finally, we present the detailed steps in Algorithms~\ref{alg:4}, to which is referred as $\vone$. 

\textbf{Comparison with RSVRB.} 
$\vone$ is different from RSVRB regarding both  algorithm design and theoretical analysis. We summarize the key differences in algorithm design below, and leave the  differences  of theoretical analysis to section~\ref{sec:analysis_RS}. First of all, RSVRB keeps variance-reduced estimators for all partial gradients, Jacobians and Hessians involved in $\nabla F(\x)$, while $\vone$ only need it for the Hessians and partial gradients of lower-level problems. In fact, estimators for $\nabla_x f_i,\nabla_{xy}^2 g_i, \nabla_y f_i$ do not need variance reduction because they all have unbiased estimator and we keep a variance-reduced estimator $\z_t$ for $\nabla F(\x)$. 
Secondly, in the updates of variance-reduced estimators, RSVRB requires scaling updates for non-sampled blocks, while $\vone$ requires none. Thirdly, at each iteration RSVRB requires twice independent blocking samplings, one for updates of partial gradient estimators  and the other for the update of STORM estimator of $\nabla F(\x)$. Lastly,  RSVRB involves projection operation for the updates of $\y_{i,t}$ while $\vone$ does not. These improvements simplify the algorithm without sacrificing its convergence rate.

\subsection{For high-dimensional lower-level problems}
One limitation of $\vone$ is that computing the inverse of the Hessian estimator $H_{i,t}$  is not suitable for high-dimensional lower-level problems. To address this issue, we propose our second method $\vtwo$. The main idea is to treat $[\nabla^2_{yy}g_i(\x,\y_i)]^{-1}\nabla_y f_i(\x,\y_i)$ as the solution to a quadratic function minimization problem. As a result, $[\nabla^2_{yy}g_i(\x,\y_i)]^{-1}\nabla_y f_i(\x,\y_i)$ can be approximated in a similar way as $\y_i$  in Algorithm~\ref{alg:4}. This strategy  has been studied for solving BO problems in recent works \cite{https://doi.org/10.48550/arxiv.2206.00260,https://doi.org/10.48550/arxiv.2201.13409, https://doi.org/10.48550/arxiv.2112.04660}. However, none of them directly applies to variance reduction methods for MBBO, which incurs additional challenge to be discussed shortly. 

Let us define $m$ quadratic problems and their solutions: 
\begin{align}\label{v_problem}
    \phi_i(\v,\x,\y_i) &:=\frac{1}{2} \v^T \nabla^2_{yy}g_i(\x,\y_i) \v - \v^T \nabla_y f_i(\x,\y_i) \notag\\
    \v_i(\x,\y_i) &:= \argmin\nolimits_{\v\in\R^{d_y}} \phi_i(\v,\x,\y_i).
\end{align}
It is not difficult to show that $\v_i(\x,\y_i)$ is equal to $[\nabla^2_{yy}g_i(\x,\y_i)]^{-1}\nabla_y f_i(\x,\y_i)$. 
Since $\v_i(\x_t,\y_{i,t})$ can be viewed as solution to another layer of lower-level problem, we conduct similar updates for $\v_{i,t}$  to that for $\y_{i,t}$. Define a stochastic estimator $\nabla_v \phi_i(\v,\x,\y_i;\cB_i,\wB_i) := \nabla^2_{yy}g_i(\x,\y_i;\wB_i)\v -  \nabla_y f_i(\x,\y_i;\cB_i)$. Then an MSVR estimator $\u_{i,t+1}$ for gradient $\nabla_v \phi_i(\v_{i,t},\x_t,\y_{i,t})$ is given by 
\begin{align}\label{u_update}
&\u_{i,t+1} = (1-\balp_t)\u_{i,t}+\balp_t\nabla_v \phi_i(\v_{i,t},\x_t,\y_{i,t};\cB_i^t,\wB_i^t) \notag\\
&\quad\quad\quad\quad+\bgam_t\big[\nabla_v \phi_i(\v_{i,t},\x_t,\y_{i,t};\cB_i^t,\wB_i^t) \\
&\quad\quad\quad\quad -\nabla_v \phi_i(\v_{i,t-1},\x_t,\y_{i,t-1};\cB_i^t,\wB_i^t)\big], i\in\mI_t \notag, 
\end{align}
and then an update $\v_{i, t+1} = \left[\v_{i,t}-\btau_t\u_{i,t}\right]$ for the sampled blocks can be conducted. Then we  compute  STORM gradient estimator of $\nabla F(\x_t)$ using the following two stochastic gradient estimations: 
\begin{align}\label{G_v2}
        G_t&=\frac{1}{ I }\sum_{i\in \mI_t}\big[\nabla_x f_i(\x_t,\y_{i,t};\B_i^t)-\nabla_{xy}^2 g_i(\x_t,\y_{i,t};\wB_i^t)\v_{i,t}\big], \notag\\
        \widetilde{G}_t &= \frac{1}{ I }\sum_{i\in \mI_t}\big[\nabla_x f_i(\x_{t-1},\y_{i,t-1};\B_i^t)\\
        &\quad\quad\quad\quad\quad\quad\quad -\nabla_{xy}^2 g_i(\x_{t-1},\y_{i,t-1};\wB_i^t)\v_{i,t-1}\big].\notag
\end{align}

{\bf Updates for $\x_{t+1}$, $\y_{t+1}$ and $\v_{t+1}$.} The updates of $\x_{t+1}$ and $\y_{t+1}$ will be conducted similarly as before. However, the update for $\v_{t+1}$ is more subtle. 
First, the stochastic estimator $\nabla_v \phi_i(\v_{i,t},\x_t,\y_{i,t};\cB_i,\wB_i) = \nabla^2_{yy}g_i(\x_t,\y_{i,t};\wB_i) \v_{i,t} -  \nabla_y f_i(\x_t,\y_{i,t};\cB_i)$ has no bounded variance unless $\v_{i,t}$ is bounded. To this end, we derive an upper bound $V=\frac{C_{fy}}{\lambda}$ so that $\v_i(\x, \y_i)\in\V=\{\v_i: \|\v_i\|_2^2\leq V^2\}$ under the Assumption~\ref{ass:1}, \ref{ass:2} and \ref{ass:3} (cf. Appendix~\ref{Appen:rsvrbv2}). Then the update of $\v_{i,t+1}$ is modified as $\v_{i, t+1} = \Pi_{\V}\left[\v_{i,t}-\btau_t\u_{i,t}\right]$. A similar approach has been used in~\cite{https://doi.org/10.48550/arxiv.2206.00260}. 
Second, similar to the problem of updating $\y_{i,t+1}$ only for $i\in\mI_t$ in $\vone$, updating $\v_{i,t+1}$ only for $i\in\mI_t$ in $\vtwo$ will lead to worse scaling factor in iteration complexity. To avoid this, we update all blocks of $\v_{i,t+1}$ using its MSVR gradient estimators $\u_{i,t+1}$. 
With these changes, we present detailed steps in Algorithm~\ref{alg:2} for $\vtwo$. Finally, we remark that (i) the sample complexity per-iteration of $\vtwo$ is the same as $\vone$, i.e., only two mini-batches $\B_i^t$ and $\wB_i^t$ are required for each sampled block; (ii) similar to (\ref{delay_update}), the updates of non-sampled blocks  $\v_{i,t+1}$ can be delayed until they are sampled again due to that their corresponding gradient estimators do not change and they are not used for computing the gradient estimator $\z_{t+1}$ until their corresponding blocks are sampled again.

\begin{algorithm}[t]
\caption {Block-wise Stochastic Variance-Reduced Bilevel Method (version 2): $\vtwo$}\label{alg:2}
\begin{algorithmic}[1]
\STATE{Initialization: $\x_0=\x_1,\y_0=\y_1,\v_0=\v_1,\s_1,\u_1,\z_1$}
\FOR{$t=1,2, ..., T$}
\STATE Sample a subset of lower problems $\mI_t$ 
\STATE Sample two batches $\B_i^t\sim \P_i,\widetilde{\B}_i^t\sim \Q_i$
 for $i\in \mI_t$.
\STATE Update $\s_{i,t+1}$ and $\u_{i,t+1}$  according to (\ref{sH_compute}), (\ref{u_update}).
\STATE Compute $G_t,\widetilde{G}_t$ according to (\ref{G_v2}).
\STATE $\z_{t+1}=(1-\beta_t)(\z_t-\widetilde{G}_t) +G_t$
 \STATE Update $\y_{t+1} = \y_{t} - \tau \tau_t \s_{t} $
 \STATE Update $\v_{t+1} = \Pi_{\mathcal{V}^m}\left[\v_{t}-\btau_t\u_{t}\right]$
\STATE Update $\x_{t+1} = \x_{t} - \eta_t \z_{t+1}$
\ENDFOR 
\STATE{\textbf{return} $(\x_{\Tilde{t}}, \y_{\Tilde{t}},\v_{\Tilde{t}},\s_{\Tilde{t}},\u_{\Tilde{t}},\z_{\Tilde{t}})$ for a randomly selected $\Tilde{t}$} 
\end{algorithmic}
\end{algorithm}

\section{Convergence Analysis of $\name$}\label{sec:analysis_RS}
In this section we provide the convergence analysis of the proposed algorithms and highlight how it is different from the analysis of RSVRB.
First of all, we make the following assumptions regarding problem~(\ref{eqn:msbo}).
\begin{assumption}\label{ass:1}
For any $\x$, $g_i(\x, \cdot)$ is  $L_g$-smooth and $\lambda$-strongly convex, i.e., $L_gI\succeq \nabla_{yy}^2 g_i(\x,  \y_i)\succeq \lambda I$. 
\end{assumption}
\begin{assumption}\label{ass:2}
Assume the following conditions hold 
\begin{itemize}
[noitemsep,topsep=0pt,parsep=0pt,partopsep=0pt]
\item $\nabla_x f_i(\x, \y_i; \xi)$ is $L_{fx}$-Lipschitz continuous, $\nabla_y f_i(\x, \y_i; \xi)$ is $L_{fy}$-Lipschitz continuous, $\nabla_y g_i(\x, \y_i; \zeta)$ is $L_{gy}$-Lipschitz continuous, $\nabla_{xy}^2 g_i(\x, \y_i; \zeta)$ is $L_{gxy}$-Lipschitz continuous, $\nabla_{yy}^2 g_i(\x, \y_i; \zeta)$ is $L_{gyy}$-Lipschitz continuous, all with respect to $(\x, \y_i)$. 
\item $\|\nabla_x f_i(\x, \y_i)\|^2\leq C_{fx}^2$, $\|\nabla_y f_i(\x, \y_i)\|^2\leq C_{fy}^2$.
\item All stochastic estimators $\nabla_x f_i(\x, \y_i; \xi)$, $ \nabla_y f_i(\x, \y_i; \xi)$, $\nabla_y g_i(\x, \y_i; \zeta)$, $\nabla_{xy}^2 g_i(\x, \y_i; \zeta)$, $\nabla_{yy}^2 g_i(\x, \y_i; \zeta)$ have bounded variance $\sigma^2$.
\end{itemize}
\end{assumption}
\begin{assumption}\label{ass:3}
    $\|\nabla_{yy}^2 g_i(\x, \y_i,\zeta)\|^2\preceq \tilde{C}_{gyy}^2 I$.
\end{assumption}
Assumption~\ref{ass:1} is made in many existing works for SBO~\cite{DBLP:journals/corr/abs-2102-04671,99401,DBLP:journals/corr/abs-2007-05170,DBLP:journals/corr/abs-2010-07962}. Assumption~\ref{ass:2} ii)~iii) are also standard in the literature~\cite{https://doi.org/10.48550/arxiv.2010.07962,99401,DBLP:journals/corr/abs-2007-05170}.  To employ variance reduction technique, Lipchitz continuity of stochastic gradients, i.e., Assumption~\ref{ass:2} i), is required~\cite{https://doi.org/10.48550/arxiv.2106.04692,cutkosky2019momentum}. Note that the assumption $\nabla_y g_i(\x, \y_i; \zeta)$ is $L_{gy}$-Lipschitz continuous implies that $\|\nabla_{xy}^2 g_i(\x, \y_i)\|^2\leq C_{gxy}^2$ and $\|\nabla_{yy}^2 g_i(\x, \y_i)\|^2\leq C_{gyy}^2$ with $C_{gxy} = C_{gyy} = L_{gy}$. Assumption~\ref{ass:3} is only required by $\vtwo$ to ensure the Lipschitz continuity of $\nabla_v\phi_i(\v_i,\x,\y_i,\xi,\zeta)$. It is notable that an even stronger assumption $\tilde{C}_{gyy}^2 I\succeq \nabla_{yy}^2 g(\x, \y; \zeta)\succeq \lambda I$ is made in \citep{99401,DBLP:journals/corr/abs-2007-05170,https://doi.org/10.48550/arxiv.2106.04692} due to the use of the Neumann series (cf. the proof of Lemma 3.2 in~\citep{99401}, Assumption 1\&2 in~\cite{https://doi.org/10.48550/arxiv.2106.04692}).


Comparing to the assumptions made for RSVRB, $\name$ no longer requires the boundedness of $\y_i(\x)$ and the expectation of the  stochastic gradients norms $\nabla_x f_i(\x,\y_i;\xi)$, $\nabla_y f_i(\x,\y_i;\xi)$, $\nabla_y g_i(\x,\y_i;\zeta)$, $\nabla_{xy}^2 g_i(\x,\y_i;\zeta)$, $\nabla_{yy}^2 g_i(\x,\y_i;\zeta)$. The latter boundedness requirements in RSVRB come from the error bound analysis of the randomized coordinate STORM estimators, which uses $(0,\dots,m\nabla_y g_i(\x_t,\y_{i,t};\zeta_t),\dots,0)$ as an unbiased estimator of $\nabla_y g(\x_t,\y_t)$.  It  is also the reason for not having parallel speed-up of using multiple samples for each block~\cite{wang2022finite}. 

Next, we present our main result about the convergence of $\vone$ and $\vtwo$ unified in the following theorem. 
\begin{theorem}\label{thm:informal}
Under Assumptions~\ref{ass:1}, \ref{ass:2} and \ref{ass:3} (for $\vtwo$), with $|\B_i^t|=|\wB_i^t|=B$, $\tau\leq \frac{2}{3L_g}$,  $\alpha_t = \O(B\epsilon^2)$, $\balp_t,\beta_t = \O((\frac{\I(I<m)}{I}+\frac{1}{B})^{-1}\epsilon^2)$, $\tau_t, \btau_t= \O(\sqrt{\frac{ I}{m}}(\frac{\I(I<m)}{I}+\frac{1}{B})^{-1/2}\epsilon)$, $\eta_t =  \O(\frac{ I }{m}(\frac{\I(I<m)}{I}+\frac{1}{B})^{-1/2}\epsilon)$, and by using a large mini-batch size of $\O(1/\epsilon)$ at the initial iteration, both Algorithm~\ref{alg:4} and \ref{alg:2} give $\E\left[\frac{1}{T}\sum_{t=1}^T\|\nabla F(\x_t)\|^2\right]\leq \epsilon^2$ with an iteration complexity $T =  \O(\frac{m\epsilon^{-3}\mathbb I(I<m)}{I\sqrt{I}} + \frac{m\epsilon^{-3}}{I\sqrt{B}})$.
\end{theorem}
{\bf Remark:} The achieved iteration complexity (i) matches the SOTA results for standard BO problems with only one block when $I=m=1$~\cite{https://doi.org/10.48550/arxiv.2102.07367,NEURIPS2021_71cc107d,https://doi.org/10.48550/arxiv.2105.02266}; (ii) has a parallel speed-up  by using multiple blocks $I$ and multiple samples in the mini-batches  $\B_i^t$ and $\wB_i^t$. It is worth mentioning that the above theorem requires using a large batch size at the initial iteration. This is mainly because that we use fixed small parameters for $\alpha_t, \balp_t, \beta_t, \eta_t$ for simplicity of exposition and for proving faster convergence  under a Polyak-\L ojasiewicz (PL) condition in next section, for which we do not require the large batch size at the initial iteration.  We can also use decreasing parameters as in previous works to remove the large batch size at the initial iteration for finding an $\epsilon$-stationary point. 

\section{Faster Convergence for Gradient-Dominant Functions}
In this section, we use a standard restarting trick to improve the convergence of BSVRB under  the gradient dominant condition (aka. Polyak-\L ojasiewicz (PL) condition), i.e., 
\begin{align*}
\mu(F(\x) - \min_{\x'} F(\x'))\leq \|\nabla F(\x)\|^2. 
\end{align*}
The procedure is described in~Algorithm \ref{alg:5} and its convergence is stated below. 

\begin{algorithm}[t]
    \centering
    \caption{RE-$\name$}
    \label{alg:5}
    \begin{algorithmic}[1]
        \STATE Initialize the set of variables $\Theta_0=\{\x_0=\x_1,\y_0=\y_1,\v_0=\v_1,\s_1,H_1,\u_1,\z_1\}$
        \STATE Define parameters $\{\Xi_k\}_{k=1}^K$ according to Theorem~\ref{thm:RE_informal}
       \FOR {$k = 1,\ldots,K$} 
        \STATE $\Theta_k$ = BSVRB($\Theta_{k-1}, \Xi_k)$  
        \ENDFOR
        \STATE \textbf{Return:} $\x_K$
    \end{algorithmic}
\end{algorithm}

\begin{theorem}\label{thm:RE_informal}
Suppose Assumptions~\ref{ass:1}, \ref{ass:2} and \ref{ass:3} (for RE-$\vtwo$) hold and the PL condition holds. Set appropriate initial parameters $\Xi_1=(\beta_1,\alpha_1, \balp_1, \btau_1,\tau_1,\eta_1,T_1)$. Define proper constant $\epsilon_1=\O(\frac{1}{\mu }(\frac{\I(I<m)}{I}+\frac{1}{B}))$ and $\epsilon_k=\epsilon_1/2^{k-1}$. For $k\geq 2$, set parameter $\Xi_k$ such that $\beta_k,\alpha_k,\balp_k=\O(\mu\epsilon_k(\frac{\I(I<m)}{I}+\frac{1}{B})^{-1})$, $\btau_k,\tau_k,\eta_k=\O(\frac{I\sqrt{\mu \epsilon_k}}{m}(\frac{\I(I<m)}{I}+\frac{1}{B})^{-1/2})$ and 
$T_k = \O(\max\left\{\frac{1}{\mu\eta_k},\frac{1}{\tau_k},\frac{1}{\alpha_k}\right\})$ (for RE-$\vone$), or $T_k = \O(\max\left\{\frac{1}{\mu\eta_k},\frac{1}{\beta_k},\frac{1}{\tau_k },\frac{1}{\btau_k }\right\})$ (for RE-$\vtwo$), then after  $K=\O(\log(\epsilon_1/\epsilon))$ stages, the output of RE-$\name$ satisfies  $\E[F(\x_K) - \min_{\x}F(\x)]\leq \epsilon$. 
\end{theorem}
{\bf Remark:} For both RE-$\vone$ and RE-$\vtwo$, it follows from Theorem~\ref{thm:RE_informal} that the total sample complexity is $\sum_{k=1}^{K} T_k = O(\frac{m}{I\mu^{3/2}\epsilon^{1/2}}(\frac{\I(I<m)}{I}+\frac{1}{B})^{1/2}+ \frac{m}{I\mu \epsilon})$. If $\mu\geq \epsilon$, the dependence on $\mu$ and $\epsilon$ matches the optimal rate of $O(\frac{1}{\mu \epsilon})$ to minimize a strongly convex problem, which is a stronger condition than PL condition. In the existing works, RE-RSVRB~\cite{https://doi.org/10.48550/arxiv.2105.02266} has a similar result. Other than that, to get $\E[\|\x-\x^*\|^2]\leq \epsilon$, STABLE~\cite{DBLP:journals/corr/abs-2102-04671} takes a complexity of $O(\frac{m}{\mu^4 \epsilon})$, TTSA  \cite{{DBLP:journals/corr/abs-2007-05170}} takes  $\widetilde{O}(\frac{m}{\mu^3\epsilon^{3/2}})$, and BSA \cite{99401} takes $O(\frac{m}{\mu^5\epsilon^2})$, all  under the strong convexity.

\section{Experiments}
\subsection{Hyper-parameter Optimization}
In this subsection, we consider solving MBBO with high-dimensional lower-level problems. In particular, we consider a hyper-parameter optimization problem for classification with imbalanced data and noisy labels. For handling data imbalance,  we assign the $j$-th training data $\zeta_j$ a weight $\sigma(p_j)\in(0,1)$, where $\sigma(\cdot)$ is a sigmoid function, and $p_j$ is a decision variable which will be learned by a bilevel optimization. In order to tackle noisy labels in the training data, we consider using a robust loss function given by $\mathcal{L}_{\tau}(\w;x,y):=\log\left(1+\exp\left(-y (\w^Tx+w_0)/\tau\right)\right)$, where $x\in\R^d$ is input feature, $y\in\{1,-1\}$ denotes its label,  $\tau>0$ is a temperature parameter. This loss function has been shown to be robust to label noise by tuning the $\tau$ in~\cite{DBLP:journals/corr/abs-2112-14869}. In our experiment, instead of tuning $\tau$, we consider multiple values of them and learn a model that is robust for different temperature values.  In particular, for each temperature value $\tau_i$,  we learn a model $\w_i(\p)$ following the weighted empirical risk minimization using the $i$-th loss $ \mathcal{L}_{i}(\w;x,y) =  \mathcal{L}_{\tau_i}(\w;x,y)$. 

As a result, a MBBO problem is imposed as:      
\begin{equation*}
\begin{aligned}
&\min_{\p\in\R^n}F(\p):=\frac{1}{m}\sum\nolimits_{i=1}^m \E_{\xi_j \sim \D_{val}}\left[ \mathcal{L}_i(\w_i(\p);\xi_j)\right]\\
&\w_i(\p)= \arg \min_{\w\in\R^d}  \E_{\zeta_j \sim \D_{tr}}\left[ \sigma(p_j) \mathcal{L}_i(\w;\zeta_j)\right]+\frac{\lambda}{2}\|\w\|^2,\\
&\quad\quad\quad\quad\quad\quad\quad\quad\quad\quad\quad\quad\quad\quad \text{for all } i=1,\dots,m,
\end{aligned}
\end{equation*}
where $\D_{tr}$ contains $n$ training data points and $\D_{val}$ is a validation set, $\E_{\zeta_j \sim \D_{tr}}$ denotes an average of data from the given set. 

In the first experiment on hyper-parameter optimization, we aim to compare BSVRB and RSVRB, compare $\vtwo$ with $\vone$ for high-dimensional lower-level problems, and to verify the parallel speedup of both $\vone$ and $\vtwo$ with respect to block sampling size $I$ and the batch size $B$ of samples.

\textbf{Data.} We use two binary classification datasets, UCI Adult benchmark dataset \textit{a8a}~\cite{Platt99} and web page classification dataset \textit{w8a}~\cite{Dua2019}. \textit{a8a} and \textit{w8a} have a feature dimensionality of $123$ and $300$ respectively, and contain $22696$ and $49749$ training samples. For both \textit{a8a} and \textit{w8a}, we follow $80\%$/$20\%$ training/validation split.

\textbf{Setup.} We set the number of loss functions to be $100$ using randomly generated $\{\tau_i\}_{i=1}^m$ in the range of $[1,11)$. For methods comparison, we sample $10$ blocks at each iteration and set the sample batch size to be $32$. The regularization parameter $\lambda$ is chosen from $\{0.00001,0.0001, 0.001,0.01\}$. For all methods, we tune the upper-level problem learning rate $\eta_t$ from $\{0.001, 0.01, 0.1\}$ and the lower-level problem learning rates $\tau_t, \btau_t$ from $\{0.01, 0.1, 0.5, 1, 5, 10\}$. Parameters $\alpha_t=\balp_t$ and $\gamma_t=\bgam_t$ in MSVR estimator are tuned from $\{0.5, 0.9,1,10,100\}$ and $\{0.001, 0.01, 0.1, 1, 10, 100\}$ respectively. In RSVRB, the STORM parameter $\beta$ is chosen from $\{0.1, 0.5, 0.9 ,1\}$. We runs $4$ trails for each setting and plot the average curves. This experiment is performed on a computing node with Intel Xeon 8352Y (Ice Lake) processor and $64$GB memory.

\textbf{Results.} We plot the curves of validation loss for $\name$ and RSVRB in Figure~\ref{fig:HO_perf}. For both datasets, all methods perform similarly in terms of epochs. However, in terms of running time, both $\vone$ and $\vtwo$ have better performance than RSVRB. For dataset $w8a$, that has a higher lower-level problem dimension $300$, $\vtwo$ shows its greater advantage against $\vone$ and RSVRB. This is consistent with our theory that $\vtwo$ is more suitable for high-dimensional lower-level problems. Note that one of the major issues that slows down RSVRB is maintaining the Jacobian estimators, e.g. a matrix of size $100*300*39799$ for \textit{w8a}, which is avoided by $\vone$ and $\vtwo$.
In Figure~\ref{fig:abl}, we compare the loss curves of $\vone$ and $\vtwo$ with different values of $I$ (\# of sampled blocks) and $B$ (\# of sampled data per sampled block) on a8a. It shows that the convergence speed increases as $I$ and $B$ increases, which verifies the parallel speedup of our algorithms.

{\bf Classification with Imbalanced Data with Noisy Labels.} To further demonstrate the benefit of our multi-block bilevel optimization formulation for classification with imbalanced data and noisy labels,  we artificially construct an imbalanced a8a data with varied label noise.  We remove $70\%$ of the positive samples in training data to produce an imbalanced version. Moreover, we add noise by flipping the labels of the remaining training data with a certain probability, i.e., the noise level from $0$ to $0.4$. 

We compare three methods, i) logistic regression with the standard logistic loss as the baseline, ii) $\vone$ for solving the bilevel formulation with only one lower level problem ($m=1$ and using the standard logistic loss), and iii) our method for solving multi-block bilevel formulation with $m=100$ blocks corresponding to 100 settings of the scaling factor $\tau_i$ in the logistic loss. Since these methods optimize different objectives, we use the accuracy on a separate testing data as the performance measure for comparison. For our method, we have multiple models learned with different loss functions. We select the best model on the validation data and measure its accuracy on testing data. In terms of parameter tuning, for logistic regression we tune the step size in the range $\{0.001, 0.005,0.01, 0.05, 0.1, 0.5\}$. For $\vone$ we follow the same parameter tuning strategy described in the previous experiment. For each setting, we repeat the experiment 3 times by changing the random seeds. We present the results in Figure~\ref{fig:HO_eff} as a bar graph. We defer the numeric results to Appendix~\ref{app:HO}. As we can see from the results, our multi-block bilevel optimization formulation of hyper-parameter optimization has superior performance, especially with high noise level.

\begin{figure}[h]
\begin{center}
\includegraphics[width=0.4\textwidth]{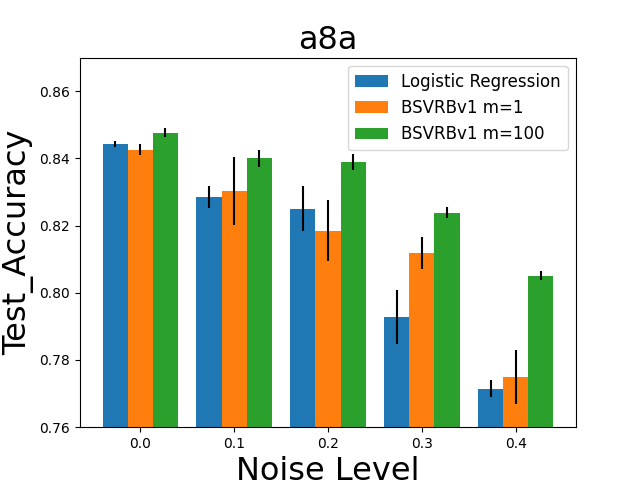}
\end{center}
\caption{Comparison of testing accuracy of models learned by regular logistic regression, $\text{BSVRB}^{\text{v1}}$ with $m=1$ lower-level problem, and $\text{BSVRB}^{\text{v1}}$ with $m=100$ lower-level problems on a corrupted dataset \textit{a8a}  with various noise levels.}\label{fig:HO_eff}
\end{figure}

\begin{figure}[h]
\begin{minipage}[c]{0.236\textwidth}
\centering\includegraphics[width=1\textwidth]{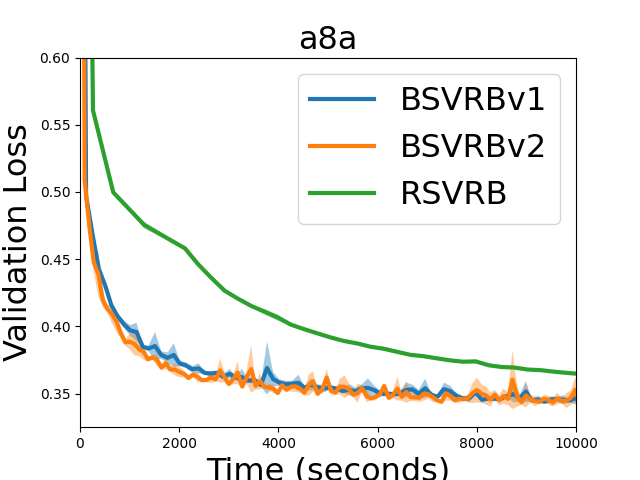}
\end{minipage}
\begin{minipage}[c]{0.236\textwidth}
\centering\includegraphics[width=1\textwidth]{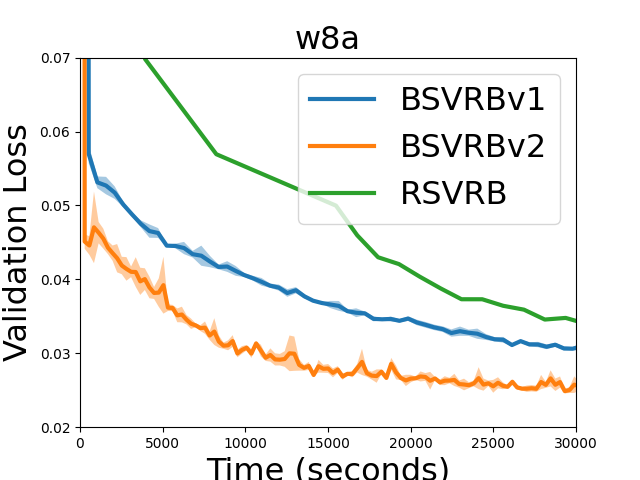}
\end{minipage}
\begin{minipage}[c]{0.236\textwidth}
\centering\includegraphics[width=1\textwidth]{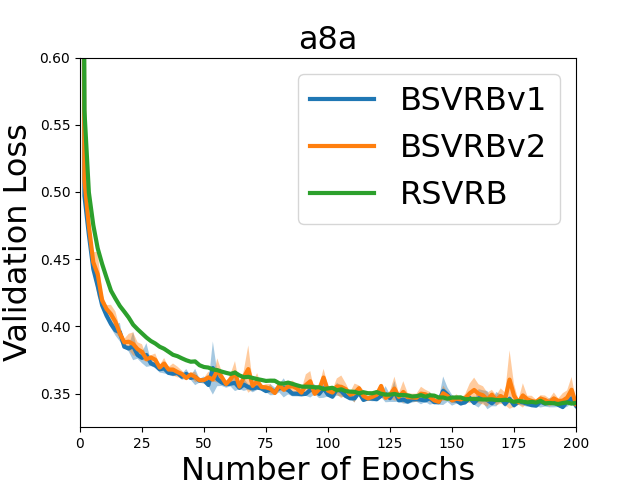}
\end{minipage}
\begin{minipage}[c]{0.236\textwidth}
\centering\includegraphics[width=1\textwidth]{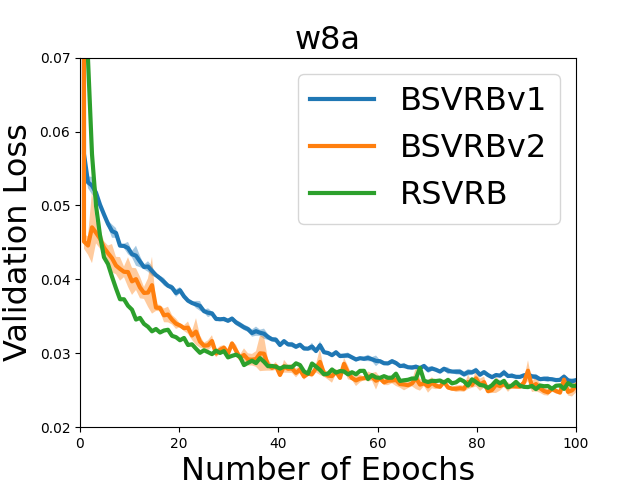}
\end{minipage}
\caption{Comparison of convergence curves of different methods in terms of validation loss on datasets a8a and w8a.}
\label{fig:HO_perf}

\begin{minipage}[c]{0.236\textwidth}
\centering\includegraphics[width=1\textwidth]{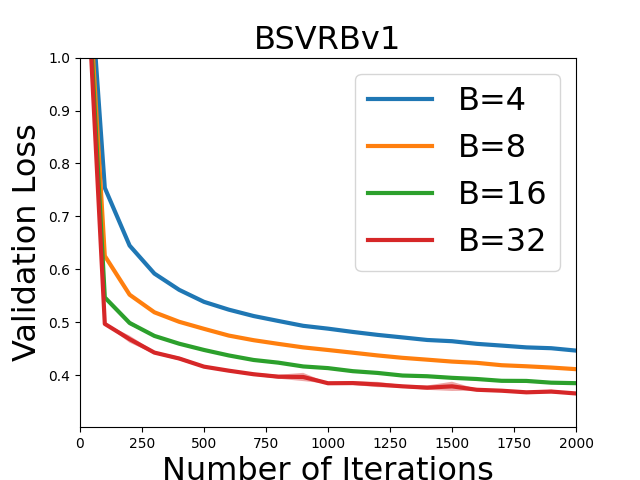}
\end{minipage}
\begin{minipage}[c]{0.236\textwidth}
\centering\includegraphics[width=1\textwidth]{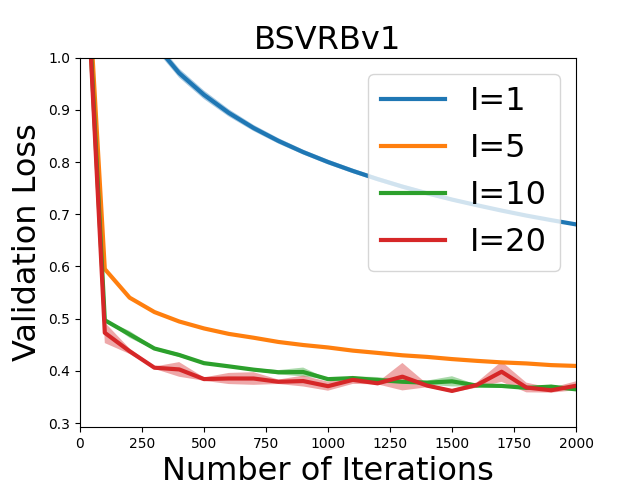}
\end{minipage}
\begin{minipage}[c]{0.236\textwidth}
\centering\includegraphics[width=1\textwidth]{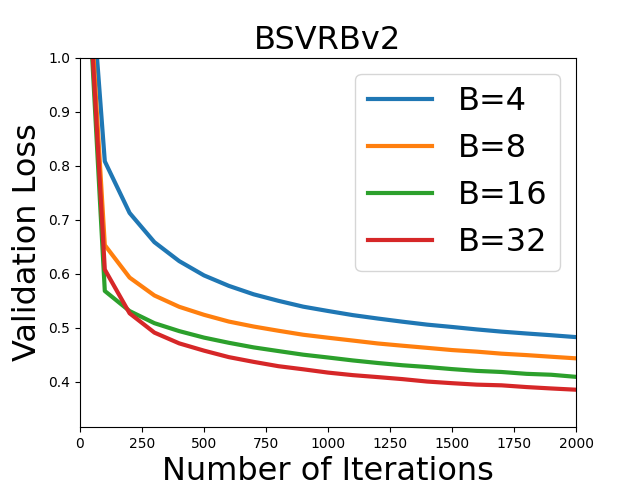}
\end{minipage}
\begin{minipage}[c]{0.236\textwidth}
\centering\includegraphics[width=1\textwidth]{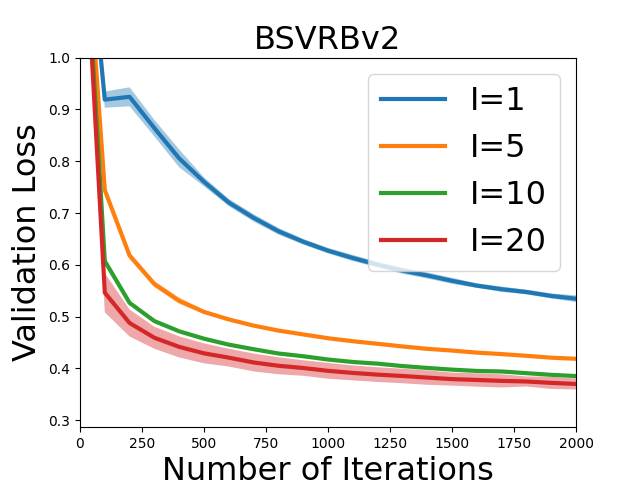}
\end{minipage}
\caption{Comparison of convergence curves of $\name$ algorithms with different values of $I$ and and $B$ on  a8a.}
\label{fig:abl}
\end{figure}

\subsection{Top-K NDCG Optimization}

\begin{figure}[t]
\begin{minipage}[c]{0.236\textwidth}
\centering\includegraphics[width=1\textwidth]{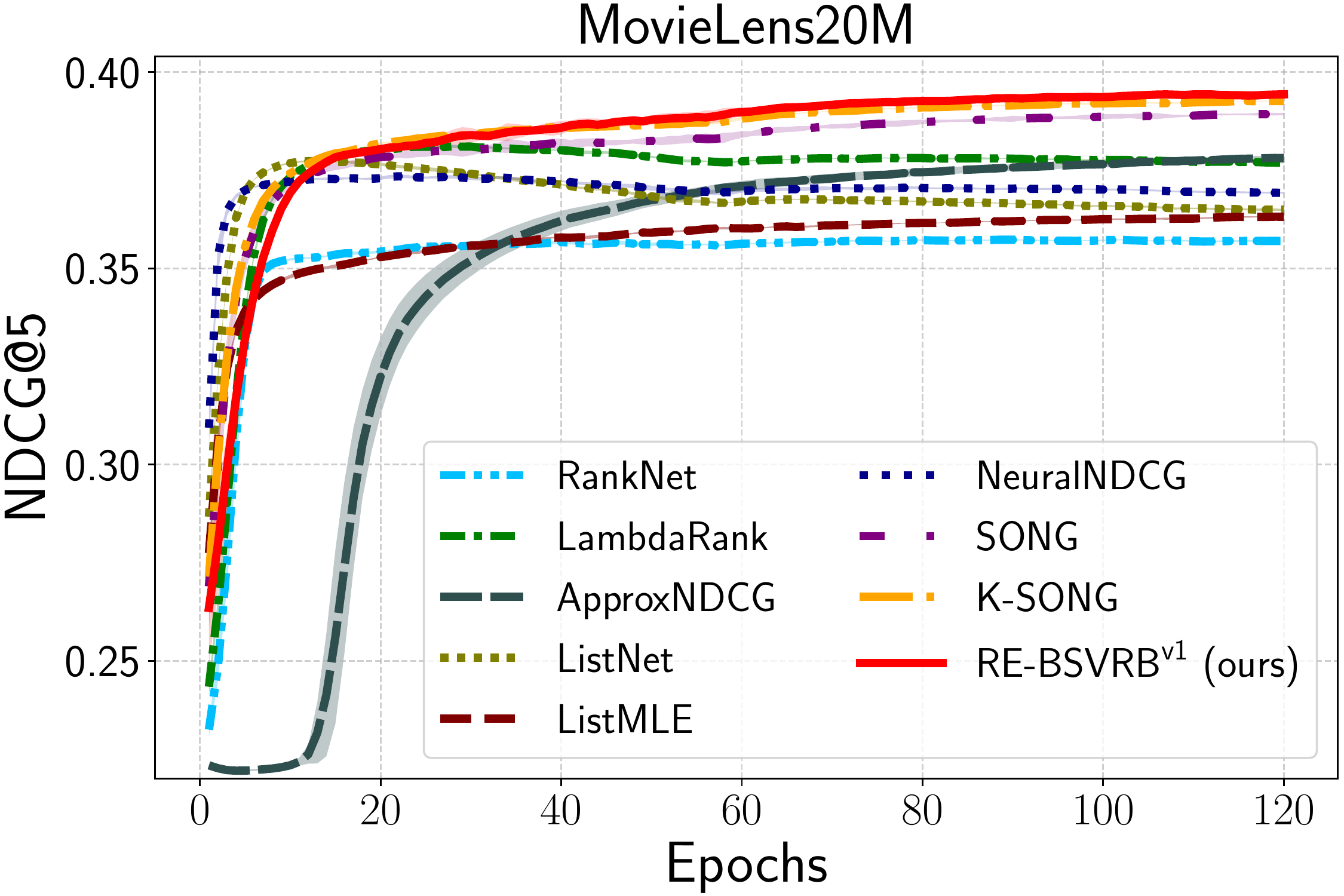}
\end{minipage}
\begin{minipage}[c]{0.236\textwidth}
\centering\includegraphics[width=1\textwidth]{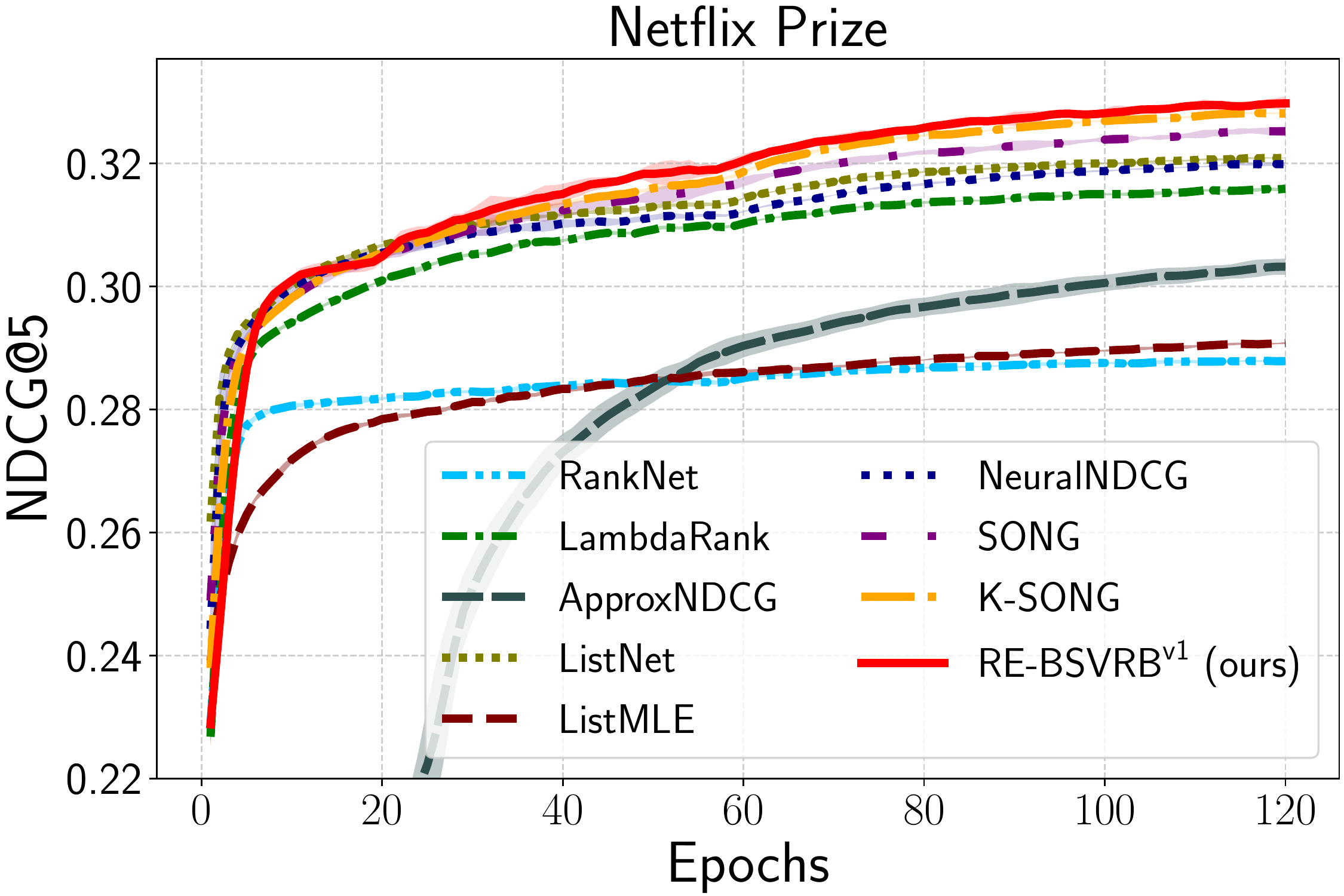}
\end{minipage}
\caption{Comparison of convergence of different methods in terms of validation NDCG@5 on two movie recommendation datasets.}
\label{fig:ndcg}
\end{figure}

In this experiment, we consider the top-$K$ NDCG optimization proposed in \cite{https://doi.org/10.48550/arxiv.2202.12183}, and reformulate it into an equivalent MBBO problem. Let $q\in \mathcal{Q}$ denote a query, $\mS_q=\{(\x_i^q,y_i^q)\}_{i=1}^{N_q}$ denote a set of items and their relevance scores w.r.t to $q$, $\mS$ denote the set of relevant query-item pairs, and $h_q(\cdot,\cdot)$ denote the predictive model for query $q$. Then the MBBO formulation of this problem is:
\begin{equation*}
\begin{aligned} 
&\min \frac{1}{|\mS|}\sum_{(q, \x^q_i)\in\mS}\psi(h_q(\x^q_i; \w)- \lambda_q(\w))f_{q,i}(g(\w; \x^q_i)), \\
& \text{where}\quad \lambda_q(\w) = \arg\min_{\lambda} \frac{K+\varepsilon}{N_q}\lambda+ \frac{\tau_2}{2}\lambda^2 \\&+\frac{1}{N_q}\sum_{\x_i\in\mS_q}\tau_1\ln(1+\exp((h_q(\x_i;\w)-\lambda)/\tau_1)),\\
& g(\w; \x^q_i)=\argmin_g \frac{1}{2}(g-g(\w; \x^q_i, \mS_q))^2, \forall (q,\x^q_i)\in\mS,\\
\end{aligned}
\end{equation*}
where $f_{q,i}(g)=\frac{1}{Z^K_q}\frac{1-2^{y^q_i}}{\log_2(N_q g+ 1)}$, $g(\w; \x^q_i, \mS_q)=\frac{1}{|\mS_q|}\sum_{\x'\in\mS_q}\ell(h_q(\x';\w)-h_q(\x_i^q;\w))$, $\ell(\cdot)=(\cdot+c)^2_+$ with a margin parameter $c$, $\psi(\cdot)$ is sigmoid function, and $Z_q^K$ is the top-K DCG score of the perfect ranking. We refer the readers to \cite{https://doi.org/10.48550/arxiv.2202.12183} for more detailed description of the  problem which is omitted due to limite of space.

We follow the exactly same experimental settings as~\cite{https://doi.org/10.48550/arxiv.2202.12183}. Specifically, we adopt two movie recommendation datasets, i.e., MovieLens20M~\cite{ml-20m} and Netflix Prize dataset~\cite{netflix}, employ the same evaluation protocols, model architectures, and hyper-parameters for training. For our method, we tune $\alpha,\balp$ and $\gamma,\bgam$ in the ranges of $\{0.7,0.8,0.9\}$ and $\{0.001,0.005,0.01, 0.1, 1, 10\}$, respectively. Details of data and experimental setups are presented in Appendix~\ref{app:ndcg_exps}.

Since all lower-level problems have one-dimensional variable for optimization, we only compare $\text{RE-BSVRB}^{\text{v1}}$ with K-SONG and other methods reported in~\cite{https://doi.org/10.48550/arxiv.2202.12183}.  We plot the convergence curves for optimizing top-10 NDCG on two datasets in Figure~\ref{fig:ndcg}, and note that our $\text{RE-BSVRB}^{\text{v1}}$ converges faster than other methods. We also provide NDCG@10 scores on the test data for all methods in Table~\ref{tab:cf-testing-results} and more results in Table~\ref{tab:cf-testing-results_full} in Appendix~\ref{app:ndcg_exps}. We  observe that our method is better for top-$K$ NDCG optimization than other methods. Specifically, our method improves upon K-SONG by 5.24\% and 6.49\% on NDCG@10 for Movielens data and Netflix data, respectively.
\begin{table}[h]
\caption{The test NDCG@10 scores on two movie recommendation datasets averaged over 5 trials. More results for other metrics are in Table~\ref{tab:cf-testing-results_full} in Appendix~\ref{app:ndcg_exps}}
\label{tab:cf-testing-results}
\begin{center}
\begin{small}
\begin{sc}
\renewcommand{\arraystretch}{0.8}
\begin{tabular}{p{2.2cm}p{2.0cm}p{2.0cm}}
\toprule
Method & MovieLens & Netflix \\
\midrule
RankNet & 0.0538$\pm$0.0011 &	0.0362$\pm$0.0002  \\
ListNet & 0.0660$\pm$0.0003  &	0.0532$\pm$0.0002  \\
ListMLE & 0.0588$\pm$0.0001 &	0.0376$\pm$0.0003 \\
LambdaRank & 0.0697$\pm$0.0001  &	0.0531$\pm$0.0002 \\
ApproxNDCG & 0.0735$\pm$0.0005  &	0.0434$\pm$0.0005\\
NeuralNDCG & 0.0692$\pm$0.0003 &	0.0554$\pm$0.0002 \\
SONG & 0.0748$\pm$0.0002  &	0.0571$\pm$0.0002  \\
K-SONG & 0.0747$\pm$0.0002  &	0.0573$\pm$0.0003 \\
\textbf{RE-$\vone$} & \textbf{0.0749}$\pm$0.0003  &	\textbf{0.0585}$\pm$0.0004 
 \\
\bottomrule
\end{tabular}
\end{sc}
\end{small}
\end{center}
\end{table}

The code for reproducing the experimental results in this section is available at \url{https://github.com/Optimization-AI/ICML2023_BSVRB}.

\section{Conclusions}
In this paper, we have proposed novel stochastic algorithms for solving MBBO problems. We have established the state-of-the-art complexity with a parallel speed-up. Our experiments on both algorithms for low-dimensional and high-dimensional lower problems demonstrate the effectiveness of our algorithms against existing algorithms of MBBO.  


\section*{Acknowledgements}
We thank anonymous reviewers for constructive comments. Q. Hu and T. Yang were partially supported by NSF Career Award 2246753, NSF Grant 2246757 and NSF Grant 2246756.

\bibliography{ref,example_paper,all,draft,ndcg}
\bibliographystyle{icml2023}

\newpage
\appendix
\onecolumn

\section{Top-$K$ NDCG Optimization}

\subsection{Details of data and experimental setups}
\label{app:ndcg_exps}

{\bf Data.} We use two large-scale movie recommendation datasets: MovieLens20M~\cite{ml-20m} and Netflix Prize dataset~\cite{netflix}. Both datasets contain large numbers of users and movies, which are represented with integer IDs. All users have rated several movies, with ratings range from 1 to 5. To create training/validation/test sets, we use the most recent rated item of each user for testing, the second recent item for validation, and the remaining items for training, which is widely-used in the literature~\cite{loo-1,loo-2}. When evaluating models, we need to collect irrelevant (unrated) items and rank them with the relevant (rated) item to compute NDCG metrics. During training, inspired by~\citet{wang2019modeling}, we randomly sample 1000 unrated items to save time. When testing, however, we adopt the all ranking protocol~\cite{wang2019neural,he2020lightgcn} --- all unrated items are used for evaluation. 

{\bf Setup.} We choose NeuMF~\cite{NCF} as the backbone network, which is commonly used in recommendation tasks. For all methods, models are first pre-trained by our initial warm-up method for 100 epochs with the learning rate 0.001 and a batch size of 256. Then the last layer is randomly re-initialized and the network is fine-tuned by different methods. At the fine-tuning stage, the initial learning rate and weight decay are set to 0.0004 and 1e-7, respectively. We train the models for 120 epochs with the learning rate multiplied by 0.25 at 60 epochs. The hyper-parameters of all methods are individually tuned for fair comparison, e.g., we tune $\alpha_*$ and $\gamma_*$ for our method in ranges of $\{0.7,0.8,0.9\}$ and $\{0.001,0.005,0.01\}$, respectively.

\begin{table*}[t]
\caption{The full results test NDCG on two movie recommendation datasets. We report the average NDCG@$k$ ($k\in[10,20,50]$) and standard deviation over 5 runs with different random seeds.}
\label{tab:cf-testing-results_full}
\vskip -0.8in
\begin{center}
\begin{small}
\begin{sc}
\renewcommand{\arraystretch}{0.8}
\begin{tabular}{p{2.2cm}p{2.0cm}p{2.0cm}p{2.0cm}p{2.0cm}p{2.0cm}p{2.0cm}}
\toprule
\multirow{2}{*}{\thead{Method}} &
\multicolumn{3}{c}{\thead{MovieLens20M}} &
\multicolumn{3}{c}{\thead{Netflix Prize Dataset}} \\
\cmidrule(lr){2-4}
\cmidrule(lr){5-7}
& NDCG@10 & NDCG@20 & NDCG@50 & NDCG@10 & NDCG@20 & NDCG@50 \\
\midrule
RankNet & 0.0538$\pm$0.0011	 & 0.0744$\pm$0.0013 & 0.1086$\pm$0.0013 &	0.0362$\pm$0.0002 &	0.0489$\pm$0.0003 &	0.0730$\pm$0.0003 \\
ListNet & 0.0660$\pm$0.0003 &	0.0875$\pm$0.0004 &	0.1227$\pm$0.0003 &	0.0532$\pm$0.0002 &	0.0700$\pm$0.0002 &	0.0992$\pm$0.0002 \\
ListMLE & 0.0588$\pm$0.0001 &	0.0799$\pm$0.0001 &	0.1137$\pm$0.0001 &	0.0376$\pm$0.0003 &	0.0508$\pm$0.0004 &	0.0753$\pm$0.0001 \\
LambdaRank & 0.0697$\pm$0.0001 &	0.0913$\pm$0.0002 &	0.1259$\pm$0.0001 &	0.0531$\pm$0.0002 &	0.0693$\pm$0.0002 &	0.0976$\pm$0.0003 \\
ApproxNDCG & 0.0735$\pm$0.0005 &	0.0938$\pm$0.0003 &	0.1284$\pm$0.0002 &	0.0434$\pm$0.0005 &	0.0592$\pm$0.0009 &	0.0873$\pm$0.0012 \\
NeuralNDCG & 0.0692$\pm$0.0003 &	0.0901$\pm$0.0003 &	0.1232$\pm$0.0007 &	0.0554$\pm$0.0002 &	0.0718$\pm$0.0003 &	0.1003$\pm$0.0002 \\
SONG & 0.0748$\pm$0.0002 &	0.0969$\pm$0.0002 &	0.1326$\pm$0.0001 &	0.0571$\pm$0.0002 &	0.0749$\pm$0.0002 &	0.1050$\pm$0.0003 \\
K-SONG & 0.0747$\pm$0.0002 &	\textbf{0.0973}$\pm$0.0003 &	\textbf{0.1340}$\pm$0.0001 &	0.0573$\pm$0.0003 &	0.0743$\pm$0.0003 &	0.1042$\pm$0.0001  \\
\textbf{RE-$\vone$} & \textbf{0.0749}$\pm$0.0003 &	0.0963$\pm$0.0002 &	0.1314$\pm$0.0003 &	\textbf{0.0585}$\pm$0.0004 &	\textbf{0.0760}$\pm$0.0003 &	\textbf{0.1061}$\pm$0.0002
 \\
\bottomrule
\end{tabular}
\end{sc}
\end{small}
\end{center}
\vskip -0.0in
\end{table*}

\section{Convergence Analysis of $\name$}\label{sec:thm2}
In this section, we present the convergence analysis of $\name$. We let $\y_t = (\y_{1,t}, \ldots, \y_{m,t})$, $\v_t = (\v_{1,t}, \ldots, \v_{m,t})$, $\u_t = (\u_{1,t}, \ldots, \u_{m,t})$, $\s_t = (\s_{1,t}, \ldots, \s_{m,t})$, $H_t = (H_{1,t}, \ldots, H_{m,t})$, $\y(\x)=(\y_1(\x),\dots,\y_m(\x))$, $\v(\x,\y)=(\v_1(\x,\y_1),\dots,\v_m(\x,\y_m))$.

For simplicity, we define the following notations.
\begin{equation*}
    \begin{aligned}
    	&\delta_{z,t}:=\|\z_{t+1}-\Delta_t\|^2,\quad \delta_{y,t}:=\sum_{i=1}^m\|\y_{i,t}-\y_i(\x_t)\|^2, \quad \delta_{v,t}:=\sum_{i=1}^m\|\v_{i,t}-\v(\x_t,\y_{i,t})\|^2,\\
        &\delta_{s,t}:=\sum_{i=1}^m\|\s_{i,t}-\nabla_y g_i(\x_{t-1},\y_{i,t-1})\|^2,\quad \tilde{\delta}_{s,t}:=\sum_{i=1}^m\|\s_{i,t}-\nabla_y g_i(\x_{t},\y_{i,t})\|^2,\\
        &\delta_{u,t}:=\sum_{i=1}^m\|\u_{i,t}-\nabla_v \phi_i(\v_{i,t-1},\x_{t-1},\y_{i,t-1})\|^2,\quad \tilde{\delta}_{u,t}:=\sum_{i=1}^m\|\u_{i,t}-\nabla_v \phi_i(\v_{i,t},\x_{t},\y_{i,t})\|^2,\\
        &\delta_{H,t}:=\sum_{i=1}^m\|H_{i,t}-\nabla_{yy}^2 g_i(\x_{t-1},\y_{i,t-1})\|^2,\quad \tilde{\delta}_{H,t}:=\sum_{i=1}^m\|H_{i,t}-\nabla_{yy}^2 g_i(\x_{t},\y_{i,t})\|^2.
    \end{aligned}
\end{equation*}

Note that under Assumption~\ref{ass:1}, \ref{ass:2}, \ref{ass:3}, we have
\begin{equation}\label{ineq:D2d}
    \begin{aligned}
        &\tilde{\delta}_{H,t}\leq 2\delta_{H,t}+2L_{gyy}^2 (m\|\x_t-\x_{t-1}\|^2+\|\y_t-\y_{t-1}\|^2)\\
        &\tilde{\delta}_{s,t}\leq 2\delta_{s,t}+2L_{gy}^2 (m\|\x_t-\x_{t-1}\|^2+\|\y_t-\y_{t-1}\|^2)\\
        &\tilde{\delta}_{u,t}\leq 2\delta_{u,t}+2L_{\phi v}^2 (\|\v_t-\v_{t-1}\|^2+m\|\x_t-\x_{t-1}\|^2+\|\y_t-\y_{t-1}\|^2)\\
    \end{aligned}
\end{equation}

We initialize $\x_0=\x_1$, $\y_0=\y_1$ and $\v_0=\v_1$, so that we have
\begin{equation}\label{ineq:xy_sum}
    \begin{aligned}
        &\sum_{t=1}^T\|\x_t-\x_{t-1}\|^2\leq \sum_{t=1}^T\|\x_{t+1}-\x_t\|^2,\quad \sum_{t=1}^T\|\y_t-\y_{t-1}\|^2\leq \sum_{t=1}^T\|\y_{t+1}-\y_t\|^2,\\
        &\quad \sum_{t=1}^T\|\v_t-\v_{t-1}\|^2\leq \sum_{t=1}^T\|\v_{t+1}-\v_t\|^2.
    \end{aligned}
\end{equation}



We first present some standard results from non-convex optimization and bilevel optimization literature.
\begin{lemma}[Lemma 2.2 in \citep{99401}]\label{lem:201}
$F(\x)$ is $L_F$-smooth and $\y_i(\x)$ is $L_y$-Lipschitz continuous for all $i=1,\dots,m$, where $L_y$ and $L_F$ are appropriate constants.
\end{lemma}

\begin{lemma}\label{lem:01}
Let $\x_{t+1} = \x_t - \eta_t\z_{t+1}$. Under Assumptions~\ref{ass:1}, \ref{ass:2}, with $\eta_tL_F\leq 1/2$, we have
\begin{equation*}
\begin{aligned}
F(\x_{t+1})&\leq F(\x_t) +   \frac{ \eta_t}{2}\|\nabla F(\x_t) - \z_{t+1}\|^2 - \frac{\eta_t}{2}\|\nabla F(\x_t)\|^2  - \frac{\eta_t}{4}\|\z_{t+1}\|^2.
\end{aligned}
\end{equation*}
\end{lemma}

\begin{lemma}[Lemma 6 in \cite{https://doi.org/10.48550/arxiv.2105.02266}]\label{lem:3}
Let $\y_{t+1} =\y_t - \tau_t\tau \s_t$ with $\tau\leq 2/(3L_g)$, we have
\begin{align*}
&\|\y_{t+1} - \y(\x_{t+1})\|^2 \leq  (1 - \frac{\tau_t\tau \lambda}{4})\|\y_t - \y(\x_t)\|^2 + \frac{8\tau_t\tau}{\lambda}\|\nabla_y g(\x_t, \y_t) -\s_t\|^2 \\
& + \frac{8L_y^2\gamma^2}{\tau_t\tau\lambda}\|\x_{t+1}-\x_t\|^2  -  \frac{2\tau}{\tau_t}(1+ \frac{\tau_t\tau \lambda}{4})(\frac{1}{2\tau} - \frac{3L_g}{4} )\|\y_t- \y_{t+1}\|^2. 
\end{align*}
\end{lemma}

\begin{lemma}\label{lem:MSVR}
Let $\Omega$ be a convex set. Suppose mapping $h_i(\e_i;\xi)$ is $L$-Lipschitz, $h_i(\e)=\E_{\xi}[h_i(\e_i;\xi)]$, $h_i(\e)\in \Omega$ and $\E_{\xi}[\|h_i(\e_i)-h_i(\e_i;\xi)\|^2]\leq \sigma^2$ for all $i=1,\dots,m$. Consider the MSVR update:
\begin{equation}\label{MSVR_update}
\h_{i,t+1}=\begin{cases} \Pi_{\Omega}\bigg[(1-\alpha)\h_{i,t}+\alpha h_i(\e_{i,t};\B_i^t)\\
\quad+\gamma(h_i(\e_{i,t};\B_i^t)- h_i(\e_{i,t-1};\B_i^t))\bigg], \,\,  i\in \mI_t\\
\h_{i,t},\,\, \text{o.w.} \end{cases}
\end{equation}
Denote $\delta_{h,t}:=\sum_{i=1}^m \|\h_{i,t}-h_i(\e_{i,t-1})\|^2$.  By setting $\gamma = \frac{m-I}{I(1-\alpha)}+(1-\alpha)$, for $\alpha\leq \frac{1}{2}$, with batch sizes $I=|\mI_t|$ and $B=|\B_i^t|$, we have
\begin{equation}\label{eqn:msvr}
\begin{aligned}
   \E\left[\delta_{h,t+1}\right]&\leq (1-\frac{I \alpha}{m})\E\left[ \delta_{h,t}\right]+\frac{2I\alpha^2 \sigma^2}{B} \\
   &\quad +\frac{8m^2 L^2}{I}\E\left[\sum_{i=1}^m\|\e_{i,t-1}-\e_{i,t}\|^2\right]
\end{aligned}
\end{equation}
\end{lemma}
With $\Omega=\R^d$ the above lemma is Lemma 1 in \cite{https://doi.org/10.48550/arxiv.2207.08540}. We refer the detailed proof to Appendix~\ref{app:MSVR}


\subsection{Convergence Analysis of $\vone$}

We first present a formal statement of Theorem~\ref{thm:informal} for $\vone$.
\begin{theorem}\label{thm:6}
Under Assumptions~\ref{ass:1} and \ref{ass:2}, with $\tau\leq \frac{2}{3L_g}$, $\bgam_t = \frac{m-I}{I(1-\balp_t)}+(1-\balp_t)$, $\gamma_t = \frac{m-I}{I(1-\alpha_t)}+(1-\alpha_t)$, $\alpha_t \leq \min\left\{\frac{1}{2},\frac{B\epsilon^2}{12C_{10}}\right\}$, $\beta_t \leq \frac{\min\{ I ,B\}\epsilon^2}{12C_{10}}$, $\balp_t\leq \min\left\{\frac{1}{2}, \frac{\epsilon^2}{12C_{10}}(\frac{\I(I<m)}{I}+\frac{1}{B})^{-1}\right\}$,  $\tau_t\leq  \sqrt{\frac{C_8}{12C_{10}}} \frac{ \sqrt{I}\epsilon}{\sqrt{m}}(\frac{\I(I<m)}{I}+\frac{1}{B})^{-1/2}$, $\eta_t \leq \min\left\{\frac{1}{2L_F}, \sqrt{C_{11}} \frac{ I \epsilon}{m }(\frac{\I(I<m)}{I}+\frac{1}{B})^{-1/2}\right\}$, where $C_{10}, C_{11}$ are constants specified in the proof, and by using a large mini-batch size of $\O(1/\epsilon)$ at the initial iteration for computing $\z_1, \s_1,H_1$ and computing an accurate solution $\y_1$ such that $\delta_{y,1}\leq \O(1)$, Algorithm~\ref{alg:4} gives $\E\left[\frac{1}{T}\sum_{i=1}^m\|\nabla F(\x_t)\|^2\right]\leq \epsilon^2$ with sample complexity $T =  \O\left(\frac{m\epsilon^{-3}\mathbb I(I<m)}{I\sqrt{I}} + \frac{m\epsilon^{-3}}{I\sqrt{B}}\right)$.
\end{theorem}

Define
\begin{equation}
\Delta_t = \frac{1}{m}\sum_{i=1}^m \nabla_x f_i(\x_t,\y_{i,t})-\nabla_{xy}^2 g_i(\x_t,\y_{i,t})\E_t[[H_{i,t}]^{-1}]\nabla_y f_i(\x_t,\y_{i,t})
\end{equation}
so that $\E_t[G_t]=\Delta_t$, $\E_t[\widetilde{G}_t]=\Delta_{t-1}$.

\begin{lemma}\label{lem:301}
With constants $C_1,C_2$ defined in the proof, we have
\begin{equation}
\|\Delta_t-\nabla F(\x_t)\|^2 \leq  C_1 \delta_{y,t}+C_2\tilde{\delta}_{H,t}.
\end{equation}
\end{lemma}

\begin{lemma}\label{lem:302}
Consider the updates in Algorithm~\ref{alg:4}, we have
\begin{equation}
    \begin{aligned}
        \E_t[\left\|\z_{t+1}-\Delta_t\right\|^2]&\leq (1-\beta_t) \|\z_t-\Delta_{t-1}\|^2+2C_3\|\x_t-\x_{t-1}\|^2+\frac{2C_3}{m}\|\y_t-\y_{t-1}\|^2\\
        &\quad +\frac{2C_4}{m}\|H_t-H_{t-1}\|^2 +2\beta_t^2C_5\left(\frac{\I(I<m)}{I}+\frac{1}{B}\right)
    \end{aligned}
\end{equation}
\end{lemma}

\begin{lemma}\label{lem:303}
With MSVR updates for $H_{t+1}$, if $\balp_t\leq \frac{1}{2}$, we have
\begin{equation}
\begin{aligned}
\E\left[\|H_{t+1}-H_t\|^2\right]&\leq \frac{2I\balp_t^2\sigma^2}{B}+\frac{8I \balp_t^2}{m}\E[\delta_{H,t}]+\frac{9m^2 L_{gyy}^2}{I}\E[\|\x_{t-1}-\x_t\|^2]+\frac{9m L_{gyy}^2}{I}E[\|\y_{t-1}-\y_t\|^2]
\end{aligned}
\end{equation}
\end{lemma}


By applying Lemma~\ref{lem:MSVR} to $\s_t,H_t$, we have
\begin{equation*}
\begin{aligned}
    \E\left[\delta_{H,t+1}\right]&\leq (1-\frac{I \balp_t}{m})\E\left[\delta_{H,t}\right]+\frac{2I\balp_t^2 \sigma^2}{B} +\frac{8m^2 L_{gyy}^2}{I}\E\left[\|\x_{t-1}-\x_t\|^2+\frac{1}{m}\|\y_{t-1}-\y_t\|^2\right]
\end{aligned}
\end{equation*}
and
\begin{equation*}
\begin{aligned}
    \E\left[\delta_{s,t+1}\right]&\leq (1-\frac{I \alpha_t}{m})\E\left[\delta_{s,t}\right]+\frac{2I\alpha_t^2 \sigma^2}{B} +\frac{8m^2 L_{gy}^2}{I}\E\left[\|\x_{t-1}-\x_t\|^2+\frac{1}{m}\|\y_{t-1}-\y_t\|^2\right]
\end{aligned}
\end{equation*}
Take summation over $t=1,\dots,T$, then we obtain
\begin{equation}\label{ineq:305}
\begin{aligned}
    \E\left[\sum_{t=1}^T\delta_{H,t}\right]&\leq \E\Bigg[\frac{m}{I \balp_t}\delta_{H,1}+\frac{2m\balp_t \sigma^2T}{B} +\frac{8m^3 L_{gyy}^2}{\balp_t I^2}\sum_{t=1}^T\left[\|\x_{t-1}-\x_t\|^2+\frac{1}{m}\|\y_{t-1}-\y_t\|^2\right]\Bigg]
\end{aligned}
\end{equation}
and
\begin{equation}\label{ineq:304}
\begin{aligned}
    \E\left[\sum_{t=1}^T\delta_{s,t}\right]&\leq \E\Bigg[\frac{m}{I \alpha_t}\delta_{s,1}+\frac{2m\alpha_t \sigma^2T}{B} +\frac{8m^3 L_{gy}^2}{\alpha_t I^2}\sum_{t=1}^T\left[\|\x_{t-1}-\x_t\|^2+\frac{1}{m}\|\y_{t-1}-\y_t\|^2\right]\Bigg]
\end{aligned}
\end{equation}


\subsubsection{Proof of Theorem~\ref{thm:6}}
\begin{proof}
By Lemma~\ref{lem:01}, we have
\begin{equation}\label{ineq:300}
    \begin{aligned}
        F(\x_{t+1})-F(\x_t)\leq \frac{\eta_t}{2}\|\z_{t+1}-\nabla F(\x_t)\|^2-\frac{\eta_t}{2}\|\nabla F(\x_t)\|^2-\frac{\eta_t}{4}\|\z_{t+1}\|^2.
    \end{aligned}
\end{equation}
The first term on the right hand side can be divided into two terms.
\begin{equation}\label{ineq:301}
    \begin{aligned}
    \left\|\z_{t+1}-\nabla F(\x_t)\right\|^2
    \leq 2\left\|\z_{t+1}-\Delta_t\right\|^2 +2\left\|\Delta_t -\nabla F(\x_t)\right\|^2
    \end{aligned}
\end{equation}
where we have recursion for the first term on the right hand side in Lemma~\ref{lem:302} and the second term is bounded by Lemma~\ref{lem:301}. Combining inequalities~\ref{ineq:300},\ref{ineq:301} and Lemma~\ref{lem:301} gives
\begin{equation}
    \begin{aligned}
        F(\x_{t+1})-F(\x_t)&\leq \eta_t\delta_{z,t}+\frac{\eta_tC_1}{m}\delta_{y,t}+\frac{\eta_tC_2}{m}\tilde{\delta}_{H,t} -\frac{\eta_t}{2}\|\nabla F(\x_t)\|^2-\frac{\eta_t}{4}\|\z_{t+1}\|^2.
    \end{aligned}
\end{equation}

Taking summation over $t=1,\dots,T$ yields

\begin{equation}
    \begin{aligned}
        \sum_{t=1}^T\|\nabla F(\x_t)\|^2
        &\leq \frac{2}{\eta_t}(F(\x_1)-F(\x^*))+ 2\sum_{t=1}^T\delta_{z,t}+\frac{2C_1}{m}\sum_{t=1}^T\delta_{y,t}+\frac{2C_2}{m}\sum_{t=1}^T\tilde{\delta}_{H,t} -\frac{1}{2}\sum_{t=1}^T\|\z_{t+1}\|^2
    \end{aligned}
\end{equation}

We enlarge the values of constants $C_1$ so that
\begin{equation}\label{ineq:20_1}
    \begin{aligned}
        &\sum_{t=1}^T\|\nabla F(\x_t)\|^2+\sum_{t=1}^T\delta_{z,t}+\frac{1}{m}\sum_{t=1}^T\delta_{y,t}\\
        &\leq \frac{2}{\eta_t}(F(\x_1)-F(\x^*))+ 3\sum_{t=1}^T\delta_{z,t}+\frac{C_1}{m}\sum_{t=1}^T\delta_{y,t}+\frac{C_2}{m}\sum_{t=1}^T\tilde{\delta}_{H,t}-\frac{1}{2}\sum_{t=1}^T\|\z_{t+1}\|^2
    \end{aligned}
\end{equation}

It follows from Lemma~\ref{lem:302} that
\begin{equation}
	\begin{aligned}
        \E\left[\sum_{t=1}^T\delta_{z,t}\right]
        &\leq \E\Bigg[\frac{\delta_{z,1}}{\beta_t}+\frac{2C_3}{\beta_t}\sum_{t=1}^T\|\x_t-\x_{t+1}\|^2+\frac{2C_3}{m\beta_t}\sum_{t=1}^T\|\y_t-\y_{t+1}\|^2\\
        &\quad+\frac{2C_4}{m\beta_t}\sum_{t=1}^T\|H_t-H_{t+1}\|^2 +2\beta_tC_5T\left(\frac{\I(I<m)}{I}+\frac{1}{B}\right)\Bigg]
    \end{aligned}
\end{equation}
Combing with Lemma~\ref{lem:303}, we have

\begin{equation}\label{ineq:306}
\begin{aligned}
\E\left[\sum_{t=1}^T\delta_{z,t}\right]
 &\leq \E\Bigg[\frac{\delta_{z,1}}{\beta_t}+\left(\frac{2C_3}{\beta_t}+\frac{18C_4m L_{gyy}^2}{I\beta_t}\right)\sum_{t=1}^T\|\x_t-\x_{t+1}\|^2+\left(\frac{2C_3}{m\beta_t}+\frac{18 L_{gyy}^2}{I\beta_t}\right)\sum_{t=1}^T\|\y_t-\y_{t+1}\|^2\\
 &\quad +2\beta_tC_5T\left(\frac{\I(I<m)}{I}+\frac{1}{B}\right)+\frac{4C_4I\balp_{t+1}^2\sigma^2T}{m\beta_t B}+\frac{8C_4I \balp_{t+1}^2}{m^2\beta_t}\sum_{t=1}^T\delta_{H,t}\Bigg]
\end{aligned}
\end{equation}

Following from Lemma~\ref{lem:3}, we have


\begin{equation}\label{ineq:303}
\begin{aligned}
    \E\left[\sum_{t=1}^T\delta_{y,t}\right]&\leq \E\Bigg[\frac{4}{\tau_t\tau \lambda }\delta_{y,1} +  \frac{32}{\lambda^2 }\sum_{t=1}^T\tilde{\delta}_{s,t}  - \frac{8}{\tau_t^2 \lambda I}(\frac{1}{2\tau} - \frac{3L_g}{4} )\sum_{t=1}^T\|\y_{t+1}- \y_t\|^2\\
    &\quad +\frac{32mL_y^2}{\tau_t^2\tau^2 \lambda^2 }\sum_{t=1}^T\|\x_t-\x_{t+1}\|^2\Bigg]
\end{aligned}
\end{equation}

Following from inequalities~(\ref{ineq:303}), (\ref{ineq:306}) and (\ref{ineq:20_1}), we have
\begin{equation*}
    \begin{aligned}
        &\E\left[\sum_{t=1}^T\|\nabla F(\x_t)\|^2+\sum_{t=1}^T\delta_{z,t}+\frac{1}{m}\sum_{t=1}^T\delta_{y,t}\right]\\
&\leq \E\Bigg[\frac{2}{\eta_t}(F(\x_1)-F(\x^*))+\frac{3\delta_{z,1}}{\beta_t} +6\beta_tC_5T\left(\frac{\I(I<m)}{I}+\frac{1}{B}\right)+\frac{12C_4I\balp_{t+1}^2\sigma^2T}{m\beta_t B}\\
&\quad +\frac{C_1}{m}\sum_{t=1}^T\delta_{y,t}+\frac{C_2}{m}\sum_{t=1}^T\tilde{\delta}_{H,t}+\frac{24C_4I \balp_{t+1}^2}{m^2\beta_t}\sum_{t=1}^T\delta_{H,t}\\
&\quad +\left(\frac{6C_3}{\beta_t}+\frac{54C_4m L_{gyy}^2}{I\beta_t}\right)\sum_{t=1}^T\|\x_t-\x_{t+1}\|^2-\frac{1}{2}\sum_{t=1}^T\|\z_{t+1}\|^2+\left(\frac{6C_3}{m\beta_t}+\frac{54 L_{gyy}^2}{I\beta_t}\right)\sum_{t=1}^T\|\y_t-\y_{t+1}\|^2\\
&\leq \E\Bigg[\frac{2}{\eta_t}(F(\x_1)-F(\x^*))+\frac{3\delta_{z,1}}{\beta_t} +\frac{4C_1}{\tau_t\tau \lambda m}\delta_{y,1} +6\beta_tC_5T\left(\frac{\I(I<m)}{I}+\frac{1}{B}\right)+\frac{12C_4I\balp_{t+1}^2\sigma^2T}{m\beta_t B}\\
&\quad +  \frac{32C_1}{m\lambda^2 }\sum_{t=1}^T\tilde{\delta}_{s,t}   +\frac{C_2}{m}\sum_{t=1}^T\tilde{\delta}_{H,t}+\frac{24C_4I \balp_{t+1}^2}{m^2\beta_t}\sum_{t=1}^T\delta_{H,t}\\
&\quad +\left(\frac{6C_3}{\beta_t}+\frac{54C_4m L_{gyy}^2}{I\beta_t}+\frac{32L_y^2C_1}{\tau_t^2\tau^2 \lambda^2 }\right)\sum_{t=1}^T\|\x_t-\x_{t+1}\|^2-\frac{1}{2}\sum_{t=1}^T\|\z_{t+1}\|^2\\
&\quad +\left(\frac{6C_3}{m\beta_t}+\frac{54 L_{gyy}^2}{I\beta_t} - \frac{8C_1}{\tau_t^2 \lambda m}(\frac{1}{2\tau} - \frac{3L_g}{4} )\right)\sum_{t=1}^T\|\y_t-\y_{t+1}\|^2\\
    \end{aligned}
\end{equation*}

Then we replace $\tilde{\delta}_{s,t}, \tilde{\delta}_{H,t}$ by $\delta_{s,t}, \delta_{H,t}$ following inequality~(\ref{ineq:D2d}) and (\ref{ineq:xy_sum}).

\begin{equation*}
    \begin{aligned}
        &\E\left[\sum_{t=1}^T\|\nabla F(\x_t)\|^2+\sum_{t=1}^T\delta_{z,t}+\frac{1}{m}\sum_{t=1}^T\delta_{y,t}+\frac{1}{m}\sum_{t=1}^T\delta_{H,t}+\frac{1}{m}\sum_{t=1}^T\delta_{s,t}\right]\\
        &\leq \E\Bigg[\frac{2}{\eta_t}(F(\x_1)-F(\x^*))+\frac{3\delta_{z,1}}{\beta_t} +\frac{4C_1}{\tau_t\tau \lambda m}\delta_{y,1} +6\beta_tC_5T\left(\frac{\I(I<m)}{I}+\frac{1}{B}\right)+\frac{12C_4I\balp_{t+1}^2\sigma^2T}{m\beta_t B} +  \frac{C_6}{m}\sum_{t=1}^T\delta_{s,t} \\
&\quad +\frac{C_7}{m}\sum_{t=1}^T\delta_{H,t} +\left(\frac{6C_3}{\beta_t}+\frac{54C_4m L_{gyy}^2}{I\beta_t}+\frac{32L_y^2C_1}{\tau_t^2\tau^2 \lambda^2 }+\frac{64C_1L_{gy}^2}{\lambda^2 }+2C_2L_{gyy}^2\right)\sum_{t=1}^T\|\x_t-\x_{t+1}\|^2-\frac{1}{2}\sum_{t=1}^T\|\z_{t+1}\|^2\\
&\quad +\left(\frac{6C_3}{m\beta_t}+\frac{54 L_{gyy}^2}{I\beta_t} - \frac{8C_1}{\tau_t^2 \lambda m}(\frac{1}{2\tau} - \frac{3L_g}{4} )+\frac{64C_1L_{gy}^2}{m\lambda^2 }+\frac{2C_2L_{gyy}^2}{m}\right)\sum_{t=1}^T\|\y_t-\y_{t+1}\|^2\\
    \end{aligned}
\end{equation*}
where $\frac{C_6}{m}\geq \frac{64C_1}{m\lambda^2 }+\frac{1}{m}$ and $\frac{C_7}{m}\geq \frac{2C_2+1}{m}+\frac{24C_4I \balp_{t+1}^2}{m^2\beta_t}$

Then we plug in $\sum_{t=1}^T \delta_{s,t}$ and $\sum_{t=1}^T \delta_{H,t}$ to the right hand side following (\ref{ineq:305}) and (\ref{ineq:304}), 

\begin{equation*}
    \begin{aligned}
        &\E\left[\sum_{t=1}^T\|\nabla F(\x_t)\|^2+\sum_{t=1}^T\delta_{z,t}+\frac{1}{m}\sum_{t=1}^T\delta_{y,t}+\frac{1}{m}\sum_{t=1}^T\delta_{H,t}+\frac{1}{m}\sum_{t=1}^T \delta_{s,t}\right]\\
&\leq \E\Bigg[\frac{2}{\eta_t}(F(\x_1)-F(\x^*))+\frac{3\delta_{z,1}}{\beta_t} +\frac{4C_1}{\tau_t\tau \lambda m}\delta_{y,1} + \frac{C_6}{I \alpha_t}\delta_{s,1}+\frac{C_7}{I \balp_t}\delta_{H,1}+6\beta_t C_5T\left(\frac{\I(I<m)}{I}+\frac{1}{B}\right)\\
&\quad +\frac{12C_4I\balp_{t+1}^2\sigma^2T}{m\beta_t B}+\frac{2C_6\alpha_t \sigma^2T}{B}+\frac{2C_7\balp_t \sigma^2T}{B} +\Bigg(\frac{6C_3}{\beta_t}+\frac{54C_4m L_{gyy}^2}{I\beta_t}+\frac{32L_y^2C_1}{\tau_t^2\tau^2 \lambda^2 }+\frac{64C_1L_{gy}^2}{\lambda^2 }+2C_2L_{gyy}^2\\
&\quad +\frac{8m^2 C_6L_{gy}^2}{\alpha_t I^2}+\frac{8m^2 C_7L_{gyy}^2}{\balp_t I^2}\Bigg)\sum_{t=1}^T\|\x_t-\x_{t+1}\|^2-\frac{1}{2}\sum_{t=1}^T\|\z_{t+1}\|^2+\Bigg(\frac{6C_3}{m\beta_t}+\frac{54 L_{gyy}^2}{I\beta_t}\\
&\quad  - \frac{8C_1}{\tau_t^2 \lambda m}(\frac{1}{2\tau} - \frac{3L_g}{4} )+\frac{64C_1L_{gy}^2}{m\lambda^2 }+\frac{2C_2L_{gyy}^2}{m}+\frac{8m C_6L_{gy}^2}{\alpha_t I^2}+\frac{8m C_7L_{gyy}^2}{\balp_t I^2}\Bigg)\sum_{t=1}^T\|\y_t-\y_{t+1}\|^2\\
    \end{aligned}
\end{equation*}
To ensure the coefficient of $\sum_{t=1}^T\|\y_t-\y_{t+1}\|^2$ is non positive, we need 
\begin{equation*}
\begin{aligned}
\tau_t^2&= C_8 \min\left\{ \frac{I\beta_t}{m}, \frac{\balp_t I^2}{m^2},\frac{\alpha_t I^2}{m^2}\right\}\\
&\leq \frac{48C_1}{ \lambda m}(\frac{1}{2\tau} - \frac{3L_g}{4} ) \min \left\{\frac{m\beta_t}{6C_3},\frac{I\beta_t}{54 L_{gyy}^2},\frac{m\lambda^2 }{64C_1L_{gy}^2},\frac{m}{2C_2L_{gyy}^2},\frac{\alpha_t I^2}{8m C_6L_{gy}^2},\frac{\balp_t I^2}{8m C_7L_{gyy}^2}\right\}\\
\end{aligned}
\end{equation*}
where $C_8:=\frac{48C_1}{ \lambda I}(\frac{1}{2\tau} - \frac{3L_g}{4} ) \min \left\{\frac{1}{6C_3},\frac{1}{54 L_{gyy}^2},\frac{\lambda^2 }{64C_1L_{gy}^2},\frac{1}{2C_2L_{gyy}^2},\frac{1}{8 C_6L_{gy}^2},\frac{1}{8 C_7L_{gyy}^2}\right\}$. 

To ensure the coefficient of $\sum_{t=1}^T\|\x_t-\x_{t+1}\|^2$ is non positive, we need 
\begin{equation*}
\begin{aligned}
\eta_t^2 & = C_9\min\left\{\frac{I\beta_t}{m}, \frac{\balp_t I^2}{m^2 }, \tau_t^2, \frac{\alpha_t I^2}{m^2 } \right\}\\
&\leq \min\{ \frac{\beta_t}{84C_3},\frac{I\beta_t}{756C_4m L_{gyy}^2},\frac{\tau_t^2\tau^2 \lambda^2 }{448L_y^2C_1},\frac{\lambda^2 }{896C_1L_{gy}^2},\frac{1}{28C_2L_{gyy}^2},\frac{\alpha_t I^2}{112m^2 C_6L_{gy}^2},\frac{\balp_t I^2}{112m^2 C_7L_{gyy}^2}\}
\end{aligned}
\end{equation*}
where $C_9:= \min\{ \frac{1}{84C_3},\frac{1}{756C_4 L_{gyy}^2},\frac{\tau^2 \lambda^2}{448L_y^2C_1},\frac{\lambda^2 }{896C_1L_{gy}^2},\frac{1}{28C_2L_{gyy}^2},\frac{1}{112 C_6L_{gy}^2},\frac{1}{112 C_7L_{gyy}^2}\}$

Then it follows

\begin{equation*}
\begin{aligned}
 &\E\left[\sum_{t=1}^T\|\nabla F(\x_t)\|^2+\sum_{t=1}^T\delta_{z,t}+\frac{1}{m}\sum_{t=1}^T\delta_{y,t}+\frac{1}{m}\sum_{t=1}^T\delta_{H,t}+\frac{1}{m}\sum_{t=1}^T\delta_{s,t}\right]\\
&\leq \E\Bigg[\frac{2}{\eta_t}(F(\x_1)-F(\x^*))+\frac{3\delta_{z,1}}{\beta_t} +\frac{4C_1}{\tau_t\tau \lambda m}\delta_{y,1} + \frac{C_6}{I \alpha_t}\delta_{s,1}+\frac{C_7}{I \balp_t}\delta_{H,1}\\
&\quad +6\beta_t C_5T\left(\frac{\I(I<m)}{I}+\frac{1}{B}\right)+\frac{12C_4I\balp_{t+1}^2\sigma^2T}{m\beta_t B}+\frac{2C_6\alpha_t \sigma^2T}{B}+\frac{2C_7\balp_t \sigma^2T}{B}\Bigg] \\
&\leq \E\Bigg[\frac{2\Delta}{\eta_t}+C_{10}\Bigg(\frac{1}{\beta_t}\delta_{z,1} + \frac{1}{ I \alpha_t}\delta_{s,1}+\frac{1}{\tau_t I}\delta_{y,1}+\frac{1}{I \balp_t}\delta_{H,1}+\frac{\balp_t T}{B} +\frac{\alpha_t T}{B} +\beta_tT\left(\frac{\I(I<m)}{I}+\frac{1}{B}\right)+\frac{I\balp_{t+1}^2T}{m\beta_t B}\Bigg)\Bigg]\\
\end{aligned}
\end{equation*}
where $C_{10}:=\max\left\{3 ,\frac{4C_1}{\tau \lambda } ,C_6,C_7,6 C_5,12C_4\sigma^2,2C_6 \sigma^2,2C_7 \sigma^2\right\}$.


Set $\alpha_t \leq \frac{B\epsilon^2}{12C_{10}}$, $\balp_t,\beta_t \leq \frac{\epsilon^2}{12C_{10}\left(\frac{\I(I<m)}{I}+\frac{1}{B}\right)}$, so that 

\begin{equation*}
C_{10}\left(\frac{\balp_t }{B} +\frac{\alpha_t }{B} +\beta_t\left(\frac{\I(I<m)}{I}+\frac{1}{B}\right)+\frac{I\balp_{t+1}^2}{m\beta_t B}\right) = (\frac{\epsilon^2}{12B}+\frac{I\epsilon^2}{12m B})\left(\frac{\I(I<m)}{I}+\frac{1}{B}\right)^{-1} +\frac{\epsilon^2}{12} +\frac{\epsilon^2}{12}\leq \frac{\epsilon^2}{3}
\end{equation*}

As a result, we have
\begin{equation*}
\begin{aligned}
\tau_t^2& \leq C_8 \min\left\{\frac{I\beta_t}{m} ,  \frac{\balp_t I^2}{m^2},\frac{\alpha_t I^2}{m^2}\right\}\\
& = \frac{C_8}{12C_{10}} \min\left\{ \frac{I^2\epsilon^2}{m^2}\left(\frac{\I(I<m)}{I}+\frac{1}{B}\right)^{-1} ,\frac{ I^2B\epsilon^2}{m^2}\right\}\\
& = \frac{C_8}{12C_{10}} \frac{ I^2\epsilon^2}{m^2}\left(\frac{\I(I<m)}{I}+\frac{1}{B}\right)^{-1} 
\end{aligned}
\end{equation*}
and
\begin{equation*}
\begin{aligned}
\eta_t^2 & \leq C_9\min\left\{\frac{I\beta_t}{m }, \frac{\balp_t I^2}{m^2 }, \tau_t^2, \frac{\alpha_t I^2}{m^2 } \right\}\\
& = C_{11}\min\left\{\frac{I\epsilon^2}{m }\left(\frac{\I(I<m)}{I}+\frac{1}{B}\right)^{-1} , \frac{ I^2\epsilon^2}{m^2 }\left(\frac{\I(I<m)}{I}+\frac{1}{B}\right)^{-1} , \frac{ I^2 B\epsilon^2}{m^2 } \right\}\\
& = C_{11} \frac{ I^2\epsilon^2}{m^2 }\left(\frac{\I(I<m)}{I}+\frac{1}{B}\right)^{-1} 
\end{aligned}
\end{equation*}
where $C_{11} = \frac{C_8}{12C_{10}}$.

Thus, with $T = c_T \epsilon^{-3} :=6\Delta \sqrt{C_{11}}\frac{m}{I}\left(\frac{\I(I<m)}{\sqrt{I}}+\frac{1}{\sqrt{B}}\right)\geq 6\Delta \sqrt{C_{11}}\frac{m}{I}\left(\frac{\I(I<m)}{I}+\frac{1}{B}\right)^{1/2} \epsilon^{-3}$, we have
\begin{equation*}
\begin{aligned}
\frac{2\Delta}{\eta_t T} = 2\Delta \sqrt{C_{11}}\frac{m}{I}\left(\frac{\I(I<m)}{I}+\frac{1}{B}\right)^{1/2} \epsilon^{-1}\frac{1}{T} \leq \frac{\epsilon^2}{3}
\end{aligned}
\end{equation*}

Note that $\frac{C_{10}}{T}\E\left[\frac{1}{\beta_t}\delta_{z,1}+\frac{1}{\tau_t I}\delta_{y,1}+\frac{1}{\alpha_t I}\delta_{s,1}+\frac{1}{\balp_t I}\delta_{H,1}\right]\leq \frac{\epsilon^2}{3}$ can be achieved by processing all lower problems at the beginning and finding good initial solutions $\delta_{z,1},\delta_{s,1},\delta_{H,1}=\O(\epsilon)$, with complexity $\O(\epsilon^{-1})$, and $\delta_{y,1}=\O(1)$ with complexity $\O(1)$. Denote the iteration number for initialization as $T_0=\O(\epsilon^{-1})$. Then the total iteration complexity is $\O\left(\frac{m\epsilon^{-3}\mathbb I(I<m)}{I\sqrt{I}} + \frac{m\epsilon^{-3}}{I\sqrt{B}}\right)$.


\end{proof}

\subsection{Convergence Analysis of $\vtwo$}\label{Appen:rsvrbv2}
We first present the formal statement of Theorem~\ref{thm:informal} for $\vtwo$.
\begin{theorem}\label{thm:5}
Under Assumptions~\ref{ass:1}, \ref{ass:2} and \ref{ass:3}, with  $\tau\leq 2/(3L_g)$, $\bgam_t = \frac{m-I}{I(1-\balp_t)}+(1-\balp_t)$, $\gamma_t = \frac{m-I}{I(1-\alpha_t)}+(1-\alpha_t)$, 
 $\tau_t \leq  \sqrt{\frac{C_8'\min\left\{1,C_{\btau}\right\}}{18C_9'}}\frac{I}{m}(\frac{\I(I<m)}{I}+\frac{1}{B})^{-1/2}\epsilon$, $\beta_t\leq \frac{1}{18C_9'}(\frac{\I(I<m)}{I}+\frac{1}{B})^{-1}\epsilon^2$, $\balp_t  \leq\min\left\{\frac{1}{2}, \frac{1}{18C_9'}(\frac{\I(I<m)}{I}+\frac{1}{B})^{-1}\epsilon^2\right\}$, $\alpha_t\leq\min\left\{\frac{1}{2},  \frac{B}{18C_9'}\epsilon^2\right\}$, $\btau_t\leq \min\left\{\frac{\lambda}{8L_{\phi v}^2},\frac{\lambda}{2},\frac{1}{\lambda},\frac{1}{2\sqrt{C_6'}}, \frac{\sqrt{C_{\btau}} I}{m}\sqrt{\balp_t}\right\}$,
$\eta_t\leq \min\left\{\frac{1}{2L_F},\frac{ \sqrt{C_{11}'}I}{m }(\frac{\I(I<m)}{I}+\frac{1}{B})^{-1/2} \epsilon\right\}$, where $C_8',C_9',C_{\btau},C_{11}'$ are constants specified in the proof, and by using a large mini-batch size of $\O(1/\epsilon)$ at the initial iteration for computing $\z_1, \s_1,\u_1$ and computing an accurate solution $\y_1,\v_1$ such that $\delta_{y,1}\leq \O(m)$, Algorithm~\ref{alg:4} gives $\E\left[\frac{1}{T}\sum_{i=1}^m\|\nabla F(\x_t)\|^2\right]\leq \epsilon^2$ with sample complexity $T =  \O\left(\frac{m\epsilon^{-3}\mathbb I(I<m)}{I\sqrt{I}} + \frac{m\epsilon^{-3}}{I\sqrt{B}}\right)$.
\end{theorem}

First, we note that the bounded variance of $\nabla_v \phi_i(v,\x,\y_i;\cB_i)$ can be derived as
\begin{equation*}
\begin{aligned}
    &\E_{\cB_i^t}[\|\nabla_v \phi_i(\v_{i,t},\x_t,\y_{i,t};\cB_i^t,\wB_i^t)-\nabla_v \phi_i(\v_{i,t},\x_t,\y_{i,t})\|^2]\\
    &=\E_{\cB_i^t,\wB_i^t}[\|\nabla^2_{yy}g_i(\x_t,\y_{i,t};\wB_i^t) \v_{i,t} -  \nabla_y f_i(\x_t,\y_{i,t};\cB_i^t)-\nabla^2_{yy}g_i(\x_t,\y_{i,t}) \v^t_i +  \nabla_y f_i(\x_t,\y_{i,t})\|^2]\\
    &\leq \E_{\cB_i^t,\wB_i^t}[2\|\nabla^2_{yy}g_i(\x_t,\y_{i,t};\wB_i^t) \v_{i,t} -\nabla^2_{yy}g_i(\x_t,\y_{i,t}) \v_{i,t} \|^2+2\|  \nabla_y f_i(\x_t,\y_{i,t})-  \nabla_y f_i(\x_t,\y_{i,t};\cB_i^t)\|^2]\\
    &\leq \frac{2\sigma^2}{B}\|\v_{i,t}\|^2+\frac{2\sigma^2}{B}\leq (1+\V^2) \frac{2\sigma^2}{B}.
\end{aligned}
\end{equation*}
Moreover, to achieve the variance-reduced estimation error bound, we need the stochastic gradient $\nabla_v \phi_i(\v_i,\x,\y_i;\xi,\zeta)$ to be $L_{\phi v}$-Lipschitz with some constant $L_{\phi v}$. The value of $L_{\phi v}$ can be derived as following. Assume that $(\v_i,\x,\y_i)$ and $(\v_i',\x',\y_i')$ are parameters from some iterations in algorithm~\ref{alg:4}, then under Assumptions~\ref{ass:2} and \ref{ass:3} we have
\begin{equation*}
    \begin{aligned}
        &\|\nabla_v \phi_i(\v_i,\x,\y_i;\xi,\zeta)-\nabla_v \phi_i(\v_i',\x',\y_i';\xi,\zeta)\|^2\\
        &\leq 4\|\nabla_{yy}^2 g_i(\x,\y_i;\zeta)\v_{i,t}-\nabla_{yy}^2 g_i(\x',\y_i';\zeta)\v_{i,t}\|^2+4\|\nabla_{yy}^2 g_i(\x',\y_i';\zeta)\v_{i,t}-\nabla_{yy}^2 g_i(\x',\y_i';\zeta)\v_{i,t}'\|\\
        &\quad +2\|\nabla_y f_i(\x,\y_i;\xi)-\nabla_y f_i(\x',\y_i';\xi)\|^2\\
        &\leq (4L_{gyy}^2\V^2+2L_{fy}^2)(\|\x-\x'\|^2+\|\y_i-\y_i'\|^2)+4\tilde{C}_{gyy}^2\|\v_{i,t}-\v_{i,t}'\|^2\\
        &\leq L_{\phi v}^2(\|\v_{i,t}-\v_{i,t}'\|^2+\|\x-\x'\|^2+\|\y_i-\y_i'\|^2)
    \end{aligned}
\end{equation*}
where $L_{\phi v}:=\max\{4L_{gyy}^2\V^2+2L_{fy}^2,4\tilde{C}_{gyy}^2\}$.

\begin{lemma}[\citep{99401}(Lemma 2.2)]
For all $i=1,\dots,m$, $\v_i(\x,\y_i)$ is $L_v$-Lipschitz continuous with $L_v=\frac{L_{fy}C_{gyy}+C_{fy}L_{gyy}}{\lambda^2}$.
\end{lemma}


Define
\begin{equation}
    \Delta_t := \frac{1}{m}\sum_{i=1}^m \nabla_x f_i(\x_{t},\y_{i,t})-\nabla_{xy}^2 g_i(\x_t,\y_{i,t})\v_{i,t}
\end{equation}
Note that $\E_t[G_t]=\Delta_t$, $\E_t[\widetilde{G}_t]=\Delta_{t-1}$. Then we have the following two lemmas.

\begin{lemma}\label{lem:203}
For all $t> 0$, we have
    \begin{equation}
    \begin{aligned}
        &\left\|\Delta_t-\nabla F(\x_t)\right\|^2\leq\frac{C_1'}{m}\delta_{y,t}+\frac{C_2'}{m}\delta_{v,t}
    \end{aligned}
\end{equation}
\end{lemma}

\begin{lemma}\label{lem:202}
For all $t> 0$, we have
    \begin{equation}
    \begin{aligned}
        \E_t[\left\|\z_{t+1}-\Delta_t\right\|^2] &\leq (1-\beta_t) \|\z_t-\Delta_{t-1}\|^2+2C_3'\|\x_t-\x_{t-1}\|^2+\frac{2C_3'}{m}\|\y_t-\y_{t-1}\|^2\\
        &\quad +\frac{2C_4'}{m}\|\v_t-\v_{t-1}\|^2 +2\beta_t^2C_5'\left(\frac{\I(I<m)}{I}+\frac{1}{B}\right)
    \end{aligned}
\end{equation}
\end{lemma}

Following from Lemma~\ref{lem:3}, with update $\y_{i,t+1} = \y_{i,t}-\tau_t\tau \s_{i,t}$ for all $i=1,\dots, m$, with $\tau \leq 2/(3L_g)$, we have 
\begin{equation}\label{ineq:307}
\begin{aligned}
    \E[\delta_{y,t+1}]&\leq (1 - \frac{\tau_t\tau \lambda }{4})\E[\delta_{y,t}] +  \frac{8\tau_t\tau }{\lambda }\E[\tilde{\delta}_{s,t}]  - \frac{2\tau }{\tau_t }(\frac{1}{2\tau} - \frac{3L_g}{4} )\E[\|\y_{t+1}- \y_t\|^2] +\frac{8mL_y^2}{\tau_t\tau \lambda }\E[\|\x_t-\x_{t+1}\|^2]
\end{aligned}
\end{equation}

\begin{lemma}\label{lem:204}
Consider the update $\v_{i,t+1}=\Pi_{\V}[\v_{i,t}-\btau_t \u_{i,t}]$ for all $i=1,\dots,m$, with $\btau_t\leq \min\left\{\frac{\lambda}{8L_{\phi v}^2},\frac{\lambda}{2},\frac{1}{\lambda}\right\}$, we have
\begin{equation}
\begin{aligned}
\E[\delta_{v,t+1}]&\leq (1-\frac{\lambda\btau_t}{4})\E[\delta_{v,t}]+10\lambda\btau_t\E[\tilde{\delta}_{u,t}]  +\frac{5mL_v^2}{\lambda\btau_t}E[\|\x_t-\x_{t+1}\|^2]+\frac{5L_v^2}{\lambda\btau_t}\E[\|\y_{t}-\y_{t+1}\|^2]\\
\end{aligned}
\end{equation}
\end{lemma}

By applying Lemma~\ref{lem:MSVR} to $\s_t,\u_t$, we have
\begin{equation*}
\begin{aligned}
    \E\left[\delta_{s,t+1}\right]&\leq (1-\frac{I \alpha_t}{m})\E\left[\delta_{s,t}\right]+\frac{2I\alpha_t^2 \sigma^2}{B} +\frac{8m^2 L_{gy}^2}{I}\E\left[\|\x_{t+1}-\x_t\|^2+\frac{1}{m}\|\y_{t+1}-\y_t\|^2\right]
\end{aligned}
\end{equation*}
and
\begin{equation*}
\begin{aligned}
    \E\left[\delta_{u,t+1}\right]&\leq (1-\frac{I \balp_t}{m})\E\left[\delta_{u,t}\right]+\frac{2I\balp_t^2 \sigma^2}{B} +\frac{8m^2 L_{\phi v}^2}{I}\E\left[\|\x_{t+1}-\x_t\|^2+\frac{1}{m}\|\y_{t+1}-\y_t\|^2+\frac{1}{m}\|\v_{t+1}-\v_t\|^2\right]
\end{aligned}
\end{equation*}
which implies

\begin{equation}\label{ineq:24}
\begin{aligned}
    \E\left[\sum_{t=1}^T\delta_{s,t}\right]&\leq \E\Bigg[\frac{m}{I \alpha_t}\delta_{s,1}+\frac{2m\alpha_t \sigma^2T}{B} +\frac{8m^3 L_{gy}^2}{\alpha_t I^2}\sum_{t=1}^T\left(\|\x_{t+1}-\x_t\|^2+\frac{1}{m}\|\y_{t+1}-\y_t\|^2\right)\Bigg]
\end{aligned}
\end{equation}
and
\begin{equation}\label{ineq:25}
\begin{aligned}
    \E\left[\sum_{t=1}^T\delta_{u,t}\right]&\leq \E\Bigg[\frac{m}{I \balp_t}\delta_{u,1}+\frac{2m\balp_t \sigma^2T}{B} +\frac{8m^3 L_{\phi v}^2}{\balp_t I^2}\sum_{t=1}^T\bigg(\frac{1}{m}\|\v_{t+1}-\v_t\|^2+\|\x_{t+1}-\x_t\|^2+\frac{1}{m}\|\y_{t+1}-\y_t\|^2\bigg)\Bigg]
\end{aligned}
\end{equation}


\subsubsection{Proof of Theorem~\ref{thm:5}}
\begin{proof}
By Lemma~\ref{lem:01}, we have
\begin{equation}\label{ineq:0}
    \begin{aligned}
        F(\x_{t+1})-F(\x_t)\leq \frac{\eta_t}{2}\|\z_{t+1}-\nabla F(\x_t)\|^2-\frac{\eta_t}{2}\|\nabla F(\x_t)\|^2-\frac{\eta_t}{4}\|\z_{t+1}\|^2.
    \end{aligned}
\end{equation}
The first term on the right hand side can be divided into two terms.
\begin{equation}\label{ineq:1}
    \begin{aligned}
    \left\|\z_{t+1}-\nabla F(\x_t)\right\|^2
    \leq 2\left\|\z_{t+1}-\Delta_t\right\|^2 +2\left\|\Delta_t -\nabla F(\x_t)\right\|^2
    \end{aligned}
\end{equation}
where we have recursion for the first term on the right hand side in Lemma~\ref{lem:202} and the second term is bounded by Lemma~\ref{lem:203}. Combining inequalities~\ref{ineq:0},\ref{ineq:1} and Lemma~\ref{lem:203} gives
\begin{equation}
    \begin{aligned}
        F(\x_{t+1})-F(\x_t)&\leq \eta_t\delta_{z,t}+\frac{\eta_tC_1'}{m}\delta_{y,t}+\frac{\eta_tC_2'}{m}\delta_{v,t} -\frac{\eta_t}{2}\|\nabla F(\x_t)\|^2-\frac{\eta_t}{4}\|\z_{t+1}\|^2.
    \end{aligned}
\end{equation}

Taking summation over $t=1,\dots,T$ yields

\begin{equation}
    \begin{aligned}
        \sum_{t=1}^T\|\nabla F(\x_t)\|^2
        &\leq \frac{2}{\eta_t}(F(\x_1)-F(\x^*))+ 2\sum_{t=1}^T\delta_{z,t}+\frac{2C_1'}{m}\sum_{t=1}^T\delta_{y,t}+\frac{2C_2'}{m}\sum_{t=1}^T\delta_{v,t} -\frac{1}{2}\sum_{t=1}^T\|\z_{t+1}\|^2
    \end{aligned}
\end{equation}

We enlarge the values of constants $C_1',C_2'$ so that
\begin{equation}\label{ineq:20}
    \begin{aligned}
        &\sum_{t=1}^T\|\nabla F(\x_t)\|^2+\sum_{t=1}^T\delta_{z,t}+\frac{1}{m}\sum_{t=1}^T\delta_{y,t}+\frac{1}{m}\sum_{t=1}^T\delta_{v,t}\\
        &\leq \frac{2}{\eta_t}(F(\x_1)-F(\x^*))+ 3\sum_{t=1}^T\delta_{z,t}+\frac{C_1'}{m}\sum_{t=1}^T\delta_{y,t}+\frac{C_2'}{m}\sum_{t=1}^T\delta_{v,t}-\frac{1}{2}\sum_{t=1}^T\|\z_{t+1}\|^2
    \end{aligned}
\end{equation}

Following from inequality~(\ref{ineq:307}) and Lemma~\ref{lem:204} we have

\begin{equation}\label{ineq:22}
\begin{aligned}
    \E\left[\sum_{t=1}^T\delta_{y,t}\right]&\leq \E\Bigg[\frac{4}{\tau_t\tau \lambda}\delta_{y,1} +  \frac{32}{\lambda^2 }\sum_{t=1}^T\tilde{\delta}_{s,t}  - \frac{8}{\tau_t^2 \lambda }(\frac{1}{2\tau} - \frac{3L_g}{4} )\sum_{t=1}^T\|\y_{t+1}- \y_t\|^2+\frac{32mL_y^2}{\tau_t^2\tau^2 \lambda^2}\sum_{t=1}^T\|\x_{t+1}-\x_t\|^2\Bigg]
\end{aligned}
\end{equation}


\begin{equation}\label{ineq:23}
\begin{aligned}
\E\left[\sum_{i=1}^m\delta_{v,t}\right]&\leq\E\Bigg[\frac{4}{\lambda\btau_t}\E[\delta_{v,1}]+40\sum_{t=1}^T\tilde{\delta}_{u,t}  +\frac{20L_v^2}{\lambda^2\btau_t^2}\sum_{t=1}^T\|\x_{t+1}-\x_{t}\|^2+\frac{20mL_v^2}{\lambda^2\btau_t^2}\sum_{t=1}^T\|\y_{t+1}-\y_{t}\|^2\Bigg]\\
\end{aligned}
\end{equation}

It follows from Lemma~\ref{lem:202} that
\begin{equation}\label{ineq:21}
	\begin{aligned}
        \E\left[\sum_{t=1}^T\delta_{z,t}\right]
        &\leq \E\Bigg[\frac{\delta_{z,1}}{\beta_t}+\frac{2C_3'}{\beta_t}\sum_{t=1}^T\|\x_t-\x_{t+1}\|^2+\frac{2C_3'}{m\beta_t}\sum_{t=1}^T\|\y_t-\y_{t+1}\|^2\\
        &\quad+\frac{2C_4'}{m\beta_t}\sum_{t=1}^T\|\v_t-\v_{t+1}\|^2 +2\beta_tC_5'T\left(\frac{\I(I<m)}{I}+\frac{1}{B}\right)\Bigg]
    \end{aligned}
\end{equation}
Note that 
\begin{equation}
\begin{aligned}
\|\v_t-\v_{t+1}\|^2&\leq \sum_{i=1}^m\btau_t^2\|\u_{i,t}\|^2\\
&= \sum_{i=1}^m\btau_t^2\|\u_{i,t}-\nabla_v \phi_i(\v_{i,t},\x_t,\y_{i,t})+\nabla_v \phi_i(\v_{i,t},\x_t,\y_{i,t})\|^2\\
&\leq \sum_{i=1}^m2\btau_t^2\|\u_{i,t}-\nabla_v \phi_i(\v_{i,t},\x_t,\y_{i,t})\|^2+2\btau_t^2\|\nabla_v \phi_i(\v_{i,t},\x_t,\y_{i,t})-\nabla_v \phi_i(\v_i(\x_t,\y_{i,t}),\x_t,\y_{i,t})\|^2\\
&\leq \sum_{i=1}^m2\btau_t^2\|\u_{i,t}-\nabla_v \phi_i(\v_{i,t},\x_t,\y_{i,t})\|^2+2\btau_t^2L_{\phi v}^2\|\v_{i,t}-\v_i(\x_t,\y_{i,t})\|^2\\
&= 2\btau_t^2\tilde{\delta}_{u,t}+2\btau_t^2L_{\phi v}^2 \delta_{v,t}
\end{aligned}
\end{equation}

Taking summation over all iterations and expectation, and combining with inequality~(\ref{ineq:23}) yields
\begin{equation*}
\begin{aligned}
&\E\left[\sum_{t=1}^T\|\v_t-\v_{t+1}\|^2\right]\\
&\leq \E\left[2\btau_t^2\sum_{t=1}^T\tilde{\delta}_{u,t}+2\btau_t^2L_{\phi v}^2\sum_{t=1}^T\delta_{v,t}\right]\\
&\leq \E\Bigg[\left(2\btau_t^2+80L_{\phi v}^2\btau_t^2\right)\sum_{t=1}^T\tilde{\delta}_{u,t}+\frac{8L_{\phi v}^2\btau_t}{\lambda}\delta_{v,1}  +\frac{40mL_{\phi v}^2L_v^2}{\lambda^2}\sum_{t=1}^T\|\x_t-\x_{t+1}\|^2+\frac{40L_v^2L_{\phi v}^2}{\lambda^2}\sum_{t=1}^T\|\y_{t}-\y_{t+1}\|^2\Bigg]\\
&\stackrel{(a)}{\leq} \E\Bigg[\left(4\btau_t^2+160L_{\phi v}^2\btau_t^2\right)\sum_{t=1}^T\delta_{u,t} +\frac{8L_{\phi v}^2\btau_t}{\lambda}\delta_{v,1} +(4\btau_t^2L_{\phi v}^2+160L_{\phi v}^4\btau_t^2) \sum_{t=1}^T\|\v_{t+1}-\v_t\|^2\\
&\quad +\left(\frac{40mL_{\phi v}^2L_v^2}{\lambda^2}+4m\btau_t^2L_{\phi v}^2+160mL_{\phi v}^4\btau_t^2\right)\sum_{t=1}^T\|\x_t-\x_{t+1}\|^2\\
&\quad +\left(\frac{40L_v^2L_{\phi v}^2}{\lambda^2}+4\btau_t^2L_{\phi v}^2+160L_{\phi v}^4\btau_t^2\right)\sum_{t=1}^T\|\y_{t}-\y_{t+1}\|^2\Bigg]\\
&\stackrel{(b)}{\leq} \E\Bigg[C_6'\btau_t^2\sum_{t=1}^T\delta_{u,t} +\frac{8L_{\phi v}^2\btau_t}{\lambda}\delta_{v,1} +C_6'\btau_t^2 \sum_{t=1}^T\|\v_{t+1}-\v_t\|^2+\left(\frac{40mL_{\phi v}^2L_v^2}{\lambda^2}+mC_6'\btau_t^2\right)\sum_{t=1}^T\|\x_t-\x_{t+1}\|^2\\
&\quad +\left(\frac{40L_v^2L_{\phi v}^2}{\lambda^2}+C_6'\btau_t^2\right)\sum_{t=1}^T\|\y_{t}-\y_{t+1}\|^2\Bigg]
\end{aligned}
\end{equation*}
where inequality (a) follows from inequality~(\ref{ineq:D2d}) and (\ref{ineq:xy_sum}), and in (b) we denote $C_6'=(4+160L_{\phi v}^2)\max\{1,L_{\phi v}^2\}$.

Combining with inequality~(\ref{ineq:25}), we have
\begin{equation*}
\begin{aligned}
&\E\left[\sum_{t=1}^T\|\v_t-\v_{t+1}\|^2\right]\\
&\leq \E\Bigg[\frac{mC_6'\btau_t^2}{I \balp_t}\delta_{u,1} +\frac{8L_{\phi v}^2\btau_t}{\lambda}\delta_{v,1} +\frac{2C_6'\btau_t^2m\balp_t \sigma^2T}{B} +\left(C_6'\btau_t^2+\frac{8C_6'\btau_t^2m^2 L_{\phi v}^2}{\balp_t I^2}\right) \sum_{t=1}^T\|\v_{t+1}-\v_t\|^2\\
&\quad +\left(\frac{40mL_{\phi v}^2L_v^2}{\lambda^2}+mC_6'\btau_t^2+\frac{8C_6'\btau_t^2m^3 L_{\phi v}^2}{\balp_t I^2}\right)\sum_{t=1}^T\|\x_t-\x_{t+1}\|^2\\
&\quad +\left(\frac{40L_v^2L_{\phi v}^2}{\lambda^2}+C_6'\btau_t^2+\frac{8C_6'\btau_t^2m^2 L_{\phi v}^2}{\balp_t I^2}\right)\sum_{t=1}^T\|\y_{t}-\y_{t+1}\|^2\Bigg]
\end{aligned}
\end{equation*}

Setting $\btau_t^2 \leq \min\{\frac{1}{4C_6'}, C_{\btau}\frac{I^2}{m^2}\balp_t\}$ where $C_{\btau}:= \frac{1}{32C_6' L_{\phi v}^2}$, i.e. $\frac{8C_6'm^2 L_{\phi v}^2\btau_t^2}{\balp_t I^2}\leq \frac{1}{4}$ and $C_6'\btau^2\leq \frac{1}{4}$, we have
\begin{equation}\label{ineq:29}
\begin{aligned}
\E\left[\sum_{t=1}^T\|\v_t-\v_{t+1}\|^2\right]
&\leq \E\Bigg[\frac{2mC_6'\btau_t^2}{I \balp_t}\delta_{u,1} +\frac{16L_{\phi v}^2\btau_t}{\lambda}\delta_{v,1} +\frac{4C_6'\btau_t^2m\balp_t \sigma^2T}{B} \\
&\quad +\left(\frac{80mL_{\phi v}^2L_v^2}{\lambda^2}+2mC_6'\btau_t^2+\frac{16C_6'\btau_t^2m^3 L_{\phi v}^2}{\balp_t I^2}\right)\sum_{t=1}^T\|\x_t-\x_{t+1}\|^2\\
&\quad +\left(\frac{80L_v^2L_{\phi v}^2}{\lambda^2}+2C_6'\btau_t^2+\frac{16C_6'\btau_t^2m^2 L_{\phi v}^2}{\balp_t I^2}\right)\sum_{t=1}^T\|\y_{t}-\y_{t+1}\|^2\Bigg]
\end{aligned}
\end{equation}

Combining inequalities~(\ref{ineq:20}), (\ref{ineq:21}), (\ref{ineq:22}), (\ref{ineq:23}), we obtain

\begin{equation}
    \begin{aligned}
        &\E\Bigg[\sum_{t=1}^T\|\nabla F(\x_t)\|^2+\sum_{t=1}^T\delta_{z,t}+\frac{1}{m}\sum_{t=1}^T\delta_{y,t}+\frac{1}{m}\sum_{t=1}^T\delta_{v,t}\Bigg]\\
        &\leq \E\Bigg[\frac{2}{\eta_t}(F(\x_1)-F(\x^*))+ \frac{3\delta_{z,1}}{\beta_t}+6\beta_tC_5'T\left(\frac{\I(I<m)}{I}+\frac{1}{B}\right)+\frac{4C_1'}{m\tau_t\tau \lambda}\delta_{y,1} +\frac{4C_2'}{m\lambda\btau_t}\delta_{v,1}\\
        &\quad +\left(\frac{6C_3'}{\beta_t}+\frac{32C_1'L_y^2}{\tau_t^2\tau^2 \lambda^2}+\frac{20C_2'L_v^2}{\lambda^2\btau_t^2}\right)\sum_{t=1}^T\|\x_t-\x_{t+1}\|^2-\frac{1}{2}\sum_{t=1}^T\|\z_{t+1}\|^2\\
        &\quad +\left(\frac{6C_3'}{m\beta_t} - \frac{8C_1'}{m\tau_t^2 \lambda }(\frac{1}{2\tau} - \frac{3L_g}{4} )+\frac{20C_2'L_v^2}{m\lambda^2\btau_t^2}\right)\sum_{t=1}^T\|\y_t-\y_{t+1}\|^2\\
        &\quad +\frac{6C_4'}{m\beta_t}\sum_{t=1}^T\|\v_t-\v_{t+1}\|^2  +  \frac{32C_1'}{m\lambda^2 }\sum_{t=1}^T\tilde{\delta}_{s,t}  +\frac{40C_2'}{m}\sum_{t=1}^T\tilde{\delta}_{u,t}  \Bigg]\\
         &\stackrel{(a)}{\leq} \E\Bigg[\frac{2}{\eta_t}(F(\x_1)-F(\x^*))+ \frac{3\delta_{z,1}}{\beta_t}+6\beta_tC_5'T\left(\frac{\I(I<m)}{I}+\frac{1}{B}\right)+\frac{4C_1'}{m\tau_t\tau \lambda}\delta_{y,1} +\frac{4C_2'}{m\lambda\btau_t}\delta_{v,1}\\
        &\quad +\left(\frac{6C_3'}{\beta_t}+\frac{32C_1'L_y^2}{\tau_t^2\tau^2 \lambda^2}+\frac{20C_2'L_v^2}{\lambda^2\btau_t^2}+\frac{64C_1'L_{gy}^2}{\lambda^2 }+80C_2'L_{\phi v}^2\right)\sum_{t=1}^T\|\x_t-\x_{t+1}\|^2-\frac{1}{2}\sum_{t=1}^T\|\z_{t+1}\|^2\\
        &\quad +\left(\frac{6C_3'}{m\beta_t} - \frac{8C_1'}{m\tau_t^2 \lambda }(\frac{1}{2\tau} - \frac{3L_g}{4} )+\frac{20C_2'L_v^2}{m\lambda^2\btau_t^2}+\frac{64C_1'L_{gy}^2}{m\lambda^2 }+\frac{80C_2'L_{gyy}^2}{m}\right)\sum_{t=1}^T\|\y_t-\y_{t+1}\|^2\\
        &\quad +\left(\frac{6C_4'}{m\beta_t}+\frac{80C_2'L_{gyy}^2}{m}\right)\sum_{t=1}^T\|\v_t-\v_{t+1}\|^2  +  \frac{64C_1'}{m\lambda^2 }\sum_{t=1}^T\delta_{s,t}  +\frac{80C_2'}{m}\sum_{t=1}^T\delta_{u,t}  
    \end{aligned}
\end{equation}
where (a) uses inequality~(\ref{ineq:D2d}) and (\ref{ineq:xy_sum}).

Enlarge the value of constant $C_2'$ so that $C_2'\geq \max\left\{\frac{64C_1'}{\lambda^2}+1, 80C_2'+1\right\}$.

Combining with inequalities~(\ref{ineq:24}), (\ref{ineq:25}), we have
\begin{equation}
    \begin{aligned}
    &\E\Bigg[\sum_{t=1}^T\|\nabla F(\x_t)\|^2+\sum_{t=1}^T\delta_{z,t}+\frac{1}{m}\sum_{t=1}^T\delta_{y,t}+\frac{1}{m}\sum_{t=1}^T\delta_{v,t}+\frac{1}{m}\sum_{t=1}^T\delta_{s,t}+\frac{1}{m}\sum_{t=1}^T\delta_{u,t}\Bigg]\\
        &\leq \E\Bigg[\frac{2}{\eta_t}(F(\x_1)-F(\x^*))+ \frac{3\delta_{z,1}}{\beta_t}+6\beta_tC_5'T\left(\frac{\I(I<m)}{I}+\frac{1}{B}\right)+\frac{4C_1'}{m\tau_t\tau \lambda}\delta_{y,1} +\frac{4C_2'}{m\lambda\btau_t}\delta_{v,1}+  \frac{64C_1'}{\lambda^2 I \alpha_t}\delta_{s,1}+\frac{80C_2'}{I \balp_t}\delta_{u,1}\\
        &\quad +\frac{128C_1'\alpha_t \sigma^2T}{\lambda^2 B} +\frac{160C_2'\balp_t \sigma^2T}{B} +\Bigg(\frac{6C_3'}{\beta_t}+\frac{32C_1'L_y^2}{\tau_t^2\tau^2 \lambda^2}+\frac{20C_2'L_v^2}{\lambda^2\btau_t^2}+\frac{64C_1'L_{gy}^2}{\lambda^2 }+80C_2'L_{\phi v}^2\\
        &\quad +\frac{512m^2L_{gy}^2}{\lambda^2 \alpha_t I^2}+\frac{640C_2'm^2 L_{\phi v}^2}{\balp_t I^2}\Bigg)\sum_{t=1}^T\|\x_t-\x_{t+1}\|^2-\frac{1}{2}\sum_{t=1}^T\|\z_{t+1}\|^2 +\Bigg(\frac{6C_3'}{m\beta_t} - \frac{8C_1'}{m\tau_t^2 \lambda }(\frac{1}{2\tau} - \frac{3L_g}{4} )\\
        &\quad +\frac{20C_2'L_v^2}{m\lambda^2\btau_t^2}+\frac{64C_1'L_{gy}^2}{m\lambda^2 }+\frac{80C_2'L_{gyy}^2}{m}+\frac{512mL_{gy}^2}{\lambda^2 \alpha_t I^2}+\frac{640C_2'm L_{\phi v}^2}{\balp_t I^2}\Bigg)\sum_{t=1}^T\|\y_t-\y_{t+1}\|^2\\
        &\quad +\left(\frac{6C_4'}{m\beta_t}+\frac{80C_2'L_{gyy}^2}{m}+\frac{640C_2'm L_{\phi v}^2}{\balp_t I^2}\right)\sum_{t=1}^T\|\v_t-\v_{t+1}\|^2  \Bigg]
    \end{aligned}
\end{equation}
Combining with inequality~(\ref{ineq:29}), we have
\begin{equation}
    \begin{aligned}
    &\E\Bigg[\sum_{t=1}^T\|\nabla F(\x_t)\|^2+\sum_{t=1}^T\delta_{z,t}+\frac{1}{m}\sum_{t=1}^T\delta_{y,t}+\frac{1}{m}\sum_{t=1}^T\delta_{v,t}+\frac{1}{m}\sum_{t=1}^T\delta_{s,t}+\frac{1}{m}\sum_{t=1}^T\delta_{u,t}\Bigg]\\
        &\leq \E\Bigg[\frac{2}{\eta_t}(F(\x_1)-F(\x^*))+ \frac{3\delta_{z,1}}{\beta_t}+6\beta_tC_5'T\left(\frac{\I(I<m)}{I}+\frac{1}{B}\right)+\frac{4C_1'}{m\tau_t\tau \lambda}\delta_{y,1} +\frac{4C_2'}{m\lambda\btau_t}\delta_{v,1}+  \frac{64C_1'}{\lambda^2 I \alpha_t}\delta_{s,1}+\frac{80C_2'}{I \balp_t}\delta_{u,1}\\
        &\quad +\frac{128C_1'\alpha_t \sigma^2T}{\lambda^2 B} +\frac{160C_2'\balp_t \sigma^2T}{B} +\Bigg(\frac{6C_3'}{\beta_t}+\frac{32C_1'L_y^2}{\tau_t^2\tau^2 \lambda^2}+\frac{20C_2'L_v^2}{\lambda^2\btau_t^2}+\frac{64C_1'L_{gy}^2}{\lambda^2 }+80C_2'L_{\phi v}^2\\
        &\quad +\frac{512m^2L_{gy}^2}{\lambda^2 \alpha_t I^2}+\frac{640C_2'm^2 L_{\phi v}^2}{\balp_t I^2}\Bigg)\sum_{t=1}^T\|\x_t-\x_{t+1}\|^2-\frac{1}{2}\sum_{t=1}^T\|\z_{t+1}\|^2 +\Bigg(\frac{6C_3'}{m\beta_t} - \frac{8C_1'}{m\tau_t^2 \lambda }(\frac{1}{2\tau} - \frac{3L_g}{4} )\\
        &\quad +\frac{20C_2'L_v^2}{m\lambda^2\btau_t^2}+\frac{64C_1'L_{gy}^2}{m\lambda^2 }+\frac{80C_2'L_{gyy}^2}{m}+\frac{512mL_{gy}^2}{\lambda^2 \alpha_t I^2}+\frac{640C_2'm L_{\phi v}^2}{\balp_t I^2}\Bigg)\sum_{t=1}^T\|\y_t-\y_{t+1}\|^2\\
        &\quad +\left(\frac{6C_4'}{m\beta_t}+\frac{80C_2'L_{gyy}^2}{m}+\frac{640C_2'm L_{\phi v}^2}{\balp_t I^2}\right)\Bigg[\frac{2mC_6'\btau_t^2}{I \balp_t}\delta_{u,1} +\frac{16L_{\phi v}^2\btau_t}{\lambda}\delta_{v,1} +\frac{4C_6'\btau_t^2m\balp_t \sigma^2T}{B} \\
&\quad +\left(\frac{80mL_{\phi v}^2L_v^2}{\lambda^2}+2mC_6'\btau_t^2+\frac{16C_6'\btau_t^2m^3 L_{\phi v}^2}{\balp_t I^2}\right)\sum_{t=1}^T\|\x_t-\x_{t+1}\|^2\\
&\quad +\left(\frac{80L_v^2L_{\phi v}^2}{\lambda^2}+2C_6'\btau_t^2+\frac{16C_6'\btau_t^2m^2 L_{\phi v}^2}{\balp_t I^2}\right)\sum_{t=1}^T\|\y_{t}-\y_{t+1}\|^2\Bigg]  \Bigg]
    \end{aligned}
\end{equation}
Recall that $\btau_t^2 \leq \min\{\frac{1}{4C_6'}, C_{\btau}\frac{I^2}{m^2}\balp_t\}$, i.e. $2C_6'\btau_t^2+\frac{16C_6'\btau_t^2m^2 L_{\phi v}^2}{\balp_t I^2}\leq 1$, and let 
\begin{equation*}
    C_7'=\max\left\{18C_4', 240C_2'L_{gyy}^2,1920C_2' L_{\phi v}^2\right\}\left(\frac{80L_v^2L_{\phi v}^2}{\lambda^2}+1\right) 
\end{equation*}
then
\begin{equation}
    \begin{aligned}
    &\E\Bigg[\sum_{t=1}^T\|\nabla F(\x_t)\|^2+\sum_{t=1}^T\delta_{z,t}+\frac{1}{m}\sum_{t=1}^T\delta_{y,t}+\frac{1}{m}\sum_{t=1}^T\delta_{v,t}+\frac{1}{m}\sum_{t=1}^T\delta_{s,t}+\frac{1}{m}\sum_{t=1}^T\delta_{u,t}\Bigg]\\
        &\leq \E\Bigg[\frac{2}{\eta_t}(F(\x_1)-F(\x^*))+ \frac{3\delta_{z,1}}{\beta_t}+6\beta_tC_5'T\left(\frac{\I(I<m)}{I}+\frac{1}{B}\right)+\frac{4C_1'}{m\tau_t\tau \lambda}\delta_{y,1} +\frac{4C_2'}{m\lambda\btau_t}\delta_{v,1}+  \frac{64C_1'}{\lambda^2 I \alpha_t}\delta_{s,1}\\
        &\quad +\frac{80C_2'}{I \balp_t}\delta_{u,1}+\frac{128C_1'\alpha_t \sigma^2T}{\lambda^2 B} +\frac{160C_2'\balp_t \sigma^2T}{B} +\Bigg(\frac{6C_3'}{\beta_t}+\frac{32C_1'L_y^2}{\tau_t^2\tau^2 \lambda^2}+\frac{20C_2'L_v^2}{\lambda^2\btau_t^2}+\frac{64C_1'L_{gy}^2}{\lambda^2 }+80C_2'L_{\phi v}^2\\
        &\quad +\frac{512m^2L_{gy}^2}{\lambda^2 \alpha_t I^2}+\frac{640C_2'm^2 L_{\phi v}^2}{\balp_t I^2}\Bigg)\sum_{t=1}^T\|\x_t-\x_{t+1}\|^2-\frac{1}{2}\sum_{t=1}^T\|\z_{t+1}\|^2 +\Bigg(\frac{6C_3'}{m\beta_t} - \frac{8C_1'}{m\tau_t^2 \lambda }(\frac{1}{2\tau} - \frac{3L_g}{4} )\\
        &\quad +\frac{20C_2'L_v^2}{m\lambda^2\btau_t^2}+\frac{64C_1'L_{gy}^2}{m\lambda^2 }+\frac{80C_2'L_{gyy}^2}{m}+\frac{512mL_{gy}^2}{\lambda^2 \alpha_t I^2}+\frac{640C_2'm L_{\phi v}^2}{\balp_t I^2}\Bigg)\sum_{t=1}^T\|\y_t-\y_{t+1}\|^2\\
        &\quad +C_7'\left(\frac{1}{m\beta_t}+\frac{1}{m}+\frac{m }{\balp_t I^2}\right)\left(\frac{2mC_6'\btau_t^2}{I \balp_t}\delta_{u,1} +\frac{16L_{\phi v}^2\btau_t}{\lambda}\delta_{v,1} +\frac{4C_6'\btau_t^2m\balp_t \sigma^2T}{B} \right)\\
&\quad +C_7'\left(\frac{1}{\beta_t}+1+\frac{m^2 }{\balp_t I^2}\right)\sum_{t=1}^T\|\x_t-\x_{t+1}\|^2 +C_7'\left(\frac{1}{m\beta_t}+\frac{1}{m}+\frac{m }{\balp_t I^2}\right)\sum_{t=1}^T\|\y_{t}-\y_{t+1}\|^2\Bigg]  \Bigg]
    \end{aligned}
\end{equation}

To ensure the coefficient of $\sum_{t=1}^T\|\y_{t+1}-\y_t\|^2$, we set
\begin{equation*}
\begin{aligned}
\tau_t^2 &= C_8'\min\left\{\beta_t,\btau_t^2,\frac{I^2 \alpha_t}{m^2},\frac{I^2 \balp_t}{m^2}\right\}\\
&\leq \frac{8C_1'}{9m \lambda }(\frac{1}{2\tau} - \frac{3L_g}{4} ) \min\left\{\frac{m\beta_t}{6C_3'}  ,\frac{m\lambda^2\btau_t^2}{20C_2'L_v^2}, \frac{m\lambda^2 }{64C_1'L_{gy}^2}, \frac{m}{80C_2'L_{gyy}^2}, \frac{\lambda^2 \alpha_t I^2}{512mL_{gy}^2}, \frac{\balp_t I^2}{640C_2'm L_{\phi v}^2},\frac{m\beta_t}{C_7'}, \frac{m}{C_7'}, \frac{\balp_t I^2}{mC_7' }\right\}
\end{aligned}
\end{equation*}
where $C_8' =\frac{8C_1'}{9 \lambda }(\frac{1}{2\tau} - \frac{3L_g}{4} ) \min\left\{\frac{1}{6C_3'}  ,\frac{\lambda^2}{20C_2'L_v^2}, \frac{\lambda^2 }{64C_1'L_{gy}^2}, \frac{1}{80C_2'L_{gyy}^2}, \frac{\lambda^2 }{512L_{gy}^2}, \frac{1}{640C_2' L_{\phi v}^2},\frac{1}{C_7'}\right\}$. Let 
\begin{equation*}
    \begin{aligned}
        C_9' = 11\max\bigg\{ 3, 6 C_5', \frac{4C_1'}{\tau \lambda}, \frac{4C_2'}{\lambda},  \frac{64C_1'}{\lambda^2 }, 80C_2', \frac{128C_1' \sigma^2}{\lambda^2 } , 160C_2' \sigma^2,2C_7'C_6' ,\frac{16C_7'L_{\phi v}^2}{\lambda}, 4C_7'C_6' \sigma^2 \bigg\}
    \end{aligned}
\end{equation*}

It follows
\begin{equation}
    \begin{aligned}
    &\E\Bigg[\sum_{t=1}^T\|\nabla F(\x_t)\|^2+\sum_{t=1}^T\delta_{z,t}+\frac{1}{m}\sum_{t=1}^T\delta_{y,t}+\frac{1}{m}\sum_{t=1}^T\delta_{v,t}+\frac{1}{m}\sum_{t=1}^T\delta_{s,t}+\frac{1}{m}\sum_{t=1}^T\delta_{u,t}\Bigg]\\
        &\leq \E\Bigg[\frac{2}{\eta_t}(F(\x_1)-F(\x^*))+ C_9'\Bigg( \frac{\delta_{z,1}}{\beta_t}+\frac{\delta_{y,1}}{m\tau_t} +\left(\frac{1}{m\btau_t}+\frac{\btau_t}{m\beta_t}+\frac{\btau_t}{m}+\frac{m \btau_t}{\balp_t I^2}\right)\delta_{v,1}\\
        &\quad +  \frac{\delta_{s,1}}{ I \alpha_t}+\left(\frac{1}{I \balp_t}+\frac{\btau_t^2}{I \balp_t\beta_t}+\frac{\btau_t^2}{I \balp_t}+\frac{m^2\btau_t^2}{I \balp_t^2 I^2}\right)\delta_{u,1}\\
        &\quad +\beta_t T\left(\frac{\I(I<m)}{I}+\frac{1}{B}\right)+\frac{\alpha_t T}{B} +\frac{\balp_t T}{B}+\frac{\btau_t^2\balp_t T}{\beta_tB}+\frac{\btau_t^2\balp_t T}{B}+\frac{m^2\btau_t^2 T }{ BI^2}\Bigg) \\
        &\quad +\Bigg(\frac{6C_3'}{\beta_t}+\frac{32C_1'L_y^2}{\tau_t^2\tau^2 \lambda^2}+\frac{20C_2'L_v^2}{\lambda^2\btau_t^2}+\frac{64C_1'L_{gy}^2}{\lambda^2 }+80C_2'L_{\phi v}^2\\
        &\quad +\frac{512m^2L_{gy}^2}{\lambda^2 \alpha_t I^2}+\frac{640C_2'm^2 L_{\phi v}^2}{\balp_t I^2}\Bigg)\sum_{t=1}^T\|\x_t-\x_{t+1}\|^2-\frac{1}{2}\sum_{t=1}^T\|\z_{t+1}\|^2 \\
&\quad +C_7'\left(\frac{1}{\beta_t}+1+\frac{m^2 }{\balp_t I^2}\right)\sum_{t=1}^T\|\x_t-\x_{t+1}\|^2 \Bigg]  \Bigg]
    \end{aligned}
\end{equation}

Setting $\beta_t,\balp_t  \leq \frac{\epsilon
^2}{18C_9'\left(\frac{\I(I<m)}{I}+\frac{1}{B}\right)}$, $\alpha_t\leq \frac{B}{18C_9'}\epsilon^2$, $\btau_t^2 \leq \frac{C_{\btau} I^2}{m^2}\balp_t$, we have
\begin{equation}
C_9'\bigg(\beta_t \left(\frac{\I(I<m)}{I}+\frac{1}{B}\right)+\frac{\alpha_t }{B} +\frac{\balp_t }{B}+\frac{\btau_t^2\balp_t }{\beta_tB}+\frac{\btau_t^2\balp_t }{B}+\frac{m^2\btau_t^2 }{ BI^2}\bigg)\leq \frac{\epsilon^2}{3}
\end{equation}
and
\begin{equation}
\begin{aligned}
\tau_t^2 &\leq C_8'\min\left\{\beta_t,\btau_t^2,\frac{I^2 \alpha_t}{m^2},\frac{I^2 \balp_t}{m^2},\right\}\\
&= \frac{C_8'}{18C_9'}\min\left\{\frac{\epsilon^2}{\frac{\I(I<m)}{I}+\frac{1}{B}},\frac{C_{\btau} I^2\epsilon^2}{m^2\left(\frac{\I(I<m)}{I}+\frac{1}{B}\right)}\epsilon^2,\frac{I^2B }{m^2}\epsilon^2,\frac{I^2\epsilon^2}{m^2\left(\frac{\I(I<m)}{I}+\frac{1}{B}\right)}\epsilon^2,\right\}\\
&= \frac{C_8'}{18C_9'}\frac{I^2 \min\left\{1,C_{\btau}\right\}}{m^2\left(\frac{\I(I<m)}{I}+\frac{1}{B}\right)}\epsilon^2\\
\end{aligned}
\end{equation}

To ensure the coefficient of $\sum_{t=1}^T\|\x_t-\x_{t+1}\|^2$ is non-positive, we set
\begin{equation}
\begin{aligned}
\eta_t^2&\leq C_{11}'\frac{ I^2\epsilon^2}{m^2 } \left(\frac{\I(I<m)}{I}+\frac{1}{B}\right)^{-1}\\
&\leq C_{10}'\min\left\{\beta_t, \tau_t^2, \btau_t^2, \frac{\alpha_t I^2}{m^2 }, \frac{\balp_t I^2}{m^2 } \right\}\\
&\leq \frac{1}{20}\min\left\{\frac{\beta_t}{6C_3'},\frac{\tau_t^2\tau^2 \lambda^2}{32C_1'L_y^2},\frac{\lambda^2\btau_t^2}{20C_2'L_v^2},\frac{\lambda^2 }{64C_1'L_{gy}^2},\frac{1}{80C_2'L_{\phi v}^2},\frac{\lambda^2 \alpha_t I^2}{512m^2L_{gy}^2},\frac{\balp_t I^2}{640C_2'm^2 L_{\phi v}^2},\frac{\beta_t}{C_7'},\frac{1}{C_7'},\frac{\balp_t I^2}{m^2C_7'}\right\}
\end{aligned}
\end{equation}
where $C_{10}'= \frac{1}{20}\min\left\{\frac{1}{6C_3'},\frac{\tau^2 \lambda^2}{32C_1'L_y^2},\frac{\lambda^2}{20C_2'L_v^2},\frac{\lambda^2 }{64C_1'L_{gy}^2},\frac{1}{80C_2'L_{\phi v}^2},\frac{\lambda^2}{512L_{gy}^2},\frac{1}{640C_2' L_{\phi v}^2},\frac{1}{C_7'},\frac{1}{C_7'},\frac{1}{C_7'}\right\}$, and $C_{11}'=C_{10}'\min\left\{\frac{1}{18C_9'},\frac{C_8'\min\left\{1,C_{\btau}\right\}}{18C_9'},\frac{C_{\btau}}{18C_9'}\right\}$

Thus, with $T = c_T \epsilon^{-3} :=6\Delta \frac{m}{\sqrt{C_{11}'}I}\left(\frac{\I(I<m)}{\sqrt{I}}+\frac{1}{\sqrt{B}}\right)\epsilon^{-3}\geq 6\Delta \frac{m}{\sqrt{C_{11}'}I}\left(\frac{\I(I<m)}{I}+\frac{1}{B}\right)^{1/2}\epsilon^{-3}$, we have
\begin{equation}
\begin{aligned}
\frac{2\Delta}{\eta_t T} = 2\Delta 
\frac{m}{\sqrt{C_{11}'}I\sqrt{\min\{I,B\}}}\epsilon^{-1}\frac{1}{T} = \frac{\epsilon^2}{3}
\end{aligned}
\end{equation}

$\frac{C_9'}{T}\bigg[\frac{\delta_{z,1}}{\beta_t}+\frac{\delta_{y,1}}{m\tau_t} +\left(\frac{1}{m\btau_t}+\frac{\btau_t}{m\beta_t}+\frac{\btau_t}{m}+\frac{m \btau_t}{\balp_t I^2}\right)\delta_{v,1}+  \frac{\delta_{s,1}}{ I \alpha_t}+\left(\frac{1}{I \balp_t}+\frac{\btau_t^2}{I \balp_t\beta_t}+\frac{\btau_t^2}{I \balp_t}+\frac{m^2\btau_t^2}{ \balp_t^2 I^3}\right)\delta_{u,1}\bigg] \leq \frac{\epsilon^2}{3}$ can be achieved by processing all lower problems at the beginning and finding good initial solutions $\delta_{z,1},\delta_{s,1},\delta_{u,1}$ with accuracy $\O(\epsilon)$ with complexity $\O(\epsilon^{-1})$, and $\delta_{y,1},\delta_{v,1}$ with accuracy $\O(1)$ with complexity $\O(1)$. Denote the iteration number for initialization as $T_0$. Then the total iteration complexity is $\O\left(\frac{m\epsilon^{-3}\mathbb I(I<m)}{I\sqrt{I}} + \frac{m\epsilon^{-3}}{I\sqrt{B}}\right)$.

\end{proof}

\section{Convergence Analysis of RE-$\name$}

\subsection{Convergence Analysis of RE-$\vone$}

We present the formal statement of Theorem~\ref{thm:RE_informal} for RE-$\vone$.

\begin{theorem}\label{thm:6_RE}
Suppose Assumptions~\ref{ass:1} and~\ref{ass:2} hold and the PL condition holds. Set $\alpha_1=\balp_1=\beta_1\leq \frac{1}{2}$, $\tau_1=\sqrt{\frac{C_8I\alpha_1}{m}}$, $\eta_1=\min\left\{\frac{1}{2L_F},\sqrt{\frac{C_{9}I^2\alpha_1}{m^2}}\right\}$, $T_1 =O\left( \max\left\{ \frac{m}{\mu\eta_1},\frac{m}{\mu\beta_1}\left(\frac{\I(I<m)}{I}+\frac{1}{B}\right)^{-1},\frac{m}{\mu\tau_1 }\left(\frac{\I(I<m)}{I}+\frac{1}{B}\right)^{-1}\right\}\right)$. Define a constant $\epsilon_1 = \frac{7C_{10}(\beta_1+\alpha_1+\balp_1)}{\mu}\left(\frac{\I(I<m)}{I}+\frac{1}{B}\right)$ and $\epsilon_k=\epsilon_1/2^{k-1}$. For $k\geq2$, setting $\beta_k=\alpha_k=\balp_k \leq \frac{\mu \epsilon_k }{21C_{10}}\left(\frac{\I(I<m)}{I}+\frac{1}{B}\right)^{-1}$, $\tau_k= \frac{\sqrt{C_8\alpha_k} I}{m}$, $\eta_k=\sqrt{C_{9}}\sqrt{\min\left\{   \tau_k^2, \frac{\alpha_k I^2}{m^2 } \right\}}$, $T_k = O\left(\max\left\{\frac{1}{\mu\eta_k},\frac{1}{\beta_k},\frac{1}{\tau_k}\right\}\right)$, where and $C_1\sim C_{11}$ are as used in Theorem~\ref{thm:6}, then after  $K=O(\log(\epsilon_1/\epsilon))$ stages, the output of RE-$\vone$ satisfies  $\E[F(\x_K) - F(\x^*)]\leq \epsilon$. 
\end{theorem}

\begin{proof}
Following from the proof of Theorem~\ref{thm:6}, we have
\begin{equation}\label{ineq:28}
\begin{aligned}
 &\E\left[\sum_{t=1}^T\|\nabla F(\x_t)\|^2+\sum_{t=1}^T\delta_{z,t}+\frac{1}{m}\sum_{t=1}^T\delta_{y,t}+\frac{1}{m}\sum_{t=1}^T\delta_{H,t}+\frac{1}{m}\sum_{t=1}^T\delta_{s,t}\right]\\
&\leq \E\Bigg[\frac{2\Delta}{\eta_t}+C_{10}\Bigg(\frac{\delta_{z,1}}{\beta_{t+1}} + \frac{\delta_{s,1}}{ I \alpha_t}+\frac{\delta_{y,1}}{\tau_t m}+\frac{\delta_{H,1}}{I \balp_t}+\frac{\balp_t T}{B} +\frac{\alpha_t T}{B} +\beta_{t+1}T\left(\frac{\I(I<m)}{I}+\frac{1}{B}\right)+\frac{I\balp_{t+1}^2T}{m\beta_{t+1} B}\Bigg)\Bigg]\\
&\stackrel{(a)}{\leq}\E\Bigg[\frac{2\Delta}{\eta_t}+C_{10}\Bigg(\frac{\delta_{z,1}}{\beta_{t+1}} + \frac{\delta_{s,1}}{ m \alpha_t}+\frac{\delta_{y,1}}{\tau_t m}+\frac{\delta_{H,1}}{m \balp_t}+(\balp_t+\alpha_t+\beta_{t+1})T\left(\frac{\I(I<m)}{I}+\frac{1}{B}\right)\Bigg)\Bigg]\\
\end{aligned}
\end{equation}
where in $(a)$ we redefine the constant $C_{10}=\frac{m}{I}C_{10}$ and use the setting $\balp_t=\beta_{t+1}$.

From Theorem~\ref{thm:6}, we know that it is required that $\alpha_0,\balp_0,\beta_0\leq \frac{1}{2}$, $\tau_0=\sqrt{C_8}\sqrt{\min\left\{\frac{I\beta_0}{m},\frac{I^2\balp_0}{m^2},\frac{I^2\alpha_0}{m^2}\right\}}$, $\eta_0=\min\left\{\frac{1}{2L_F},\sqrt{C_{9}}\sqrt{\min\left\{\frac{I\beta_0}{m},\frac{I^2\balp_0}{m^2},\frac{I^2\alpha_0}{m^2}\right\}}\right\}$.

Without loss of generality, let us set $\alpha_0=\balp_0=\beta_0$ and assume that $\epsilon_0 = 2\Delta > \frac{7C_{10} (\beta_0+\alpha_0+\balp_0)}{\mu}\left(\frac{\I(I<m)}{I}+\frac{1}{B}\right)$. The case that $2\Delta \leq \frac{7C_{10} (\beta_0+\alpha_0+\balp_0)}{\mu}\left(\frac{\I(I<m)}{I}+\frac{1}{B}\right)$ can be simply covered by our proof. Then denotes $\epsilon_1 = \frac{7C_{10} (\beta_0+\alpha_0+\balp_0)}{\mu}\left(\frac{\I(I<m)}{I}+\frac{1}{B}\right)$, and $\epsilon_k = \epsilon_1/2^{k-1}$.\\

In the first epoch $(k=1)$, we have initialization such that $F(\x_1)-F(\x^*)\leq \Delta$. In the following, we let the last subscript denote the epoch index. Setting $\eta_1 = \eta_0$, $\beta_1 = \beta_0$, $\alpha_1 = \alpha_0$, $\balp_1 = \balp_0$, $\tau_1=\tau_0$, and $T_1 = \max\left\{ \frac{7\Delta}{\mu\eta_1},\frac{7mC_{10}}{\mu\beta_1}\left(\frac{\I(I<m)}{I}+\frac{1}{B}\right)^{-1}(\delta_{z,0}+\delta_{s,0}+\delta_{w,0}),\frac{7C_{10}m}{\mu\tau_1 }\left(\frac{\I(I<m)}{I}+\frac{1}{B}\right)^{-1}\delta_{y,0}\right\}$. We bound the error of the first stage's output as follows,
\begin{equation}
    \begin{aligned}
        &\E\left[\|\nabla F(\x_1)\|^2+\delta_{z,1}+\frac{1}{m}\delta_{y,1}+\frac{1}{m}\delta_{s,1}+\frac{1}{m}\delta_{H,1}\right]\\
        &\leq \frac{2\Delta}{\eta_1 T_1}+\frac{C_{10}}{T_1}\left(\frac{1}{\beta_1}\delta_{z,0}+\frac{1}{\tau_1 m}\delta_{y,0} +\frac{1}{\alpha_1 m}\delta_{s,0}+\frac{1}{\balp_1 m}\delta_{w,0}\right)+C_{10}(\beta_1+\alpha_1+\balp_1)\left(\frac{\I(I<m)}{I}+\frac{1}{B}\right)\\
        &\leq \mu\epsilon_1
    \end{aligned}
\end{equation}  
where the first inequality uses (\ref{ineq:28}) and the fact that the output of each epoch is randomly sampled from all iterations, and the last line uses the choice of $\eta_1,\beta_1,\alpha_1,\balp_1,\tau_1,T_1,\epsilon_1$. If follows that
\begin{equation}
\E[F(\x_1)-\F(\x^*)]\leq \frac{1}{2\mu}\E[\|\nabla F(\x_1)\|^2]\leq \frac{\epsilon_1}{2}.
\end{equation}
Starting from the second stage, we will prove by induction. Suppose we are at $k$-th stage. Assuming that the output of $(k-1)$-the stage satisfies that $\E[F(\x_{k-1})-F(\x^*)]\leq \epsilon_{k-1}$ and $\E\left[\delta_{z,k-1}+\frac{\delta_{y,k-1}}{m} +\frac{\delta_{s,k-1}}{m}+\frac{\delta_{w,k-1}}{m}\right]\leq \mu \epsilon_{k-1}$, and setting $\beta_k=\alpha_k=\balp_k \leq \frac{\mu \epsilon_k }{21C_{10}}\left(\frac{\I(I<m)}{I}+\frac{1}{B}\right)^{-1}$, $\tau_k^2=  C_8\frac{\alpha_k I^2}{m^2}$, $\eta_k^2=C_{9}\min\left\{ \tau_k^2, \frac{\alpha_k I^2}{m^2 } \right\}$, $T_k = \max\left\{\frac{28}{\mu\eta_k},\frac{7C_{10}}{\beta_k},\frac{7C_{10}}{\tau_k}\right\}$, we have
\begin{equation}
    \begin{aligned}
        &\E\left[\|\nabla F(\x_k)\|^2+\delta_{z,k}+\frac{1}{m}\delta_{y,k}+\frac{1}{m}\delta_{s,k}+\frac{1}{m}\delta_{w,k}\right]\\
        &\leq \E\Bigg[\frac{2(F(\x_{k-1})-F(\x^*))}{\eta_k T_k}+\frac{C_{10}}{T_k}\left(\frac{1}{\beta_k}\delta_{z,k-1}+\frac{1}{\tau_k m}\delta_{y,k-1} +\frac{1}{\alpha_k m}\delta_{s,k-1}+\frac{1}{\balp_k m}\delta_{w,k-1}\right)\\
        &\quad +C_{10}(\beta_k+\alpha_k+\balp_k)\left(\frac{\I(I<m)}{I}+\frac{1}{B}\right) \Bigg]\\
        &\leq \E\Bigg[\frac{2\epsilon_{k-1}}{\eta_k T_k}+\frac{C_{10}\mu \epsilon_{k-1}}{T_k}\left(\frac{1}{\beta_k}+\frac{1}{\tau_k } +\frac{1}{\alpha_k }+\frac{1}{\balp_k}\right) +C_{10}(\beta_k+\alpha_k+\balp_k)\left(\frac{\I(I<m)}{I}+\frac{1}{B}\right)\Bigg]\\
        &\leq \mu\epsilon_k
    \end{aligned}
\end{equation}  
It follows that
\begin{equation}
\begin{aligned}
\E[F(\x_k)-F(\x^*)]\leq \frac{1}{2\mu}\E[\|\nabla F(\x_k)\|^2] \leq \frac{\epsilon_k}{2}.
\end{aligned}
\end{equation}
Thus, after $K=1+\log_2(\epsilon_1/\epsilon)\leq \log_2(\epsilon_0/\epsilon)$ stages, $\E[F(\x_k)-F(\x^*)]\leq \epsilon$.

\end{proof}

\subsection{Convergence Analysis of RE-$\vtwo$}

We present the formal statement of Thoerem~\ref{thm:RE_informal} for RE-$\vtwo$.

\begin{theorem}\label{thm:5_RE}
Suppose Assumptions~\ref{ass:1} and~\ref{ass:2} hold and the PL condition holds. Set $\beta_1=\alpha_1=\balp_1\leq \frac{1}{2}$, $\btau_1 = \min\left\{\frac{\lambda}{8L_{\phi v}^2},\frac{\lambda}{2},\frac{1}{\lambda},\frac{\sqrt{C_{\btau}}I\balp_1^{1/2}}{m}\right\}$, \quad  $\tau_1 = \sqrt{C_8'}\sqrt{\min\left\{\btau_1^2,\frac{I^2 \alpha_1}{m^2}\right\}}$, \quad  $\eta_1 = \min\bigg\{\frac{1}{2L_f},\sqrt{C_{10}'}\min\left\{ \sqrt{m}\tau_1, \sqrt{m}\btau_1, \frac{\sqrt{\alpha_1} I}{m}\right\}\bigg\}$, $T_1 = O\left(\max\left\{ \frac{1}{\mu\eta_1},\frac{\min\{ I ,B\}}{\mu\beta_1^2},\frac{I^2\min\{ I ,B\}}{\mu\tau_1^3 m^3}\right\}\right)$. Define a constant $\epsilon_1 = \frac{7C_9' (\beta_1+\alpha_1+\balp_1)}{\mu\min\{ I ,B\}}$ and $\epsilon_k=\epsilon_1/2^{k-1}$. For $k\geq2$, setting $\beta_k=\alpha_k=\balp_k \leq \frac{\mu \epsilon_k \min\{I,B\}}{21C_9'}$, $\btau_k^2 = \frac{C_{\btau} I^2}{m^2}\balp_k$, $\tau_k = \sqrt{C_8'}\sqrt{\min\left\{\btau_k^2,\frac{I^2 \alpha_k}{m^2}\right\}}$, $\eta_k=\sqrt{C_{10}'}\sqrt{\min\left\{m\tau_k^2, m\btau_k^2, \frac{\alpha_k I^2}{m^2 }\right\}}$, $T_k = O\left(\max\left\{\frac{1}{\mu\eta_k},\frac{1}{\beta_k},\frac{1}{\tau_k m},\frac{1}{\btau_k m}\right\}\right)$, where and $C_1'\sim C_{11}'$ are as used in Theorem~\ref{thm:5}, then after  $K=O(\log(\epsilon_1/\epsilon))$ stages, the output of RE-$\vtwo$ satisfies  $\E[F(\x_K) - F(\x^*)]\leq \epsilon$. 
\end{theorem}

\begin{proof}
Following from the proof of Theorem~\ref{thm:5}, we have 
\begin{equation}\label{ineq:30}
    \begin{aligned}
        &\E\Bigg[\sum_{t=1}^T\|\nabla F(\x_t)\|^2+\sum_{t=1}^T\delta_{z,t}+\frac{1}{m}\sum_{t=1}^T\delta_{y,t}+\frac{1}{m}\sum_{t=1}^T\delta_{v,t}+\frac{1}{m}\sum_{t=1}^T\delta_{s,t}+\frac{1}{m}\sum_{t=1}^T\delta_{u,t}\Bigg]\\
     &\leq \E\Bigg[\frac{2\Delta}{\eta_0}+ C_9'\Bigg[ \frac{\delta_{z,1}}{\beta_0}+\frac{\delta_{y,1}}{m\tau_0} +\left(\frac{1}{m\btau_0}+\frac{\btau_0}{m\beta_0}+\frac{\btau_0}{m}+\frac{m \btau_0}{\balp_0 I^2}\right)\delta_{v,1}+  \frac{\delta_{s,1}}{ I \alpha_0}+\left(\frac{1}{I \balp_0}+\frac{\btau_0^2}{I \balp_0\beta_0}+\frac{\btau_0^2}{I \balp_0}+\frac{m^2\btau_0^2}{I \balp_0^2 I^2}\right)\delta_{u,1}\\
        &\quad +T\bigg(\frac{\beta_t }{\min\{ I ,B\}}+\frac{\alpha_0 }{B} +\frac{\balp_0 }{B}+\frac{\btau_0^2\balp_0 }{\beta_0B}+\frac{\btau_0^2\balp_0 }{B}+\frac{m^2\btau_0^2  }{ BI^2}\bigg)\Bigg]\Bigg]\\
     &\stackrel{(a)}{\leq} \E\Bigg[\frac{2\Delta}{\eta_0}+ C_9'\bigg[\frac{1}{\beta_{0}}\delta_{z,1}+\frac{1}{m\tau_0}\delta_{y,1} +\frac{1}{ m \alpha_0}\delta_{s,1} +\frac{1}{m\balp_0}\delta_{u,1}+\frac{1}{m\btau_0}\delta_{v,1}+\frac{(\beta_{0}+\alpha_0+\balp_0)T}{\min\{ I ,B\}}\bigg]\Bigg]\\
    \end{aligned}
\end{equation}
where in $(a)$ we enlarge the constant $C_9'$ and use the setting $\balp_0=\beta_{0}$ and $\btau_0^2=\frac{C_{\btau}I^2}{m^2}\balp_0$.

From Theorem~\ref{thm:5}, we know that it is required that  $\beta_0,\alpha_0,\balp_0\leq \frac{1}{2}$, $\btau_0 = \min\left\{\frac{\lambda}{8L_{\phi v}^2},\frac{\lambda}{2},\frac{1}{\lambda},\frac{1}{2\sqrt{C_6'}},\frac{\sqrt{C_{\btau}}I\balp_0^{1/2}}{m}\right\}$, $\tau_0 = \sqrt{C_8'}\sqrt{\min\left\{\beta_0,\btau_0^2,\frac{I^2 \alpha_0}{m^2},\frac{I^2 \balp_0}{m^2},\right\}}$, $\eta_0 = \min\left\{\frac{1}{2L_f},\sqrt{C_{10}'}\sqrt{\min\left\{\beta_0, \tau_0^2, \btau_0^2, \frac{\alpha_0 I^2}{m^2 }, \frac{\balp_0 I^2}{m^2 } \right\}}\right\}$.

Without loss of generality, set $\beta_0=\alpha_0=\balp_0$ and let us assume that $\epsilon_0 = 2\Delta > \frac{7C_9' (\beta_0+\alpha_0+\balp_0)}{\mu}\left(\frac{\I(I<m)}{I}+\frac{1}{B}\right)$. The case that $2\Delta \leq \frac{7C_9' (\beta_0+\alpha_0+\balp_0)}{\mu}\left(\frac{\I(I<m)}{I}+\frac{1}{B}\right)$ can be simply covered by our proof. Then denotes $\epsilon_1 = \frac{7C_9' (\beta_0+\alpha_0+\balp_0)}{\mu}\left(\frac{\I(I<m)}{I}+\frac{1}{B}\right)$, and $\epsilon_k = \epsilon_1/2^{k-1}$.\\

In the first epoch $(k=1)$, we have initialization such that $F(\x_1)-F(\x^*)\leq \Delta$. In the following, we let the last subscript denote the epoch index. Setting $\eta_1 = \eta_0$, $\beta_1 = \beta_0$, $\alpha_1 = \alpha_0$, $\balp_1 = \balp_0$, $\tau_1=\tau_0$, $\btau_1=\btau_0$, and 
\begin{equation*}
    T_1 = \max\left\{ \frac{7\Delta}{\mu\eta_1},\max\left\{\frac{7}{\mu\beta_1}(\delta_{z,0}+\delta_{s,0}+\delta_{u,0}),\frac{7I^2C_8'}{\mu\tau_1 m}\delta_{y,0},\frac{7I^2C_{\btau}}{\mu\btau_1 m}\delta_{v,0}\right\}\left(\frac{\I(I<m)}{I}+\frac{1}{B}\right)^{-1}\right\}
\end{equation*}
We bound the error of the first stage's output as follows,
\begin{equation}
    \begin{aligned}
        &\E\left[\|\nabla F(\x_1)\|^2+\delta_{z,1}+\frac{1}{m}\delta_{y,1}+\frac{1}{m}\delta_{v,1}+\frac{1}{m}\delta_{s,1}+\frac{1}{m}\delta_{u,1}\right]\\
        &\leq \frac{2\Delta}{\eta_1 T_1}+\frac{C_9'}{T_1}\left(\frac{1}{\beta_1}\delta_{z,0}+\frac{1}{\tau_1 m}\delta_{y,0}+\frac{1}{\tau_1 m}\delta_{v,0} +\frac{1}{\alpha_1 m}\delta_{s,0}+\frac{1}{\balp_1 m}\delta_{u,0}\right)+\frac{C_9'(\beta_1+\alpha_1+\balp_1) }{\min\{ I ,B\}}\\
        &\leq \mu\epsilon_1
    \end{aligned}
\end{equation}  
where the first inequality uses (\ref{ineq:30}) and the fact that the output of each epoch is randomly sampled from all iterations, and the last line uses the choice of $\eta_1,\beta_1,\alpha_1,\balp_1, \tau_1,\btau_1,T_1,\epsilon_1$. If follows that
\begin{equation}
\E[F(\x_1)-\F(\x^*)]\leq \frac{1}{2\mu}\E[\|\nabla F(\x_1)\|^2]\leq \frac{\epsilon_1}{2}.
\end{equation}
Starting from the second stage, we will prove by induction. Suppose we are at $k$-th stage. Assuming that the output of $(k-1)$-the stage satisfies that $\E[F(\x_{k-1})-F(\x^*)]\leq \epsilon_{k-1}$ and $\E\left[\delta_{z,k-1}+\frac{\delta_{y,k-1}}{m}+\frac{\delta_{v,k-1}}{m} +\frac{\delta_{s,k-1}}{m}+\frac{\delta_{u,k-1}}{m}\right]\leq \mu \epsilon_{k-1}$, and setting $\beta_k=\alpha_k=\balp_k \leq \frac{\mu \epsilon_k }{21C_9'}\left(\frac{\I(I<m)}{I}+\frac{1}{B}\right)^{-1}$, $\btau_k^2 = \frac{C_{\btau} I^2}{m^2}\balp_k$, $\tau_k^2 = C_8'\min\left\{\beta_k,\btau_k^2,\frac{I^2 \alpha_k}{m^2},\frac{I^2 \balp_k}{m^2},\right\}$, $\eta_k=C_{10}'\min\left\{\beta_k, \tau_k^2, \btau_k^2, \frac{\alpha_k I^2}{m^2 }, \frac{\balp_k I^2}{m^2 } \right\}$, $T_k = \max\left\{\frac{28}{\mu\eta_k},\frac{7C_9'}{\beta_k},\frac{7C_9'}{\tau_k},\frac{7C_9'}{\btau_k} \right\}$, we have
\begin{equation}
    \begin{aligned}
        &\E\left[\|\nabla F(\x_k)\|^2+\delta_{z,k}+\frac{1}{m}\delta_{y,k}+\frac{1}{m}\delta_{v,k}+\frac{1}{m}\delta_{s,k}+\frac{1}{m}\delta_{u,k}\right]\\
        &\leq \E\Bigg[\frac{2(F(\x_{k-1})-F(\x^*))}{\eta_k T_k}+\frac{C_9'}{T_k}\left(\frac{1}{\beta_k}\delta_{z,k-1}+\frac{1}{\tau_k m}\delta_{y,k-1}+\frac{1}{\tau_k m}\delta_{v,k-1} +\frac{1}{\alpha_k m}\delta_{s,k-1}+\frac{1}{\balp_k m}\delta_{u,k-1}\right)\\
        &\quad +\frac{C_9'(\beta_k+\alpha_k+\balp_k) }{\min\{ I ,B\}}\Bigg]\\
        &\leq \E\Bigg[\frac{2\epsilon_{k-1}}{\eta_k T_k}+\frac{C_9'\mu \epsilon_{k-1}}{T_k}\left(\frac{1}{\beta_k}+\frac{1}{\tau_k }+\frac{1}{\btau_k } +\frac{1}{\alpha_k }+\frac{1}{\balp_k}\right) +\frac{C_9'(\beta_k+\alpha_k+\balp_k) }{\min\{ I ,B\}}\Bigg]\\
        &\leq \mu\epsilon_k
    \end{aligned}
\end{equation}  
It follows that
\begin{equation}
\begin{aligned}
\E[F(\x_k)-F(\x^*)]\leq \frac{1}{2\mu}\E[\|\nabla F(\x_k)\|^2] \leq \frac{\epsilon_k}{2}.
\end{aligned}
\end{equation}
Thus, after $K=1+\log_2(\epsilon_1/\epsilon)\leq \log_2(\epsilon_0/\epsilon)$ stages, $\E[F(\x_k)-F(\x^*)]\leq \epsilon$.

\end{proof}

\section{Proof of Lemmas}

\subsection{Proof of Lemma~\ref{lem:01}}
\begin{proof}
Due the smoothness of $F$, we can prove that under $\eta_t L_F\leq 1/2$
\begin{align*}
&F(\x_{t+1}) \leq F(\x_t) + \nabla F(\x_t)^{\top} (\x_{t+1} - \x_t) + \frac{L_F}{2}\|\x_{t+1} - \x_t\|^2\\
&= F(\x_t) - \eta_t\nabla F(\x_t)^{\top} \z_{t+1} + \frac{L_F\eta_t^2}{2}\|\z_{t+1}\|^2\\
&  = F(\x_t) +   \frac{\eta_t}{2}\|\nabla F(\x_t) - \z_{t+1}\|^2- \frac{\eta_t}{2}\|\nabla F(\x_t)\|^2 + (\frac{L_F\eta_t^2}{2} - \frac{\eta_t}{2})\|\z_{t+1}\|^2\\
&  \leq F(\x_t) +   \frac{\eta_t}{2}\|\nabla F(\x_t) - \z_{t+1}\|^2- \frac{\eta_t}{2}\|\nabla F(\x_t)\|^2  - \frac{\eta_t}{4}\|\z_{t+1}\|^2
\end{align*}
\end{proof}

\subsection{Proof of Lemma~\ref{lem:MSVR}}\label{app:MSVR}

\begin{proof}
Consider the updates
\begin{equation*}
    \h_{i,t+1} = \begin{cases}\Pi_{\Omega}\left[(1-\alpha)\h_{i,t} + \alpha h_i(\e_{i,t}; \B_i^t)+\gamma[ h_i(\e_{i,t}; \B_i^t-  h_i(\e_{i,t-1}; \B_i^t)]\right] & \text{  if }i \in I_t\\ \h_{i,t}  & \text{ o.w. }    \end{cases}
\end{equation*}
Define
\begin{equation*}
    \begin{aligned}
    & \tilde{h}_{i,t}=\Pi_{\Omega}\left[(1-\alpha)\h_{i,t} + \alpha h_i(\e_{i,t}; \B_i^t)+\gamma[ h_i(\e_{i,t}; \B_i^t-  h_i(\e_{i,t-1}; \B_i^t)]\right]\\
    & \bar{h}_{i,t}=(1-\alpha)\h_{i,t} + \alpha h_i(\e_{i,t}; \B_i^t)+\gamma[ h_i(\e_{i,t}; \B_i^t-  h_i(\e_{i,t-1}; \B_i^t)]
    \end{aligned}
\end{equation*}


We have
\begin{equation}\label{ineq:5}
    \begin{aligned}
    &\E_t\left[\|\h_{i,t+1}- h_i(\e_{i,t})\|^2\right]\\
    &=
    \E_{I_t}\E_{\B_i^t}\left[\|\h_{i,t+1}- h_i(\e_{i,t})\|^2\right]\\
    &=\frac{ I }{m}\E_{\B_i^t}\left[\| \tilde{h}_{i,t}- h_i(\e_{i,t})\|^2\right]+\left(1-\frac{ I }{m}\right)\|\h_{i,t}- h_i(\e_{i,t})\|^2\\
    &=\frac{ I }{m}\E_{\B_i^t}\left[\| \tilde{h}_{i,t}- h_i(\e_{i,t})\|^2\right] +\left(1-\frac{ I }{m}\right)\|\h_{i,t}- h_i(\e_{i,t-1})+ h_i(\e_{i,t-1})- h_i(\e_{i,t})\|^2\\
    &=\frac{ I }{m}\E_{\B_i^t}\left[\| \tilde{h}_{i,t}- h_i(\e_{i,t})\|^2\right] +\left(1-\frac{ I }{m}\right)\|\h_{i,t}- h_i(\e_{i,t-1})\|^2+\left(1-\frac{ I }{m}\right)\| h_i(\e_{i,t-1})- h_i(\e_{i,t})\|^2\\
    &\quad +\underbrace{2\left(1-\frac{ I }{m}\right)\Big\langle \h_{i,t}- h_i(\e_{i,t-1}),  h_i(\e_{i,t-1})- h_i(\e_{i,t})\Big\rangle}_{\text{\textcircled{a}}}
    \end{aligned}
\end{equation}

It follows from the non-expansive property of projection that
\begin{equation}\label{ineq:6}
    \begin{aligned}
    &\E_{\B_i^t}\left[\| \tilde{h}_{i,t}- h_i(\e_{i,t})\|^2\right]\leq \E_{\B_i^t}\left[\| \bar{h}_{i,t}- h_i(\e_{i,t})\|^2\right]\\
    &= \E_{\B_i^t}\bigg[\Big\|(1-\alpha)\h_{i,t} + \alpha h_i(\e_{i,t}; \B_i^t)+\gamma[ h_i(\e_{i,t}; \B_i^t-  h_i(\e_{i,t-1}; \B_i^t)] - h_i(\e_{i,t})\Big\|^2\bigg]\\
    &= \E_{\B_i^t}\bigg[\Big\|(1-\alpha)[\h_{i,t}- h_i(\e_{i,t-1})]+(1-\alpha)[ h_i(\e_{i,t-1})- h_i(\e_{i,t})]\\
    &\quad +\alpha[ h_i(\e_{i,t}; \B_i^t)- h_i(\e_{i,t})]+\gamma[ h_i(\e_{i,t}; \B_i^t-  h_i(\e_{i,t-1}; \B_i^t)]\Big\|^2\bigg]\\
    &\stackrel{(a)}{=} \E_{\B_i^t}\bigg[\Big\|(1-\alpha)[\h_{i,t}- h_i(\e_{i,t-1})]+(1-\alpha)[ h_i(\e_{i,t-1})- h_i(\e_{i,t})]\\
    &\quad +\gamma[ h_i(\e_{i,t}; \B_i^t)-  h_i(\e_{i,t-1}; \B_i^t)]\Big\|^2\bigg]+\alpha^2\E_{\B_i^t}\bigg[\Big\| h_i(\e_{i,t}; \B_i^t)- h_i(\e_{i,t})\Big\|^2\bigg]\\
    &\quad +2\gamma\alpha\E_{\B_i^t}\Big[\langle h_i(\e_{i,t}; \B_i^t)-  h_i(\e_{i,t-1}; \B_i^t), h_i(\e_{i,t}; \B_i^t)- h_i(\e_{i,t}) \rangle\Big]\\
    &\stackrel{(b)}{=} (1-\alpha)^2\Big\|\h_{i,t}- h_i(\e_{i,t-1})\Big\|^2\\
    &\quad +\E_{\B_i^t}\bigg[\Big\|(1-\alpha)[ h_i(\e_{i,t-1})- h_i(\e_{i,t})]+\gamma[ h_i(\e_{i,t}; \B_i^t)-  h_i(\e_{i,t-1}; \B_i^t)]\Big\|^2\bigg]\\
    &\quad + \underbrace{2(1-\alpha)(1-\alpha-\gamma)\Big\langle\h_{i,t}- h_i(\e_{i,t-1}),  h_i(\e_{i,t-1})- h_i(\e_{i,t})\Big\rangle}_{\text{\textcircled{b}}}\\
    &\quad+\frac{\alpha^2\sigma^2}{B} +2\gamma\alpha\E_{\B_i^t}\Big[\langle h_i(\e_{i,t}; \B_i^t)-  h_i(\e_{i,t-1}; \B_i^t), h_i(\e_{i,t}; \B_i^t)- h_i(\e_{i,t}) \rangle\Big]\\
    \end{aligned}
\end{equation}
where $(a)$ follows from $\E_{\B_i^t}[ h_i(\e_{i,t}; \B_i^t)- h_i(\e_{i,t})]=0$, $(b)$ follows from $\E_{\B_i^t}[ h_i(\e_{i,t}; \B_i^t)-  h_i(\e_{i,t-1}; \B_i^t)]= h_i(\e_{i,t})-  h_i(\e_{i,t-1})$.

Combining inequalities~(\ref{ineq:5}) and (\ref{ineq:6}) gives
\begin{equation}\label{ineq:7}
    \begin{aligned}
    &\E_t\left[\|\h_{i,t+1}- h_i(\e_{i,t})\|^2\right]\\
    &=\left(1-\frac{ I }{m}+\frac{(1-\alpha)^2 I }{m}\right)\Big\|\h_{i,t}- h_i(\e_{i,t-1})\Big\|^2\\
    &\quad +\frac{ I }{m}\E_{\B_i^t}\bigg[\Big\|(1-\alpha)[ h_i(\e_{i,t-1})- h_i(\e_{i,t})]+\gamma[ h_i(\e_{i,t}; \B_i^t)-  h_i(\e_{i,t-1}; \B_i^t)]\Big\|^2\bigg]\\
    &\quad + \frac{ I }{m}\text{\textcircled{b}}+\frac{\alpha^2 I \sigma^2}{Bm}+\left(1-\frac{ I }{m}\right)\Big\| h_i(\e_{i,t-1})- h_i(\e_{i,t})\Big\|^2+\text{\textcircled{a}}\\
    &\quad +\frac{2\gamma\alpha I }{m}\E_{\B_i^t}\Big[\langle h_i(\e_{i,t}; \B_i^t)-  h_i(\e_{i,t-1}; \B_i^t), h_i(\e_{i,t}; \B_i^t)- h_i(\e_{i,t}) \rangle\Big]\\
    &\stackrel{(a)}{=}\left(1-\frac{ I }{m}+\frac{(1-\alpha)^2 I }{m}\right)\Big\|\h_{i,t}- h_i(\e_{i,t-1})\Big\|^2\\
    &\quad +\frac{ I }{m}\E_{\B_i^t}\bigg[\Big\|(1-\alpha)[ h_i(\e_{i,t-1})- h_i(\e_{i,t})]+\gamma[ h_i(\e_{i,t}; \B_i^t)-  h_i(\e_{i,t-1}; \B_i^t)]\Big\|^2\bigg]\\
    &\quad +\frac{\alpha^2 I \sigma^2}{Bm}+\left(1-\frac{ I }{m}\right)\Big\| h_i(\e_{i,t-1})- h_i(\e_{i,t})\Big\|^2\\
    &\quad +\frac{2\gamma\alpha I }{m}\E_{\B_i^t}\Big[\langle h_i(\e_{i,t}; \B_i^t)-  h_i(\e_{i,t-1}; \B_i^t), h_i(\e_{i,t}; \B_i^t)- h_i(\e_{i,t}) \rangle\Big]\\
    &=\left(1-\frac{ I }{m}+\frac{(1-\alpha)^2 I }{m}\right)\Big\|\h_{i,t}- h_i(\e_{i,t-1})\Big\|^2+\frac{(1-\alpha)^2 I }{m}\Big\| h_i(\e_{i,t-1})- h_i(\e_{i,t})\Big\|^2\\
    &\quad +\frac{\gamma^2 I }{m}\E_{\B_i^t}\bigg[\Big\| h_i(\e_{i,t}; \B_i^t)-  h_i(\e_{i,t-1}; \B_i^t)\Big\|^2\bigg] - \frac{2(1-\alpha)\gamma I }{m}\Big\| h_i(\e_{i,t-1})- h_i(\e_{i,t})\Big\|^2\\
    &\quad +\frac{\alpha^2 I \sigma^2}{Bm}+\left(1-\frac{ I }{m}\right)\Big\| h_i(\e_{i,t-1})- h_i(\e_{i,t})\Big\|^2\\
    &\quad +\frac{2\gamma\alpha I }{m}\E_{\B_i^t}\Big[\langle h_i(\e_{i,t}; \B_i^t)-  h_i(\e_{i,t-1}; \B_i^t), h_i(\e_{i,t}; \B_i^t)- h_i(\e_{i,t}) \rangle\Big]\\
    &\stackrel{(b)}{=}\left(1-\frac{\alpha I }{m}\right)\Big\|\h_{i,t}- h_i(\e_{i,t-1})\Big\|^2 +\frac{4mL^2}{ I }\|\e_{i,t-1}-\e_{i,t}\|^2+\frac{\alpha^2 I \sigma^2}{Bm}\\
    &\quad +\frac{2\gamma\alpha I }{m}\E_{\B_i^t}\Big[\langle h_i(\e_{i,t}; \B_i^t)-  h_i(\e_{i,t-1}; \B_i^t), h_i(\e_{i,t}; \B_i^t)- h_i(\e_{i,t}) \rangle\Big]\\
    &\stackrel{(c)}{=}\left(1-\frac{\alpha I }{m}\right)\Big\|\h_{i,t}- h_i(\e_{i,t-1})\Big\|^2 +\frac{8mL^2}{ I }\|\e_{i,t-1}-\e_{i,t}\|^2+\frac{2\alpha^2 I \sigma^2}{Bm}
    \end{aligned}
\end{equation}

where $(a)$ is due to $\text{\textcircled{a}}+\frac{ I }{m}\text{\textcircled{b}}=0$, which follows from the setting $\gamma=\frac{m-\alpha I }{(1-\alpha) I }$, $(b)$ is due to $1-\frac{ I }{m}+\frac{(1-\alpha)^2 I }{m}\leq \frac{2(1-\alpha)\gamma I }{m}$ and $\gamma\leq \frac{2m}{ I }$, which follows from $\alpha\leq \frac{1}{2}$, $(c)$ is due to
\begin{equation*}
    \begin{aligned}
    &\frac{2\gamma\alpha I }{m}\E_{\B_i^t}\Big[\langle h_i(\e_{i,t}; \B_i^t)-  h_i(\e_{i,t-1}; \B_i^t), h_i(\e_{i,t}; \B_i^t)- h_i(\e_{i,t}) \rangle\Big]\\
    &\leq \frac{ I }{m}\E_{\B_i^t}\Big[\gamma^2\| h_i(\e_{i,t}; \B_i^t)-  h_i(\e_{i,t-1}; \B_i^t)\|^2+\alpha^2\| h_i(\e_{i,t}; \B_i^t)- h_i(\e_{i,t}) \|^2\Big]\\
    &\leq \frac{4mL^2}{ I }\|\e_{i,t-1}-\e_{i,t}\|^2+\frac{\alpha^2 I \sigma^2}{mB}\\
    \end{aligned}
\end{equation*}

Then by taking expectation over all randomness and summing over $i=1,\dots,m$, we obtain
\begin{equation}
    \begin{aligned}
    &\E\left[\sum_{i=1}^m\|\h_{i,t+1}-h_i(\e_{i,t})\|^2\right]\\
    &\leq \left(1-\frac{\alpha I }{m}\right)\E\Big[\sum_{i=1}^m\|\h_{i,t}- h_i(\e_{i,t-1})\|^2\Big] +\frac{8mL^2}{ I }\E\Big[\sum_{i=1}^m\|\e_{i,t-1}-\e_{i,t}\|^2\Big]+\frac{2\alpha^2 I \sigma^2}{B}
    \end{aligned}
\end{equation}

\end{proof}

\subsection{Proof of Lemma~\ref{lem:301}}

\begin{proof}
\begin{equation}
\begin{aligned}
&\|\Delta_t-\nabla F(\x_t)\|^2\\
&= \Bigg\|\frac{1}{m}\sum_{i=1}^m \nabla_x f_i(\x_t,\y_{i,t})-\nabla_{xy}^2 g_i(\x_t,\y_{i,t})\E_t[[H_{i,t}]^{-1}]\nabla_y f_i(\x_t,\y_{i,t}) \\
&\quad -\frac{1}{m}\sum_{i=1}^m \nabla_x f_i(\x_t,\y_i(\x_t))-\nabla_{xy}^2 g_i(\x_t,\y_i(\x_t))[\nabla_{yy}^2 g_i(\x_t,\y_i(\x_t))]^{-1} \nabla_y f_i(\x_t,\y_i(\x_t))\Bigg\|^2\\
&\leq \frac{1}{m}\sum_{i=1}^m 2\left\|\nabla_x f_i(\x_t,\y_{i,t}) - \nabla_x f_i(\x_t,\y_i(\x_t))\right\|^2\\
&\quad + 6\left\|\nabla_{xy}^2 g_i(\x_t,\y_{i,t})[\E_t[[H_{i,t}]^{-1}]-[\nabla_{yy}^2 g_i(\x_t,\y_i(\x_t))]^{-1}]\nabla_y f_i(\x_t,\y_{i,t})\right\|^2\\
&\quad + 6\left\|[\nabla_{xy}^2 g_i(\x_t,\y_{i,t})-\nabla_{xy}^2 g_i(\x_t,\y_i(\x_t))][\nabla_{yy}^2 g_i(\x_t,\y_i(\x_t))]^{-1}\nabla_y f_i(\x_t,\y_{i,t})\right\|^2\\
&\quad + 6\left\|\nabla_{xy}^2 g_i(\x_t,\y_i(\x_t))[\nabla_{yy}^2 g_i(\x_t,\y_i(\x_t))]^{-1}[\nabla_y f_i(\x_t,\y_{i,t})-\nabla_y f_i(\x_t,\y_i(\x_t))]\right\|^2\\
&\leq  \frac{1}{m}\sum_{i=1}^m  \left(2L_{fx}^2 + \frac{L_{gxy}^2C_{fy}^2}{\lambda^2}+\frac{6C_{gxy}^2L_{fy}^2}{\lambda^2}\right)\|\y_{i,t}-\y_i(\x_t)\|^2+6C_{gxy}^2C_{fy}^2\|\E_t[[H_{i,t}]^{-1}]-[\nabla_{yy}^2 g_i(\x_t,\y_i(\x_t))]^{-1}\|^2\\
&\stackrel{(a)}{\leq} \frac{1}{m}\sum_{i=1}^m  C_1 \|\y_{i,t}-\y_i(\x_t)\|^2+C_2\|H_{i,t}-\nabla_{yy}^2 g_i(\x_t,\y_{i,t}))\|^2
\end{aligned}
\end{equation}
where $C_1:=\left(2L_{fx}^2 + \frac{L_{gxy}^2C_{fy}^2}{\lambda^2}+\frac{6C_{gxy}^2L_{fy}^2}{\lambda^2}+\frac{12C_{gxy}^2C_{fy}^2L_{gyy}^2}{\lambda^4}\right)$, $C_2:=\frac{12C_{gxy}^2C_{fy}^2}{\lambda^4}$, and $(a)$ uses the fact that $[H_{i,t}]^{-1}$ is irrelevant to the randomness at iteration $t$, which means $[H_{i,t}]^{-1}=E_t[[H_{i,t}]^{-1}]$, and the Lipschitz continuity of $\nabla_{yy}^2 g_i(\x,\y_i)$.
\end{proof}

\subsection{Proof of Lemma~\ref{lem:302}}
\begin{proof}

\begin{equation}
    \begin{aligned}
        &\E_t[\left\|\z_{t+1}-\Delta_t\right\|^2]\\
        &=\E_t\left[\left\|(1-\beta_t)(\z_t-\Delta_{t-1})+(1-\beta_t)(\Delta_{t-1}-\widetilde{G}_t)+G_t-\Delta_t\right\|^2\right]\\
        &=(1-\beta_t)^2 \|\z_t-\Delta_{t-1}\|^2+2(1-\beta_t)^2\E_t\left[\|\Delta_{t-1}-\widetilde{G}_t+G_t-\Delta_t\|^2\right]+2\beta_t^2\E_t\left[\|G_t-\Delta_t\|^2\right]\\
        &\stackrel{(a)}{\leq} (1-\beta_t)^2 \|\z_t-\Delta_{t-1}\|^2+2(1-\beta_t)^2\E_t\left[\|G_t-\widetilde{G}_t\|^2\right]+2\beta_t^2\E_t\left[\|G_t-\Delta_t\|^2\right]
    \end{aligned}
\end{equation}
where $(a)$ use the standard inequality $\E[\|a-\E[a]\|^2]\leq \E[\|a\|^2]$, and $\E_t[G_t]=\Delta_t$, $\E_t[\widetilde{G}_t]=\Delta_{t-1}$. We further bound the last two terms as following
\begin{equation}
    \begin{aligned}
        &\E_t\left[\|G_t-\Delta_t\|^2\right]\\
        &\leq \E_t\Bigg[\Bigg\|\frac{1}{ I }\sum_{i\in \mI_t}\left(\nabla_x f_i(\x_t,\y_{i,t};\B_i^t)-\nabla_{xy}^2 g_i(\x_t,\y_{i,t};\wB_i^t)[H_{i,t}]^{-1}\nabla_y f_i(\x_t,\y_{i,t};\B_i^t)\right)\\
        &\quad  -\frac{1}{m}\sum_{i=1}^m\left(\nabla_x f_i(\x_t,\y_{i,t})-\nabla_{xy}^2 g_i(\x_t,\y_{i,t})[H_{i,t}]^{-1}\nabla_y f_i(\x_t,\y_{i,t})\right)\Bigg\|^2\Bigg]\\
        &\leq \E_t\Bigg[2\Bigg\|\frac{1}{ I }\sum_{i\in \mI_t}\left(\nabla_x f_i(\x_t,\y_{i,t};\B_i^t)-\nabla_{xy}^2 g_i(\x_t,\y_{i,t};\wB_i^t)[H_{i,t}]^{-1}\nabla_y f_i(\x_t,\y_{i,t};\B_i^t)\right)\\
        &\quad -\frac{1}{m}\sum_{i=1}^m\left(\nabla_x f_i(\x_t,\y_{i,t};\B_i^t)-\nabla_{xy}^2 g_i(\x_t,\y_{i,t};\wB_i^t)[H_{i,t}]^{-1}\nabla_y f_i(\x_t,\y_{i,t};\B_i^t)\right)\Bigg\|^2\\
        &\quad +2\Bigg\|\frac{1}{m}\sum_{i=1}^m\left(\nabla_x f_i(\x_t,\y_{i,t};\B_i^t)-\nabla_{xy}^2 g_i(\x_t,\y_{i,t};\wB_i^t)[H_{i,t}]^{-1}\nabla_y f_i(\x_t,\y_{i,t};\B_i^t)\right)\\
        &\quad -\frac{1}{m}\sum_{i=1}^m\left(\nabla_x f_i(\x_t,\y_{i,t})-\nabla_{xy}^2 g_i(\x_t,\y_{i,t})[H_{i,t}]^{-1}\nabla_y f_i(\x_t,\y_{i,t})\right)\Bigg\|^2\Bigg]\\
        &\leq \frac{8(2C_{fx}^2+\frac{2C_{gxy}^2C_{fy}^2}{\lambda^2})}{I}+\frac{4\sigma^2}{B}+\frac{8(\frac{C_{gxy}^2+C_{fy}^2}{\lambda^2})\sigma^2}{B}=:C_5\left(\frac{\I(I<m)}{I}+\frac{1}{B}\right),
    \end{aligned}
\end{equation}
where $C_5=\max\{8(2C_{fx}^2+\frac{2C_{gxy}^2C_{fy}^2}{\lambda^2}), 4\sigma^2+8(\frac{C_{gxy}^2+C_{fy}^2}{\lambda^2})\sigma^2\}$, and

\begin{equation}
    \begin{aligned}
        &\E_t\left[\|G_t-\widetilde{G}_t\|^2\right]\\
        &=\E_t\Bigg[\Bigg\|\frac{1}{ I }\sum_{i\in \mI_t}\left(\nabla_x f_i(\x_t,\y_{i,t};\B_i^t)-\nabla_{xy}^2 g_i(\x_t,\y_{i,t};\wB_i^t)[H_{i,t}]^{-1}\nabla_y f_i(\x_t,\y_{i,t};\B_i^t)\right)\\
        &\quad -\left(\nabla_x f_i(\x_{t-1},\y_{i,t-1};\B_i^t)-\nabla_{xy}^2 g_i(\x_{t-1},\y_{i,t-1};\wB_i^t)[H_{i,t-1}]^{-1}\nabla_y f_i(\x_{t-1},\y_{i,t-1};\B_i^t)\right)\Bigg\|^2\Bigg]\\
        &\leq \frac{1}{m}\sum_{i=1}^m (2L_{fx}^2+\frac{6L_{gxy}^2C_{fy}^2}{\lambda^2}+\frac{6C_{gxy}^2L_{fy}^2}{\lambda^2})(\|\x_t-\x_{t-1}\|^2+\|\y_{i,t}-\y_{i,t-1}\|^2)+\frac{6C_{gxy}^2C_{fy}^2}{\lambda^4}\|H_{i,t}-H_{i,t-1}\|^2\\
        &=: C_3\|\x_t-\x_{t-1}\|^2+\frac{C_3}{m}\|\y_t-\y_{t-1}\|^2+\frac{C_4}{m}\|H_t-H_{t-1}\|^2.
    \end{aligned}
\end{equation}
Then we have
\begin{equation}
    \begin{aligned}
        &\E_t[\left\|\z_{t+1}-\Delta_t\right\|^2]\\
        &\leq (1-\beta_t)^2 \|\z_t-\Delta_{t-1}\|^2+2(1-\beta_t)^2\left(C_3\|\x_t-\x_{t-1}\|^2+\frac{C_3}{m}\|\y_t-\y_{t-1}\|^2+\frac{C_4}{m}\|\v_t-\v_{t-1}\|^2\right)\\
        &\quad +2\beta_t^2C_5\left(\frac{\I(I<m)}{I}+\frac{1}{B}\right)\\
        &\leq (1-\beta_t) \|\z_t-\Delta_{t-1}\|^2+2C_3\|\x_t-\x_{t-1}\|^2+\frac{2C_3}{m}\|\y_t-\y_{t-1}\|^2+\frac{2C_4}{m}\|H_t-H_{t-1}\|^2 +2\beta_t^2C_5\left(\frac{\I(I<m)}{I}+\frac{1}{B}\right)
    \end{aligned}
\end{equation}
\end{proof}

\subsection{Proof of Lemma~\ref{lem:303}}
\begin{proof}
By Lemma6 in \cite{https://doi.org/10.48550/arxiv.2207.08540}, we have, for $\balp_{t+1}\leq 1/2$, 
\begin{equation}
\begin{aligned}
\|H_{t+1}-H_t\|^2 &\leq \frac{2I\balp_{t+1}^2\sigma^2}{B}+\frac{4I \balp_{t+1}^2}{m}\E\left[\sum_{i=1}^m\|H_{i,t}-\nabla_{yy}^2 g_i(\x_t,\y_{i,t})\|^2\right]\\
&\quad +\frac{9m^2 L_{gyy}^2}{I}\E[m\|\x_{t+1}-\x_t\|^2+\|\y_{t+1}-\y_t\|^2]\\
\end{aligned}
\end{equation}
\end{proof}

\subsection{Proof of Lemma~\ref{lem:203}}
\begin{equation}
    \begin{aligned}
        &\left\|\Delta_t-\nabla F(\x_t)\right\|^2\\
        &=\Bigg\|\frac{1}{m}\sum_{i=1}^m \left(\nabla_x f_i(\x_{t},\y_{i,t})-\nabla_{xy}^2 g_i(\x_t,\y_{i,t})\v_{i,t}\right)-\frac{1}{m}\sum_{i=1}^m \left(\nabla_x f_i(\x_{t},\y_i(\x_t))-\nabla_{xy}^2 g_i(\x_t,\y_i(\x_t))\v_i(\x_t)\right) \Bigg\|^2\\
        &\leq \frac{1}{m}\sum_{i=1}^m (2L_{fx}^2+\frac{4L_{gxy}^2C_{fy}^2}{\lambda^2})\|\y_{i,t}-\y_i(\x_t)\|^2+4C_{gxy}^2\|\v_{i,t}-\v_i(\x_t)\|^2\\
        &\leq \frac{1}{m}\sum_{i=1}^m (2L_{fx}^2+\frac{4L_{gxy}^2C_{fy}^2}{\lambda^2})\|\y_{i,t}-\y_i(\x_t)\|^2+8C_{gxy}^2\|\v_{i,t}-\v_i(\x_t,\y_{i,t})\|^2+8C_{gxy}^2L_v^2\|\y_{i,t}-\y_i(\x_t)\|^2\\
        &=:\frac{C_1'}{m}\|\y_t-\y(\x_t)\|^2+\frac{C_2'}{m}\|\v_t-\v(\x_t,\y_t)\|^2
    \end{aligned}
\end{equation}

\subsection{Proof of Lemma~\ref{lem:202}}
\begin{proof}
\begin{equation}
    \begin{aligned}
        &\E_t[\left\|\z_{t+1}-\Delta_t\right\|^2]\\
        &=\E_t\left[\left\|(1-\beta_t)(\z_t-\Delta_{t-1})+(1-\beta_t)(\Delta_{t-1}-\widetilde{G}_t)+G_t-\Delta_t\right\|^2\right]\\
        &=(1-\beta_t)^2 \|\z_t-\Delta_{t-1}\|^2+2(1-\beta_t)^2\E_t\left[\|\Delta_{t-1}-\widetilde{G}_t+G_t-\Delta_t\|^2\right]+2\beta_t^2\E_t\left[\|G_t-\Delta_t\|^2\right]\\
        &\stackrel{(a)}{\leq} (1-\beta_t)^2 \|\z_t-\Delta_{t-1}\|^2+2(1-\beta_t)^2\E_t\left[\|G_t-\widetilde{G}_t\|^2\right]+2\beta_t^2\E_t\left[\|G_t-\Delta_t\|^2\right]
    \end{aligned}
\end{equation}
where $(a)$ use the standard inequality $\E[\|a-\E[a]\|^2]\leq \E[\|a\|^2]$, and $\E_t[G_t]=\Delta_t$, $\E_t[\widetilde{G}_t]=\Delta_{t-1}$. We further bound the last two terms as following
\begin{equation}
    \begin{aligned}
        &\E_t\left[\|G_t-\Delta_t\|^2\right]\\
        &\leq \E_t\Bigg[\Bigg\|\frac{1}{ I }\sum_{i\in \mI_t}\left(\nabla_x f_i(\x_t,\y_{i,t};\B_i^t)-\nabla_{xy}^2 g_i(\x_t,\y_{i,t};\wB_i^t)\v_{i,t}\right) -\frac{1}{m}\sum_{i=1}^m\left(\nabla_x f_i(\x_t,\y_{i,t})-\nabla_{xy}^2 g_i(\x_t,\y_{i,t})\v_{i,t}\right)\Bigg\|^2\Bigg]\\
        &\leq \E_t\Bigg[2\Bigg\|\frac{1}{ I }\sum_{i\in \mI_t}\left(\nabla_x f_i(\x_t,\y_{i,t};\B_i^t)-\nabla_{xy}^2 g_i(\x_t,\y_{i,t};\wB_i^t)\v_{i,t}\right)\\
        &\quad -\frac{1}{m}\sum_{i=1}^m\left(\nabla_x f_i(\x_t,\y_{i,t};\B_i^t)-\nabla_{xy}^2 g_i(\x_t,\y_{i,t};\wB_i^t)\v_{i,t}\right)\Bigg\|^2\\
        &\quad +2\Bigg\|\frac{1}{m}\sum_{i=1}^m\left(\nabla_x f_i(\x_t,\y_{i,t};\B_i^t)-\nabla_{xy}^2 g_i(\x_t,\y_{i,t};\wB_i^t)\v_{i,t}\right)-\frac{1}{m}\sum_{i=1}^m\left(\nabla_x f_i(\x_t,\y_{i,t})-\nabla_{xy}^2 g_i(\x_t,\y_{i,t})\v_{i,t}\right)\Bigg\|^2\Bigg]\\
        &\leq \frac{8(2C_{fx}^2+2C_{gxy}^2\mathcal{V}^2)}{I}+\frac{4\sigma^2}{B}+\frac{4\sigma^2\mathcal{V}^2}{B}\leq C_5'\left(\frac{\I(I<m)}{I}+\frac{1}{B}\right),
    \end{aligned}
\end{equation}
where $C_5'=\max\{8(2C_{fx}^2+2C_{gxy}^2\mathcal{V}^2, 4\sigma^2+4\sigma^2\mathcal{V}^2) \}$, and

\begin{equation}
    \begin{aligned}
        &\E_t\left[\|G_t-\widetilde{G}_t\|^2\right]\\
        &=\E_t\Bigg[\Bigg\|\frac{1}{ I }\sum_{i\in \mI_t}\left(\nabla_x f_i(\x_t,\y_{i,t};\B_i^t)-\nabla_{xy}^2 g_i(\x_t,\y_{i,t};\wB_i^t)\v_{i,t}\right)\\
        &\quad -\left(\nabla_x f_i(\x_{t-1},\y_{i,t-1};\B_i^t)-\nabla_{xy}^2 g_i(\x_{t-1},\y_{i,t-1};\wB_i^t)\v_{i,t-1}\right)\Bigg\|^2\Bigg]\\
        &\leq \frac{1}{m}\sum_{i=1}^m (2L_{fx}^2+\frac{4L_{gxy}^2C_{fy}^2}{\lambda^2})(\|\x_t-\x_{t-1}\|^2+\|\y_{i,t}-\y_{i,t-1}\|^2)+4C_{gxy}^2\|\v_{i,t}-\v_{i,t-1}\|^2\\
        &=: C_3'\|\x_t-\x_{t-1}\|^2+\frac{C_3'}{m}\|\y_t-\y_{t-1}\|^2+\frac{C_4'}{m}\|\v_t-\v_{t-1}\|^2.
    \end{aligned}
\end{equation}
Then we have
\begin{equation}
    \begin{aligned}
        &\E_t[\left\|\z_{t+1}-\Delta_t\right\|^2]\\
        &\leq (1-\beta_t)^2 \|\z_t-\Delta_{t-1}\|^2+2(1-\beta_t)^2\left(C_3'\|\x_t-\x_{t-1}\|^2+\frac{C_3'}{m}\|\y_t-\y_{t-1}\|^2+\frac{C_4'}{m}\|\v_t-\v_{t-1}\|^2\right)\\
        &\quad +2\beta_t^2C_5'\left(\frac{\I(I<m)}{I}+\frac{1}{B}\right)\\
        &\leq (1-\beta_t) \|\z_t-\Delta_{t-1}\|^2+2C_3'\|\x_t-\x_{t-1}\|^2+\frac{2C_3'}{m}\|\y_t-\y_{t-1}\|^2+\frac{2C_4'}{m}\|\v_t-\v_{t-1}\|^2 +2\beta_t^2C_5'\left(\frac{\I(I<m)}{I}+\frac{1}{B}\right)
    \end{aligned}
\end{equation}

\end{proof}

\subsection{Proof of Lemma~\ref{lem:204}}
\begin{proof}
Consider updates $\v_{i,t+1}=\Pi_{\V}[\v_{i,t}-\btau_t \u_{i,t}]$. Note that $\v_i(\x_t,\y_{i,t}) = \Pi_{\V}[\v_i(\x_t,\y_{i,t})-\btau_t \nabla_v \phi_i(\v_i(\x_t,\y_{i,t}),\x_t,\y_{i,t})]$
\begin{equation}
\begin{aligned}
&\E[\|\v_{i,t+1}-\v_i(\x_t,\y_{i,t})\|^2]\\
& =\E\left[\|\Pi_{\V}[\v_{i,t}-\btau_t \u_{i,t}]-\Pi_{\V}[\v_i(\x_t,\y_{i,t})-\btau_t \nabla_v \phi_i(\v_i(\x_t,\y_{i,t}),\x_t,\y_{i,t})]\|^2\right]\\
& \leq \E\left[\|\v_{i,t}-\btau_t \u_{i,t}-\v_i(\x_t,\y_{i,t})+\btau_t \nabla_v \phi_i(\v_i(\x_t,\y_{i,t}),\x_t,\y_{i,t})\|^2\right]\\
& \leq \E\left[\|\v_{i,t}-\v_i(\x_t,\y_{i,t})-\btau_t \u_{i,t}+\btau_t \nabla_v \phi_i(\v_{i,t},\x_t,\y_{i,t})-\btau_t \nabla_v \phi_i(\v_{i,t},\x_t,\y_{i,t})+\btau_t \nabla_v \phi_i(\v_i(\x_t,\y_{i,t}),\x_t,\y_{i,t})\|^2\right]\\
& \leq \E\big[\|\v_{i,t}-\v_i(\x_t,\y_{i,t})\|^2+\|-\btau_t \u_{i,t}+\btau_t \nabla_v \phi_i(\v_{i,t},\x_t,\y_{i,t})-\btau_t \nabla_v \phi_i(\v_{i,t},\x_t,\y_{i,t})+\btau_t \nabla_v \phi_i(\v_i(\x_t,\y_{i,t}),\x_t,\y_{i,t})\|^2\\
&\quad +\langle \v_{i,t}-\v_i(\x_t,\y_{i,t}),-\btau_t \nabla_v \phi_i(\v_{i,t},\x_t,\y_{i,t})+\btau_t \nabla_v \phi_i(\v_i(\x_t,\y_{i,t}),\x_t,\y_{i,t}) \rangle\\
&\quad +\langle \v_{i,t}-\v_i(\x_t,\y_{i,t}),-\btau_t \u_{i,t}+\btau_t \nabla_v \phi_i(\v_{i,t},\x_t,\y_{i,t})\rangle\big]\\
& \stackrel{(a)}{\leq} \E\big[\|\v_{i,t}-\v_i(\x_t,\y_{i,t})\|^2+2\btau_t^2 L_{\phi v}^2\| \v_{i,t}-\v_i(\x_t,\y_{i,t})\|^2-\lambda \btau_t\|\v_{i,t}-\v_i(\x_t,\y_{i,t})\|^2\\
&\quad +2\btau_t^2\| \u_{i,t}- \nabla_v \phi_i(\v_{i,t},\x_t,\y_{i,t})\|^2+\langle \v_{i,t}-\v_i(\x_t,\y_{i,t}),-\btau_t \u_{i,t}+\btau_t \nabla_v \phi_i(\v_{i,t},\x_t,\y_{i,t})\rangle\big]\\
& \stackrel{(b)}{\leq} \E\big[\|\v_{i,t}-\v_i(\x_t,\y_{i,t})\|^2+2\btau_t^2 L_{\phi v}^2\| \v_{i,t}-\v_i(\x_t,\y_{i,t})\|^2-\frac{3\lambda\btau_t}{4}\|\v_{i,t}-\v_i(\x_t,\y_{i,t})\|^2\\
&\quad +2\btau_t^2\| \u_{i,t}- \nabla_v \phi_i(\v_{i,t},\x_t,\y_{i,t})\|^2+4\lambda\btau_t\|\u_{i,t}- \nabla_v \phi_i(\v_{i,t},\x_t,\y_{i,t})\|^2 \big]\\
& \stackrel{(c)}{\leq} \E\big[(1-\frac{\lambda\btau_t}{2})\|\v_{i,t}-\v_i(\x_t,\y_{i,t})\|^2+5\lambda\btau_t\|\u_{i,t}- \nabla_v \phi_i(\v_{i,t},\x_t,\y_{i,t})\|^2 \big]\\
\end{aligned}
\end{equation}
where $(a)$ uses the $\lambda$-strong convexity of $\phi_i$, $(b)$ uses
\begin{equation}
\begin{aligned}
&\langle \v_{i,t}-\v_i(\x_t,\y_{i,t}),\btau_t \u_{i,t}-\btau_t \nabla_v \phi_i(\v_{i,t},\x_t,\y_{i,t})\rangle\\
&=\langle \frac{\sqrt{\lambda\btau_t}}{2}(\v_{i,t}-\v_i(\x_t,\y_{i,t})),2\sqrt{\lambda\btau_t}( \u_{i,t}- \nabla_v \phi_i(\v_{i,t},\x_t,\y_{i,t}))\rangle\\
&\leq \frac{\lambda\btau_t}{4}\|\v_{i,t}-\v_i(\x_t,\y_{i,t})\|^2+4\lambda\btau_t\|\u_{i,t}- \nabla_v \phi_i(\v_{i,t},\x_t,\y_{i,t})\|^2
\end{aligned}
\end{equation}
and $(c)$ uses the assumption $\btau_t\leq \min\left\{\frac{\lambda}{8L_{\phi v}^2},\frac{\lambda}{2}\right\}$, 

Then
\begin{equation}
\begin{aligned}
&\E[\|\v_{i,t+1}-\v_i(\x_{t+1},\y_{i,t+1})\|^2]\\
&\leq (1+\frac{\lambda\btau_t}{4})\E[\|\v_{i,t+1}-\v_i(\x_t,\y_{i,t})\|^2]+(1+\frac{4}{\lambda\btau_t})\E[\|\v_i(\x_t,\y_{i,t})-\v_i(\x_{t+1},\y_{i,t+1})\|^2]\\
&\leq (1-\frac{\lambda\btau_t}{4})\E\|\v_{i,t}-\v_i(\x_t,\y_{i,t})\|^2]+10\lambda\btau_t\E[\|\u_{i,t}- \nabla_v \phi_i(\v_{i,t},\x_t,\y_{i,t})\|^2]\\
&\quad  +\frac{5}{\lambda\btau_t}E[\|\v_i(\x_t,\y_{i,t})-\v_i(\x_{t+1},\y_{i,t+1})\|^2]\\
\end{aligned}
\end{equation}
where we use the assumption $\btau_t\leq \frac{1}{\lambda}$.
Take summation over all blocks $i=1,\dots,m$, we have
\begin{equation}
\begin{aligned}
\E[\delta_{v,t+1}]&\leq (1-\frac{\lambda\tau_t}{4})\E[\delta_{v,t}]+10\lambda\btau_t\E[\tilde{\delta}_{u,t}]  +\frac{5L_v^2}{\lambda\tau_t}E[\|\x_t-\x_{t+1}\|^2]+\frac{5L_v^2}{\lambda\tau_t}\E[\|\y_{t}-\y_{t+1}\|^2]\\
\end{aligned}
\end{equation}
\end{proof}

\section{Numeric Results of Hyper-parameter Optimization Experiment}\label{app:HO}
\begin{table}[h]
\caption{Testing accuracies and standard deviation over 3 runs with different random seeds from logistic regression, $\vone$ with $m=1$ lower-level problem, and $\vone$ with $m=100$ lower-level problems on various noise level of dataset \textit{a8a}. Noise level represents the proportion of training sample labels that are flipped. $70\%$ of the positive samples are removed from training data except for noise level $0*$, which means no label noise and no data imbalance.}
\begin{center}
\begin{tabular}{c|ccc}
\hline 
Noise Level & Logistic Regression & $\text{BSVRB}^{\text{v1}} (m=1)$ & $\text{BSVRB}^{\text{v1}} (m=100)$\\
\hline
0*&0.8528 $\pm$ 0.0005& 0.8526$\pm$ 0.0002 & \textbf{0.8509}$\pm$ \textbf{0.0011} \\
0& 0.8442 $\pm$ 0.0009 &0.8426 $\pm$ 0.0016 &  \textbf{0.8477} $\pm$ \textbf{0.0013}\\
0.1& 0.8285$\pm$ 0.0034& 0.8303 $\pm$ 0.0100 &  \textbf{0.8400}$\pm$ \textbf{0.0025} \\
0.2&  0.8250$\pm$0.0066& 0.8185$\pm$ 0.0090 & \textbf{0.8388} $\pm$ \textbf{0.0024}\\
0.3&   0.7929$\pm$ 0.0081& 0.8118 $\pm$ 0.0047 & \textbf{0.8239} $\pm$ \textbf{0.0015}\\
0.4&   0.7715$\pm$ 0.0025& 0.7749 $\pm$ 0.0079 &  \textbf{0.8051}$\pm$ \textbf{0.0013}\\
\hline
\end{tabular}
\end{center}
\end{table}
\end{document}